\documentclass[11pt]{article} 

\usepackage{amsthm,amsmath,amsfonts,amssymb,enumerate}
\usepackage{graphicx,tikz}
\usepackage{wrapfig}
\usepackage{blindtext}
\usepackage{geometry}
\usepackage{float}
\usepackage{multicol}

\usetikzlibrary{positioning,arrows}
\usepackage[utf8]{inputenc} 
\usepackage[T1]{fontenc}
\usepackage{hyperref}
\usepackage{theoremref}
\usepackage{tikz-cd}

\numberwithin{equation}{section}
\theoremstyle{plain}
\newtheorem{lem}[equation]{Lemma}

\newtheorem{prop}[equation]{Proposition}
\newtheorem{thm}[equation]{Theorem}
\newtheorem{cor}[equation]{Corollary}
\newtheorem{que}[equation]{Question}

\newtheorem{conj}[equation]{Conjecture}
\theoremstyle{definition}
\newtheorem{definition}[equation]{Definition}

\newtheorem{remark}[equation]{Remark}

\newtheorem{claim}[equation]{Claim}

\newtheorem*{claim*}{Claim}

\newcommand{\type}{\operatorname{Type}}
\newcommand{\cq}{\mathcal{Q}}
\newcommand{\od}{\widehat{\Sigma}}
\newcommand{\Si}{\Sigma}
\newcommand{\p}{\Pi}
\newcommand{\bD}{\mathbb D}

\newcommand{\bU}{\mathbb U} 
\newcommand{\bC}{\mathfrak C}
\newcommand{\bSD}{\mathbb S}

\newcommand{\wX}{\widehat X}
\newcommand{\wY}{\widehat Y}
\newcommand{\wZ}{\widehat Z}
\newcommand{\ca}{\mathcal {A}} 
\newcommand{\prj}{\operatorname{Proj}} 
 
\newcommand{\fan}{\operatorname{Fan}}
\newcommand{\vertex}{\operatorname{Vert}} 
\newcommand{\lk}{\operatorname{lk}} 
\newcommand{\ch}{\mathcal {H}} 
 
\newcommand{\cp}{\mathcal {P}} 
\newcommand{\ce}{\mathcal {E}}

\newcommand{\wE}{\widehat E}
\newcommand{\wtX}{\widetilde X}
\newcommand{\wtP}{\widetilde P}
\newcommand{\wtR}{\widetilde R}
\newcommand{\wtQ}{\widetilde Q}

\newcommand{\whC}{\widehat C}
\newcommand{\whZ}{\widehat Z}
\newcommand{\whX}{\widehat X}
\newcommand{\sX}{\mathsf X}
\newcommand{\sY}{\mathsf Y}
\newcommand{\wsX}{\widehat {\mathsf X}}
\newcommand{\wsY}{\widehat {\mathsf Y}}
\newcommand{\act}{\curvearrowright}
\newcommand{\si}{\sigma}
\newcommand{\cb}{\mathcal B}
\newcommand{\supp}{\operatorname{Supp}}

\begin{document}

\title{Cycles in spherical Deligne complexes and $K(\pi,1)$-conjecture for Artin groups}
\author{Jingyin Huang}
\maketitle
\begin{abstract}
We introduce a method of finding large non-positively curved subcomplexes in certain spherical Deligne complexes, which is effective for studying fillings of certain 6-cycles in spherical Deligne complexes. As applications, we show the $K(\pi,1)$-conjecture holds for all 3-dimensional hyperbolic type Artin groups, except one single example; and the conjecture holds for all quasi-Lann\'er hyperbolic type Artin groups up to dimension 4. In higher dimension, we show  the $K(\pi,1)$-conjecture for Artin groups whose Coxeter diagrams are complete bipartite (edge labels can be arbitrary), answering a question of J. McCammond.
\end{abstract}

\section{Introduction}

\subsection{Background and motivation}
A \emph{Coxeter group} $W_S$ is a group with generating set $S=\{s_1,\ldots,s_n\}$ and presentation
$$
\langle s_1,s_2,\ldots, s_n\mid (s_is_j)^{m_{ij}}=1\rangle
$$
with $m_{ii}=2$ for each $i$, $m_{ij}=m_{ji}\ge 2$ is either an integer or $\infty$ for $i\neq j$. Here $m_{ij}=\infty$ means no relation between $s_i$ and $s_j$.
There is an associated \emph{Coxeter diagram} $\Lambda$ (or \emph{Coxeter-Dynkin graph}) encoding the non-commutativity of generators of $W_S$: the collection of nodes of $\Lambda$ is $S$, and two nodes $s_i$ and $s_j$ are joined an edge labeled by $m_{ij}$ if $m_{ij}\ge 3$. 

For each Coxeter group $W_S$, there is an associated \emph{Artin group} $A_S$ (or \emph{Artin-Tits} group), whose generating set is $S$, and for each pair of generators $s_i\neq s_j$ with $m_{ij}<\infty$, there is a relator of form  $s_is_js_i\cdots=s_js_is_j\cdots$ with both sides being alternating words of length $m_{ij}$. We will also write $A_\Lambda$ and $W_\Lambda$ for the Artin group and Coxeter group with Coxeter diagram $\Lambda$.

A \emph{reflection} of $W_S$ is a conjugate of an element in $S$. Let $R$ be the collection of all reflections in $W_S$.
Recall that $W_S$ admits a canonical representation $\rho: W\to GL(n,\mathbb R)$, such that each element in $R$ acts as a linear reflection on $\mathbb R^n$, and the action $W_S\act \mathbb R^n$ stabilizes an open convex cone $I\subset \mathbb R^n$, called the \emph{Tits cone}, where the action of $W_S$ is properly discontinuous.
For each reflection $r\in W_S$, let $H_r$ be the set of fixed points of $\rho(r)$ in $I$. The collection of all such $H_r$ forms an arrangement of hyperplanes in $I$.
Let 
$$
M(W_S)= (I\times I)\setminus(\cup_{r\in R} (H_r\times H_r)). 
$$
The \emph{$K(\pi,1)$-conjecture for reflection arrangement complements}, due to Arnol'd, Brieskorn, Pham, and Thom, predicts that the space $M(W_S)$ is aspherical for any Coxeter group $W_S$. There is an induced action of $W_S$ on $M(W_S)$, which is free and properly discontinuous. The quotient has its fundamental group isomorphic to the Artin group $A_S$. 

Historically, the interests in $M(W_S)$ come from a connection with singularity theory: if $W_S$ is of type $A$, $D$, or $E$, then $M(W_S)/W_S$ is the complement of the discriminant of the semi-universal deformation of a simple singularity of the same type \cite{brieskorn1970singular}. More recently, it is known that many $M(W_S)$ are closely related to spaces of stability conditions on suitable triangulated categories \cite{bridgeland2009stability,ikeda2014stability,qiu2018contractible,august2022stability,heng2022categorification,dell2023fusion,qiu2023fusion,heng2024stability}, see also the survey article \cite{heng2024introduction}. In a few cases, $M(W_S)$ coincides with components of certain strata of the moduli space of abelian differential \cite{looijenga2012fine}, and moduli spaces of certain quadratic differential \cite[Appendix C]{qiu2024moduli}.

The $K(\pi,1)$-conjecture is a major unsettled problem in studying Artin groups. A well-known geometric approach to this conjecture reduces the $K(\pi,1)$-conjecture for \emph{any} Artin groups to certain properties of a short list of Artin groups, called \emph{spherical Artin groups}. These are Artin groups associated with finite Coxeter group, and they are completely classified, see the first two columns of Figure~\ref{fig:classical} for their Coxeter diagrams. The following combines \cite{CharneyDavis} and \cite{bowditch1995notes}.

\begin{thm}
	\label{thm:CharneyDavis}
	Suppose for any irreducible spherical Artin groups $A_S$, any loop of length $<2\pi$ in the associated spherical Deligne complex $\Delta_S$ with Moussong metric has a length non-increasing homotopy to the trivial loop. Then the $K(\pi,1)$-conjecture holds for \emph{any} Artin group. 
\end{thm}

Understanding loops of length $<2\pi$ in spherical Deligne complexes also leads to the solution of a highly non-trivial list of other important unsettled questions on more general Artin groups, like word problem, the intersection of parabolic subgroups, computation of symmetrical subgroups as well as normalizers and commensurators of elements, etc \cite{altobelli1998word,crisp2000symmetrical,godelle2007artin,morris2021parabolic}.

The main difficulty of this approach is to deal with the large space of loops (of length $<2\pi$) in spherical Deligne complexes, especially in high dimensions. To circumvent this difficulty, we proposed another approach in \cite{huang2023labeled} to $K(\pi,1)$-conjecture, aiming at reducing the $K(\pi,1)$-conjecture for any Artin groups to studying only short loops in the \emph{1-skeleton} of the spherical Deligne complexes $\Delta_S$. More precisely, we only need to deal with \emph{edge loops}, which are made of a sequence of edges. The meaning of ``short'' depends on the type of the associated irreducible spherical Artin groups (see Section~\ref{sec:cycle}), but they all have length $<2\pi$, and in many cases, we only need to deal with loops with $\le 6$ edges, which is much smaller compared to the collection of edge loops with length $<2\pi$ in high dimension. This approach in \cite{huang2023labeled} can be combined with the works mentioned above and \cite[Section 8]{haettel2021lattices} to answer importance questions on Artin groups other than the $K(\pi,1)$-conjecture, see \cite[Corollary 10.13]{huang2023labeled} for an example. 

%In summary, the study of short edge loops in spherical Deligne complexes of spherical Artin groups lies in the intersection of two geometric approaches to study all Artin groups. Once this goal is achieved, one can establish new properties of Artin groups by either using the method in \cite{huang2023labeled}, or trying to verify the assumption of Theorem~\ref{thm:CharneyDavis}, as the first step to understand null-homotopy of all loops of length $<2\pi$ is to understand the same question for edge loops. 

Despite their important roles, the behavior of these short edge loops is far from being well-understood. In this article, we introduce a new method for studying the geometry of spherical Deligne complexes, which is effective for studying certain short edge loops. We discuss results on short edge loops in Section~\ref{subsec:fill 6-cycle}, applications to $K(\pi,1)$-conjecture in Section~\ref{subsec:K(pi,1) application}, and the method for understanding spherical Deligne complexes in Section~\ref{subsec:geometrization}.

\subsection{Cycles in some spherical Deligne complexes}
\label{subsec:fill 6-cycle}

Each finite Coxeter $W_S$ can be obtained by taking an appropriate fundamental domain in the round sphere $\mathbb S^n$ which is a convex spherical simplex, and considering the group of isometries generated by reflections along co-dimension 1 faces of this simplex. Then $\mathbb S^n$ is tessellated by copies of this simplex, and such tessellation is preserved by the action of $W_S$. The underlying simplicial complex of this tessellation is called the \emph{Coxeter complex} $\bC_S$ of $W_S$, which can be described algebraically as follows. Vertices of $\bC_S$ are in 1-1 correspondence with left cosets of form $\{gW_{S\setminus \{s\}}\}_{g\in W_S,s\in S}$, where $W_{S\setminus \{s\}}$ denotes the subgroup of $W$ generated by $S\setminus \{s\}$; a collection of vertices of $\bC_S$ span a simplex if the corresponding cosets have non-empty common intersection. Let $A_S$ be the associated spherical Artin groups. One can define a similar complex by considering the cosets $\{gA_{S\setminus \{s\}}\}_{g\in A_S,s\in S}$ instead, which gives the \emph{spherical Deligne complex} $\Delta_S$ associated with the spherical Artin group $A_S$. A vertex of $\Delta_S$ is of \emph{type $\hat s$} if it corresponds to $gA_{S\setminus \{s\}}$.
The complex $\Delta_S$ is a union of infinitely many round spheres, intersecting each other in an intricate pattern.

An $n$-cycle in $\Delta_S$ is an edge loop made of $n$-edges. The question we need to understand is that given a ``short'' $n$-cycle $\omega$ in $\Delta_S$, what is a way to fill in $\omega$ in the 2-skeleton of $\Delta_S$ by a 2-disk made of the least number of 2-simplices of $\Delta_S$. We will assume $\omega$ is embedded and \emph{induced}, i.e. two non-adjacent vertices in the cycle are also non-adjacent in $\Delta_S$, otherwise we are reduced to filling in $m$-cycles with $m<n$. The case $n=3$ follows from \cite[Lemma 4.3.2]{CharneyDavis}, where any such cycle bounds a 2-simplex - this plays a key role in the solution of $K(\pi,1)$-conjecture and word problem for type FC Artin groups. The case $n=4$ follows from \cite[Corollary 8.2]{huang2023labeled}, where any such cycle has a \emph{center}, namely a vertex in $\Delta$ which is adjacent to each vertex of $\omega$, hence $\omega$ is filled by a disk in the 2-skeleton made of four 2-simplices. As most of the cycles we need for $K(\pi,1)$-conjecture satisfy $n\le 6$ (see Section~\ref{sec:cycle}), in this article, we focus on the case $n=6$, which is much more subtle compared the cases $n=3$ and $n=4$. It is no longer true that any embedded and induced 6-cycle $\omega$ in $\Delta_S$ has a center. Actually, the combinatorics of minimal filling disks of embedded and induced $6$-cycles are highly sensitive to the type of the spherical Artin group $A_S$, as well as the types of vertices of $\omega$.

We introduce a geometric method to understand these cycles, see Section~\ref{subsec:geometrization}. It is different in nature from the Garside theoretic method \cite{huang2023labeled} which does not work when $n>4$, and the topological methods in \cite{charney2004deligne} and \cite{haettel2021lattices} (based on \cite[Section 2]{brady2005garside}) which do not work for exceptional Artin groups. This allows us to process cycles which is not accessible by previous methods.

\begin{thm}(=Theorem~\ref{thm:weaklyflagA}+Lemma~\ref{lem:weakflag})
	\label{thm:6 cycles with center}
	Suppose $A_S$ is of type $A_n,B_n,H_3,F_4$ with $n\ge 3$. Let $\omega$ be an embedded 6-cycle in $\Delta_S$ such that the type of its vertices alternates between $\hat s$ and $\hat t$, where $\{s,t\}\in S$ does not commute. Then $\omega$ has a center.  
\end{thm}

The conclusion of this theorem breaks down if $s$ and $t$ commute. 

The most interesting cases of Theorem~\ref{thm:6 cycles with center} are type $A_n$ and $H_3$. In the exception case $H_3$, we prove more. Recall that an induced 6-cycle $\omega$ in $\Delta_S$ has \emph{quasi-center}, if there is a vertex $x\in \Delta_S$ such that $x$ is adjacent to three vertices of $\omega$ which are pairwise non-adjacent. Having a quasi-center breaks the 6-cycle $\omega$ into three 4-cycles, hence minimal filling disks of $\omega$ can be understood through minimal filing disks of 4-cycles in \cite{huang2023labeled}.

\begin{thm}(=Theorem~\ref{thm:tripleH})
	Suppose $A_S$ is of type $H_3$. Let $\omega$ be an embedded and induced 6-cycle in $\Delta_S$. Then $\omega$ has a quasi-center.
\end{thm}

Note that the theorem is not true if one replaces ``quasi-center'' by ``center'' in the conclusion. It is also interesting to compare to the case when $A_S$ is of type $A_n$, where generally induced 6-cycles in $\Delta_S$ do not have a quasi-center - this can be seen by using Deligne's description of the spherical Deligne complex in \cite{deligne}.

\begin{thm}(=Proposition~\ref{prop:F4})
	Suppose $A_S$ is of type $F_4$, with $s\in S$ being a terminal node in its Coxeter diagram. Let $\omega$ be an embedded and induced 6-cycle in $\Delta_S$ without any vertices of type $\hat s$. Then $\omega$ has a quasi-center.
\end{thm}

Complementary to the above results, some 6-cycles in type $D_n$ case are treated in a companion article \cite{huang2024Dn}, as it uses a different method. 

As a comparison to previous work, we record the following intriguing consequence of Theorem~\ref{thm:6 cycles with center}. Recall that for any subset $S'\subset S$, the \emph{relative Artin complex} $\Delta_{S,S'}$ is the full subcomplex of $\Delta_S$ spanned by vertices of type $\hat s$ with $s\in S'$.

\begin{cor}(=Corollary~\ref{cor:cat(1)})
	Suppose $A_S$ is of type $A_n$ with $n\ge 3$. Let $S'\subset S$ be a subset made of three consecutive nodes in the Coxeter diagram of $A_S$. Then $\Delta_{S,S'}$ with its simplices equipped with $A_3$-shape is CAT$(1)$.
\end{cor}

The case $n=3$ is the main result of \cite{charney2004deligne}.

\begin{figure}
	\centering
	\includegraphics[scale=0.8]{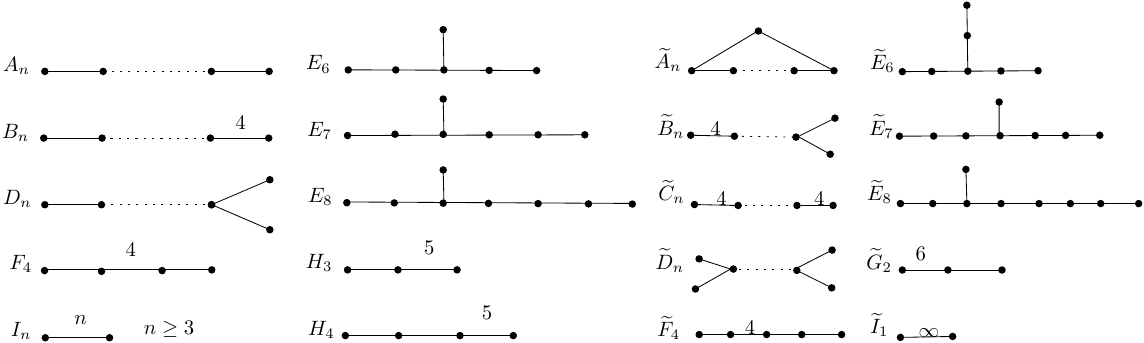}
	\caption{Spherical and affine diagrams: unlabeled edges are assumed to have label 3.}
	\label{fig:classical}
\end{figure}
\subsection{Application to the $K(\pi,1)$-conjecture}
\label{subsec:K(pi,1) application}
Applications of results in Section~\ref{subsec:fill 6-cycle} to $K(\pi,1)$-conjecture divides into two categories, those which are obtained without too much effort, and those that need substantially more work. In this article we discuss the first category. The second category of results will appear elsewhere.

\textbf{We caution the reader that our results are stated using Coxeter diagrams, which we also called Coxeter-Dynkin diagrams or Dynkin diagrams in \cite{garside,huang2023labeled}. There is a different graph used to represent Artin groups in the literature, called the defining graph or presentation graph, which we are not using here.} 

The $K(\pi,1)$-conjecture for an arbitrary Artin group reduces to the $K(\pi,1)$-conjecture for free-of-infinity Artin groups (i.e. there are no $\infty$-labeled edges in the Coxeter diagram) by work of Charney-Davis, Ellis-Sk{\"o}ldberg and Godelle-Paris \cite{CharneyDavis,ellis2010k,godelle2012k}. Among free-of-infinity Artin groups, the $K(\pi,1)$-conjecture is known in the 2-dimension case by Charney-Davis \cite{CharneyDavis}, and a few families that are not exactly 2-dimensional but have strong 2-dimensional features by work of Juhasz and Goldman \cite{juhasz2018relatively,goldman2022k,juhasz2023class}. 

Known cases for free-of-infinity Artin groups in dimension $\ge 3$ are limited. Charney proved the conjecture for Artin groups with irreducible spherical parabolics being 2-dimensional or $A_3$ \cite{charney2004deligne}. 
The conjecture is known for spherical Artin groups by Deligne \cite{deligne} and affine Artin groups by Paolini-Salvetti \cite{paolini2021proof} (based on McCammond-Sulway \cite{mccammond2017artin}). These two classes correspond to Coxeter diagrams in Figure~\ref{fig:classical}. The conjecture is also proved for Artin groups whose Coxeter diagrams are trees \cite{huang2023labeled} with some constraints on edge labeling.

As the $K(\pi,1)$ conjecture is settled for Artin groups in Figure~\ref{fig:classical} whose Coxeter groups acts on $\mathbb S^n$ and $\mathbb E^n$  \cite{deligne,paolini2021proof}, it is natural to ask the same question for \emph{hyperbolic type Artin groups}, i.e., Artin groups associated with Coxeter groups acting on the $n$-dimensional hyperbolic space $\mathbb H^n$ with finite covolume. This was completely understood in the case of $\mathbb H^2$ \cite{CharneyDavis}, with a recent alternative proof \cite{delucchi2022dual}. However, for $\mathbb H^3$, only a few sporadic examples of Artin groups are known to satisfy the $K(\pi,1)$-conjecture \cite{charney2004deligne,garside,huang2023labeled}. As an immediate consequence of Section~\ref{subsec:fill 6-cycle}, we give an almost complete treatment in the $\mathbb H^3$ case.
\begin{thm}(=Corollary~\ref{cor:3hyperbolic})
	\label{thm:3hyperbolic}
	Let $W_S$ be a Coxeter group acting properly on $\mathbb H^3$ by isometries with finite covolume. Assume $W_S\neq [3,5,3]$. Then the $K(\pi,1)$-conjecture holds for the associated Artin group $A_S$.
\end{thm}

Here $[3,5,3]$ denotes the Coxeter group whose Coxeter diagram is a linear graph with three edges, such that consecutive edges are labeled $3,5,3$. This remaining group has some exceptional feature and will be treated separately in \cite{HP} along with other classes.
Section~\ref{subsec:fill 6-cycle} also implies new results in dimension 4, with some restriction on the shape of the fundamental domain.

\begin{thm}(=Corollary~\ref{cor:quasilanner})
	\label{thm:quasilanner}
	Let $W_S$ be a Coxeter group acting properly on $\mathbb H^n$ by isometries with its fundamental domain being a finite volume non-compact simplex. Suppose $n\le 4$. Then the $K(\pi,1)$-conjecture holds for the associated Artin group $A_S$.
\end{thm}

For higher dimensional Artin groups, we prove the following.

\begin{thm}(=Theorem~\ref{thm:complete bipartite})
	\label{thm:bipartite}
	Let $A_S$ be an Artin group such that its Coxeter diagram is complete bipartite. Then $A_S$ satisfies the $K(\pi,1)$-conjecture.
\end{thm}

A \emph{complete bipartite graph} $\Lambda$ is a graph whose vertices can be
split into two non-empty subsets $V_1$ and $V_2$ so that $\Lambda$ has an edge between two vertices if and
only if one of these vertices belongs to $V_1$ and the other belongs to $V_2$.
We write $K_{m,n}$ to denote the complete bipartite graph with $m$ vertices on one
subset and $n$ vertices in the other.

In each dimension $\ge 5$, there are infinitely many free-of-infinity Artin groups satisfying the assumption of Theorem~\ref{thm:bipartite} whose $K(\pi,1)$-conjecture were not previously known. The edge labels of the diagrams in Theorem~\ref{thm:bipartite} is allowed to be arbitrary. While a few instances of Artin groups in Theorem~\ref{thm:bipartite} are of hyperbolic type, most of them are not, as hyperbolic type Artin groups exist only up to a certain dimension.

Artin groups with complete bipartite Coxeter diagrams were used by McCammond \cite[Section 8]{mccammond2017mysterious} to demonstrate how little we understand about Artin groups. When the Coxeter diagram is the star $K_{1,n}$ and all edges are labeled by $3$, the only previously known cases of $K(\pi,1)$-conjecture were $K_{1,1}$ (type $A_2$) and $K_{1,2}$ (type $A_3$) in \cite{MR150755}, $K_{1,3}$  (type $D_4$) in \cite{MR422674,deligne} and $K_{1,4}$ (type $\widetilde D_4$)  recently in \cite{paolini2021proof}. The $K(\pi,1)$-conjecture for the next diagram $K_{1,5}$, whose Coxeter group acts on $\mathbb H^5$ with finite volume fundamental domain, was already not known before. 

\medskip
We now remark briefly on the passage from results in Section~\ref{subsec:fill 6-cycle} to $K(\pi,1)$-results. An Artin group $A_S$ is \emph{almost spherical}, if it is not spherical, but $A_{S\setminus\{s\}}$ is spherical for any $s\in S$. Coxeter diagrams of irreducible almost spherical Artin groups are either contained in the last two columns of Figure~\ref{fig:classical}, or in \cite[Table 6.2]{davis2012geometry}.
A sub-diagram $\Lambda'$ of a Coxeter diagram $\Lambda$ is \emph{almost spherical}, if $A_{\Lambda'}$ is almost spherical.

The approach in \cite{huang2023labeled} studies each Artin group via the information of what are the almost spherical sub-diagrams of its Coxeter diagram, and how these sub-diagrams sit inside the ambient diagram. This information can be further related to the study of short cycles in spherical Deligne complexes. Hence combining \cite{huang2023labeled} and Section~\ref{subsec:fill 6-cycle} output results on $K(\pi,1)$-conjecture for Artin groups such that the types and relative positions of almost spherical sub-diagrams in their Coxeter diagrams satisfy certain restrictions, depending on the list of short cycles in spherical Deligne complexes we can handle. This gives the above theorems, as well as some more technical (but more general) results in the body of the article, see Corollary~\ref{cor:tree}, and Proposition~\ref{prop:reduction}. Instead of giving the full statements, we give some examples below to demonstrate them. 

\begin{figure}[h]
	\centering
	\includegraphics[scale=0.8]{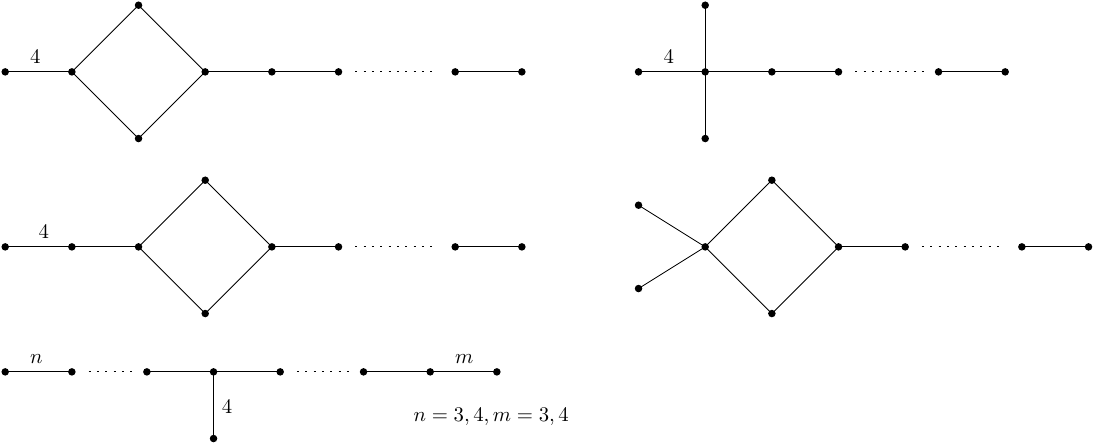}
	\caption{Some examples: unlabeled edges are assumed to have label 3.}
	\label{fig:seven families}
\end{figure}

\begin{cor}(=Corollary~\ref{cor:seven families})
	All Artin groups whose Coxeter diagrams belong to Figure~\ref{fig:seven families} satisfy the $K(\pi,1)$-conjecture.
\end{cor} 

All these are new for $K(\pi,1)$-conjecture. Figure~\ref{fig:seven families} contains examples of any given dimension and examples with larger and larger irreducible spherical Artin subgroups.
The main difference between examples here and some previous examples in \cite[Figure 3]{huang2023labeled} is the types of almost spherical sub-diagrams they are allowed to contain.

We also obtain a corollary regarding the torsion and center of some Artin groups, using \cite{jankiewicz2023k}, which belongs to fundamental unsolved problems for Artin groups.

\begin{cor}
	Let $\Lambda$ be a connected Coxeter diagram satisfying the assumptions in one of the previous theorems. Then $A_\Lambda$ is torsion-free and it has a trivial center when $\Lambda$ is not spherical.
\end{cor}

\subsection{Non-positive curvature pieces in spherical Deligne complexes}
\label{subsec:geometrization}

Now we discuss the main ingredient in this article which is used to prove results on filling short cycles in Section~\ref{subsec:fill 6-cycle}, as well as some ideas in the proof. A standard situation when one knows how to fill in cycles in the space, and has strong control over the geometry and possibly combinatorics of the filling disk, is that the space satisfies a fine notion of non-positively curvature, like CAT$(0)$. So it would be ideal if all the spherical Deligne complexes for spherical Artin groups could be metrized as CAT$(0)$ spaces. However, this naive hope breaks down as spherical Deligne complexes are homotopic to wedges of spheres \cite{deligne}, so they can not admit CAT$(0)$ metric. Nevertheless, this does not rule out the possibility of finding large non-positively curved ``islands'' inside spherical Deligne complexes, and our strategy still works if some island is large enough to contain the cycle we are interested in. This situation bears some similarity to the geometrization of 3-manifolds/orbifolds - while a 3-manifold might not be metrized by a single Thurston geometry, it contains maximal ``islands'' which can be indeed metrized by a single Thurston geometry. This consideration leads us to the following question.

\begin{que}
	\label{que:main}
	Given a spherical Artin group $A_S$ acting on its spherical Deligne complex $\Delta_S$. We want to find a subgroup $G\le A_S$, and an $G$-invariant and $G$-cocompact subcomplex $Y\subset \Delta_S$ such that $Y$ admits an $G$-invariant metric of ``non-positive curvature''. Moreover, we want $G$ and $Y$ to be as large as possible.
\end{que}

%This is inspired by Thurston's geometrization theorem for 3-manifolds which bears some similarity to our situation in the following sense. Suppose we replace $A_S$ and $\Delta_S$ by the fundamental group $G$ of an appropriate 3-manifold (or orbifold) acting on its universal cover $X$, and replace non-positive curvature by one of the eight Thurston geometries. Then the geometrization theorem implies that one can find one or a few maximal subgroups $H\le G$, and a maximal $H$-invariant subspace $Y\subset X$ such that $Y$ admits an $H$-invariant metric corresponding to the given Thurston geometry --  $Y/H$ is a piece in the geometric decomposition of the 3-manifold (or orbifold).

There is some flexibility of what we mean by ``non-positive curvature'' in Question~\ref{que:main}. Besides CAT$(0)$, one may consider other forms of non-positive curvature \cite{lang2013injective,descombes2015convex,weaklymodular,chalopin2020helly}. The bottom line is that such a form of non-positive curvature should give strong control on filling disks of cycles, so forms of non-positive curvature that are more ``coarse'' in nature might not work well for our purpose, see \cite{webb2020contractible}.

Now we describe the non-positively curved islands we found in detail.
Suppose $A_S$ is irreducible and spherical.
Let the associated pure Artin group be $PA_S$. Let $\overline{PA}_S$ be the central quotient of $PA_S$. As $A_S$ is spherical, we know $PA_S\cong \mathbb Z\times \overline{PA}_S$. Note that $\mathbb Z$ factor of $PA_S$ is canonical as it corresponds to the center of the group. However, the $\overline{PA}_S$ factor is not canonical - there are many different subgroups of $PA_S$ isomorphic to $\overline{PA}_S$, which are called \emph{$\overline{PA}_S$-sections} in $PA_S$.

Recall that quotient $\Delta_S/PA_{S}$ is isomorphic to the Coxeter complex $\mathfrak C_S$. This gives a folding map $\pi:\Delta_S\to \bC_S$. Let $H$ be a \emph{wall} in $\mathfrak C_S$ (i.e., fix points of a reflection), and let $K_H$ be the maximal subcomplex of $\mathfrak C_S$ contained in an open halfspace bounded by $H$. Then $\pi^{-1}(K_H)$ is not connected, and the stabilizer of each connected component of $\pi^{-1}(K_H)$ for the action $PA_S\curvearrowright \Delta_S$ is a $\overline{PA}_S$-section in $PA_S$.

A geodesic metric space is \emph{injective} if whenever a collection of closed metric balls in the space have pairwise non-empty intersection, then their common intersection is non-empty. This is a form of non-positive curvature \cite{lang2013injective}.

\begin{prop}(=Proposition~\ref{prop:An Helly})
	Suppose $A_S$ is of type $A_n$. Then for each wall $H$ of $\mathfrak C_S$, each component of $\pi^{-1}(K_H)$ admits a piecewise $\ell^\infty$ injective metric that is invariant under the action of the associated $\overline{PA}_S$-section.
\end{prop}

From a topological viewpoint, a component of $\pi^{-1}(K_H)$ arises naturally from the deconing of the associated $A_n$-type arrangement of hyperplanes.
When $\Delta_S$ is 2-dimensional, the component was considered by Falk in \cite{falk1995k} for a different purpose, with a different metric. 
Though the description of the complex in Falk's paper is different as well, it gives the same complex (see Section~\ref{subsec:deligne complex} for more details). Thus, from now on, we will refer to components of $\pi^{-1}(K_H)$ as \emph{Falk subcomplexes} of $\Delta_S$ associated with $H$.

There is a similar result about injective metric on Falk subcomplexes in type $B_n$ case, see Proposition~\ref{prop:Bn Helly}. We will not need Proposition~\ref{prop:Bn Helly} in this article, although we will need it in another article, so we put it in the appendix.

Now we turn to the case that $A_S$ is of type $H_3$, which is the most challenging case in this article. Let $\pi:\Delta\to \mathfrak C_S$ and $H$ as before. Unfortunately, those Falk subcomplexes do not appear to admit forms of non-positive curvature (i.e. CAT$(0)$, injective, conformally CAT$(0)$, etc) which are invariant under the action of the associated sections. However, this is not fatal to our strategy. We will solve instead a variation of Question~\ref{que:main}.

\begin{que}
	Given a spherical Artin group $A_S$ acting on its spherical Deligne complex $\Delta_S$. We want to find a subgroup $G\le A_S$, and an $G$-invariant and $G$-cocompact subcomplex $Y\subset \Delta_S$ together with a quotient homomorphism $q: G\to \bar G$, a simplicial action of $\bar G$ on a simplicial complex $\bar Y$, and an $q$-equivariant quotient simplicial map $Y\to \bar Y$ such that $\bar Y$ admits an $\bar G$-invariant metric of ``non-positive curvature''. Here we want $G$ and $Y$ to be as large as possible, and we want the fibers of $G\to \bar G$ and $Y\to \bar Y$ to be as small as possible.
\end{que}
The only difference between this question and Question~\ref{que:main} is that we do not require $Y$ to be non-positively curved on the nose. Instead, we are allowed to pass to an appropriate quotient $\bar Y$ of $Y$ which is non-positively curved. Answering this question still sheds lights on understanding minimal filling cycles in $\Delta_S$. Indeed, suppose we have a cycle $\omega$ that is in the ``island'' $Y$ of $\Delta_S$. Then we want to use the non-positive curvature in $\bar Y$ to fill in the image $\bar \omega$ of $\omega$, then try to lift at least part of the information of filling disk in $\bar Y$ back to $Y$. In order for the last step to work, we want the quotient map $Y \to \bar Y$ to lose as little information as possible.

When $A_S$ is of type $H_3$, we find several different pairs $\bar G$ and $\bar Y$, and combine their power to understand cycles in $\Delta_S$. The exact description of these pairs is a bit technical, which we refer to Section~\ref{sec:subarrangement}. Here we only explain the idea of finding such pairs. We take $G$ and $Y$ as before - $Y$ is a Falk subcomplex, and $G$ is the associated section. Here $G$ can be realized as the fundamental group of a manifold of form $\mathbb C^2-\cup_{\lambda\in \Lambda} H_\lambda$ where each $H_\lambda$ is an affine hyperplane in $\mathbb C^2$. If we consider a subset $\Lambda'\subset \Lambda$, then the inclusion $\mathbb C^2-\cup_{\lambda\in \Lambda} H_\lambda\hookrightarrow \mathbb C^2-\cup_{\lambda\in \Lambda'} H_\lambda$ naturally induces a quotient of $G$. The pairs $\bar G$ and $\bar Y$ are obtained by considering appropriate subsets $\Lambda'$ of $\Lambda$.

It remains to ask, what if the 6-cycle $\omega$ in $\Delta_S$ is not contained in any ``islands''? As these islands are chosen to be Falk subcomplexes, the failure of being contained in any islands implies the cycle has to be ``spread-out'' in an appropriate sense, hence can be handled effectively by the retraction map defined in \cite{godelle2012k,charney2014convexity,blufstein2023parabolic,godelle2023parabolic,digne2024retraction} from an Artin group to its parabolic subgroups.

\subsection{Structure of the article and reading guide}
The article has three parts. Part 1 consists of Section~\ref{sec:prelim2}, and it is a collection of preliminaries. Part 2 consists of Section~\ref{sec:F4} to Section~\ref{sec:H3}, where we study 6-cycles in different types of spherical Artin complexes. More precisely, Section~\ref{sec:F4} is about the type $F_4$ case. Section~\ref{sec:prelim1} and Section~\ref{sec:AD} are about the type $A_n$ case. Section~\ref{sec:subarrangement} and Section~\ref{sec:H3} are about the type $H_3$ case. Part 3 consists of Section~\ref{sec:weakly flag} to Section~\ref{sec:high}, where we deduce new cases of the $K(\pi,1)$-conjecture.

%All the sections in Part 2 of the article are independent of each other, apart from that  Section~\ref{sec:H3} relies on Section~\ref{sec:subarrangement}. So after reading Part 1, the reader can take any section in Part 2 and start reading without interruption, with the only exception that reading Section~\ref{sec:H3} requires material from Section~\ref{sec:subarrangement}.

Part 2 and Part 3 are mostly independent from each other.
The main results of Part 2 are Proposition~\ref{prop:F4}, Theorem~\ref{thm:weaklyflagA}, and Theorem~\ref{thm:tripleH}.
After reading Part 1, if the reader is willing to take these three main results of Part 2 for granted, then the reader can go to Part 3 directly, to see how these three results imply new cases of the $K(\pi,1)$-conjecture. 

%Here is a more detailed description of Part 3. In Section~\ref{sec:propagation} we prove some partial results on how certain properties of 6-cycles in spherical Artin complexes propagate to more general Artin complexes. Section~\ref{sec:propagation} is also independent from Part 2. In Section~\ref{sec:3-dim} to Section~\ref{sec:high}, we use the results from Part 2 and Section~\ref{sec:propagation} to prove all the main theorems.

\subsection{Acknowledgment}
The author benefits greatly from several long and detailed conversations with the following colleagues on topics closely related to this article: Ruth Charney, Mike Davis, Thomas Haettel, Damian Osajda and Piotr Prytycki. The author also thanks Grigori Avramidi, Mladen Bestvina, Srivatsav Kunnawalkam Elayavalli, Michael Falk, Alexandre Martin, Jon McCommand, Katherine Goldman, Nima Hoda, Camille Horbez, Nick Salter, Rachel Skipper, Antoine Song, and Frank Wagner for helpful related conversation.

The author is partially supported by a Sloan fellowship and NSF DMS-2305411. The author thanks the Centre de Recherches Mathématiques in Montreal for the thematic program Geometric Group Theory in 2023 where part of the work was done. 
The author thanks the Institut Henri Poincaré (UAR 839 CNRS-Sorbonne
Université) for the trimester program Groups acting on fractals, Hyperbolicity and
Self-similarity in 2022 (through
LabEx CARMIN, ANR-10-LABX-59-01) where part of the work was done.

\setcounter{tocdepth}{1}
\tableofcontents

\section{$K(\pi,1)$-conjecture via relative Artin complexes}
\label{sec:prelim2}

This is a dual-purpose section. It serves as a summary of preliminaries for the whole article, as well as a review of an approach to $K(\pi,1)$-conjecture in \cite{huang2023labeled} which provides motivation for the preliminary material.
The presentation of the approach to $K(\pi,1)$-conjecture here has a few new parts compared to \cite{huang2023labeled}. While some parts of the strategy are still under construction, the main frame is already transparent, and it is simple.

\subsection{Artin complexes and relative Artin complexes}
Let $W_S$ and $A_S$ be the Coxeter group and Artin group with generating set $S$. If $\Lambda$ is a Coxeter diagram, then we will also write $W_\Lambda$ and $A_\Lambda$ for the Coxeter group and Artin group with Coxeter diagram $\Lambda$.
There is a homomorphism $A_S\to W_S$, whose kernel is called the \emph{pure Artin group}, and is denoted by $PA_S$.

Any $S'\subset S$ generates a subgroup of $A_S$ which is also an Artin group, whose Coxeter diagram is the full subgraph $\Lambda_{S'}$ of $\Lambda$ spanned by $S'$ \cite{lek}. This subgroup is called a \emph{standard parabolic subgroup} of type $S'$. A \emph{parabolic subgroup} of $A_S$ of type $S'$ is a conjugate of a standard parabolic subgroup of type $S'$. A parabolic subgroup of $A_S$ is \emph{reducible} if its type $S'$ admits a disjoint non-trivial decomposition $S'_1\sqcup S'_2$ such that each element in $S'_1$ commutes with every element in $S'_2$. If such decomposition does not exist, then the parabolic subgroup is \emph{irreducible}.

Recalled that the \emph{Artin complex}, introduced in \cite{CharneyDavis} and further studied in \cite{godelle2012k,cumplido2020parabolic}, defined as follows. For each $s\in S$, let $A_{\hat s}$ be the standard parabolic subgroup generated by $S\setminus\{s\}$. Let $\Delta_S$ be the simplicial complex whose vertex set is corresponding to left cosets of $\{A_{\hat s}\}_{s\in S}$. Moreover, a collection of vertices span a simplex if the associated cosets have a nonempty common intersection. It follows from \cite[Proposition 4.5]{godelle2012k} that $\Delta_S$ is a flag complex. The Artin complex is an analog of \emph{Coxeter complex} in the setting of Artin group. The definition of a Coxeter complex $\bC_S$ of a Coxeter group $W_S$ is almost identical to Artin complex, except one replaces $A_{\hat s}$ by $W_{\hat s}$, which is the standard parabolic subgroup of $W_S$ generated by $S\setminus \{s\}$. Each vertex of $\bC_S$ or $\Delta_S$ corresponding a left coset of $W_{\hat s}$ or $A_{\hat s}$ has a \emph{type}, which is defined to be $\hat s=S\setminus \{s\}$. The \emph{type} of each face of $\bC_S$ or $\Delta_S$ is defined to be the subset of $S$ which is the intersection of the types of the vertices of the face. In particular, the type of each top-dimensional simplex is the empty set. The quotient complex of $\Delta_S$ under the action of the pure Artin group is isomorphic to $\bC_S$.

Note that if each $A_{\hat s}$ is spherical, but $A_S$ is not spherical, then $\bD_S$ is isomorphic to the barycentric subdivision of $\Delta_S$. When $A_S$ is spherical, then sometimes $\Delta_S$ is also called \emph{spherical Deligne complex} - indeed, let $\ca$ be the reflection arrangement associated with the spherical Artin group $A_S$. Then the spherical Deligne complex $\bSD_\ca$ defined in Section~\ref{subsec:deligne complex} is isomorphic to $\Delta_S$.

\begin{definition}(\cite{huang2023labeled})
	Let $A_S$ be an Artin group with Coxeter diagram $\Lambda$. Let $S'\subset S$. The \emph{$(S,S')$-relative Artin complex $\Delta_{S,S'}$} is defined to be the induced subcomplex of the Artin complex $\Delta_S$ of $A_S$ spanned by vertices of type $\hat s$ with $s\in S'$. In other words, vertices of $\Delta_{S,S'}$ correspond to left cosets of $\{A_{\hat s}\}_{s\in S'}$, and a collection of vertices span a simplex if the associated cosets have nonempty common intersection.
	
	Let $\Lambda'$ be the induced subgraph of $\Lambda$ spanned by $S'$. Then we will also refer an $(S,S')$-relative Artin complex as $(\Lambda,\Lambda')$-relative Artin complex, and denote it by $\Delta_{\Lambda,\Lambda'}$.
\end{definition}

The links of vertices in relative Artin complexes can be computed via the following result from \cite[Lemma 6.4]{huang2023labeled}.

\begin{lem}
	\label{lem:link}
	Let $\Delta$ be the $(\Lambda,\Lambda')$-relative Artin complex, and let $v\in \Delta$ be a vertex of type $\hat s$ with $s\in \Lambda'$. Let $\Lambda_s$ and $\Lambda'_s$ be the induced subgraph of $\Lambda$ and $\Lambda'$ respectively spanned all the nodes which are not $s$. 
	Then the following are true.
	\begin{enumerate}
		\item There is a type-preserving isomorphism between $\lk(v,\Delta)$ and the $(\Lambda_s,\Lambda'_s)$-relative Artin complex.
		\item Let $I_s$ be the union of connected components of $\Lambda_s$ that contain at least one component of $\Lambda'_s$. Then $\Lambda'_s\subset I_s$ and there is a type-preserving isomorphism between $\lk(v,\Delta)$ and the $(I_s,\Lambda'_s)$-relative Artin complex.
		\item Let $\{I_i\}_{i=1}^k$ be the connected components of $I_s$. Then $\lk(v,\Delta)=K_1*\cdots*K_k$ where $K_i$ is the induced subcomplex of $\lk(v,\Delta)$ spanned by vertices of type $\hat t$ with $t\in I_i$.
	\end{enumerate}
\end{lem}

\begin{cor}(\cite[Corollary 6.5]{huang2023labeled})
	\label{cor:adj}
	Let $\Delta$ be the $(\Lambda,\Lambda')$-relative Artin complex. For $i=1,2,3$, let $x_i\in \Delta$ be a vertex of type $\hat s_i$ with node $s_i\in \Lambda'$. Suppose $s_1$ and $s_3$ are in different components of $\Lambda\setminus\{s_2\}$.
	
	If $x_1$ is adjacent to $x_2$, and $x_2$ is adjacent to $x_3$, then $x_1$ is adjacent to $x_3$.
\end{cor}

Recall the following regarding homotopy types of relative Artin complexes.
\begin{lem}(\cite[Lemma 7.1]{huang2023labeled})
	\label{lem:dr}
	Let $\Lambda_1\subset \Lambda_2$ be two induced subgraphs of $\Lambda$ such that $\Lambda_2\setminus\Lambda_1$ contain exactly one node, denoted by $s$. If $\lk(x,\Delta_{\Lambda,\Lambda_2})$ is contractible for some (hence any) vertex $x\in \Delta_{\Lambda,\Lambda_2}$ of type $\hat s$, then $\Delta_{\Lambda,\Lambda_2}$ deformation retracts onto $\Delta_{\Lambda,\Lambda_1}$.
\end{lem}

\subsection{$K(\pi,1)$-conjecture via relative Artin complexes}
Our point of departure is the following.

\begin{thm}\cite[Theorem 3.1]{godelle2012k}
	\label{thm:kpi1}
	If $\Delta_S$ is contractible and each $\{A_{\hat s}\}_{s\in S}$ satisfies the $K(\pi,1)$-conjecture, then $A_S$ satisfies the $K(\pi,1)$-conjecture.
\end{thm}

We caution the reader that Artin complexes are not always contractible. Recall that an Artin group is \emph{spherical}, if the associated Coxeter group is finite. The associated Coxeter complex is a sphere, and the associated Artin complex is also called the \emph{spherical Deligne complex}, which is homotopy equivalent to a wedge of spheres, so, in particular, it is not contractible \cite{deligne}. However, 
it is known that the Coxeter complex is contractible whenever the associated Coxeter group is not finite. It is natural to conjecture that the $\Delta_S$ is contractible whenever $A_S$ is not spherical, which will imply the $K(\pi,1)$-conjecture for all Artin groups by Theorem~\ref{thm:kpi1} and \cite{deligne}.

By Lemma~\ref{lem:link}, the link $\lk(v,\Delta_S)$ of a vertex of type $\hat s$ is a smaller Artin complex $\Delta_{\hat s}$ with $\hat s=S\setminus \{s\}$. Thus, if $\Delta_{\hat s}$ is contractible, then the link of each vertex of type $\hat s$ is contractible. So $\Delta_S$ can deformation retract onto the relative Artin complex $\Delta_{S,S\setminus\{s\}}\subset \Delta_S$ by Lemma~\ref{lem:dr}. If $\Delta_{S,S\setminus\{s\}}$ has vertices with contractible links, then we can deformation retract onto an even smaller relative Artin complex. We would like to keep performing this kind of deformation retraction until one reaches a ``core'' where such deformation retraction is not possible, which can be described as follows.

An Artin group $A_S$ is \emph{almost spherical} if it is not spherical, but for each $s\in S$, $A_{\hat s}$ is spherical. A subset $T\subset S$ is \emph{almost spherical} if $A_T$ is almost spherical. Almost spherical Artin groups are classified: they are either affine, or one of the hyperbolic Lann\'er types \cite[Table 6.2]{davis2012geometry}. Note that if $T\subset S$ is almost spherical, then the link of each vertex $x\in \Delta_{S,T}$ is isomorphic to $\Delta_{S/\{s\},T/\{s\}}$ by Lemma~\ref{lem:link}. As $T\setminus\{s\}$ is spherical, $\Delta_{S/\{s\},T/\{s\}}$ contain embedded top-dimensional spheres, which is not contractible. So whenever $T\subset S$ is almost spherical, links of vertices in $\Delta_{S,T}$ are no longer contractible. The subcomplex $\Delta_{S,T}$ will serve as a core subcomplex of $\Delta_S$ as a terminal subcomplex in our strategy of performing deformation retraction. We conjecture that the cores are contractible. 

\begin{conj}[\cite{huang2023labeled}]
	\label{conj:contractible}
	Suppose $S'\subset S$, and $S'$ is almost spherical. Suppose the Coxeter diagrams for $S'$ and $S$ are connected and do not have $\infty$-labeled edges.
	Then $\Delta_{S,S'}$ is contractible.
\end{conj}

Conjecture~\ref{conj:contractible} implies the $K(\pi,1)$-conjecture for all Artin groups \cite[Corollary 7.4]{huang2023labeled}. In the remaining subsections, we describe a strategy to prove Conjecture~\ref{conj:contractible}.

\subsection{Contractibility of core}
\label{subsec:core}
Let $\Delta_{S,S'}$ be as in Conjecture~\ref{conj:contractible}. 
We will assign a metric to each simplex of $\Delta_{S,S'}$, which induces a metric on $\Delta_{S,S'}$. The expectation is that with appropriate assignment of shapes of simplices of $\Delta_{S,S'}$, $\Delta_{S,S'}$ will be ``non-positively curved'', hence contractible.

We now discuss shapes of simplices in $\Delta_{S,S'}$. As showing ``non-positive curvature'' is sensitive to the shape of simplices, and the choice of shape depends on the type of the almost spherical Artin group/Coxeter group $A_{S'}$, we first discuss a reduction procedure to decrease the number of shapes we need to work with.

We say the Artin group $A_S$ \emph{dominates} another Artin group $A_{S'}$ if there exists a bijection $f:S\to S'$ such that for any pair $s\neq t\in S$, we have $m_{s,t}\ge m_{f(s),f(t)}$. The following is a consequence of \cite{MR42129}. We refer to Figure~\ref{fig:classical} for types of some Artin groups mentioned below.

\begin{lem}
	\label{lem:dominate}
	Suppose $A_S$ is an almost spherical type Artin group. Then 
	\begin{enumerate}
		\item either $A_S$ dominates an irreducible Artin group of Euclidean type, or $A_S$ has Coxeter diagram $[3,5,3]$ or $[5,3,3,3]$;
		\item if in addition $A_S$ is not of type $\widetilde E_6,\widetilde E_7,\widetilde E_8,\widetilde F_4,[3,5,3],[5,3,3,5]$, then $A_S$ dominates an irreducible Artin group of whose type belong to $\{\widetilde G_2,\widetilde A_n,\widetilde B_n,\widetilde C_n,\widetilde D_n\}$.
	\end{enumerate}
	
\end{lem}

Here $[3,5,3]$ is the linear Coxeter diagram whose consecutive edges are labeled by $3,5$ and $3$. Similarly, we define $[5,3,3,3]$.

\begin{definition}[Proposed metric]
	\label{def:metric2}
	Given an Artin group $A_S$, and let $S'\subset S$ be an almost spherical subset. To metrize $\Delta_{S,S'}$, it suffices to metrize the fundamental domain with respect to the action $A_S\act \Delta_{S,S'}$, which is a simplex $K_{S'}$ whose vertices are of type $\hat s$ with $s\in S'$. Now we choose a shape for this fundamental domain as follows.
	
	We assume the type of $S'$ does not belong to $\{\widetilde F_4,\widetilde E_6,\widetilde E_7,\widetilde E_8,[3,5,3],[5,3,3,3]\}$.
	By Lemma~\ref{lem:dominate} (2), there is an Artin group $A_{S''}$ dominated by $A_{S'}$ under a bijection $f:S'\to S''$ such that the type of $A_{S''}$ belongs to $\{\widetilde G_2,\widetilde A_n,\widetilde B_n,\widetilde C_n,\widetilde D_n\}$. Then associated Coxeter group $W_{S''}$ acts properly and cocompactly by on the Euclidean space $\mathbb E^n$. Let $d_{\mathbb E^n}$ be a metric on $\mathbb E^n$ invariant under the action of $W_{S''}$, to be determined later.
	The reflection hyperplanes (i.e. fixed points of conjugates of generators of $W_{S''}$) cut $\mathbb E^n$ into a simplicial complex, which is isomorphic to the Coxeter complex $\mathfrak C_{S''}$. Let $K_{S''}$ be a fundamental domain of the action $W_{S''}\act \mathfrak C_{S''}$, which is a simplex with an inherit metric from $d_{\mathbb E^n}$. Now we identify $K_{S'}$ with $K_{S''}$ with the vertex of type $\hat s$ in $K_{S'}$ with $s\in S'$ is identified with the vertex of type $\hat f(s)$ in $K_{S''}$. This gives a metric on $K_{S'}$ determined by $d_{\mathbb E^n}$.
	
	It remains to choose the metric $d_{\mathbb E^n}$. If $S''$ is of type $\widetilde G_2$, then $d_{\mathbb E^n}$ is chosen to be the standard Euclidean metric. If $S''$ is of type $\widetilde B_n,\widetilde C_n,\widetilde D_n$, then $d_{\mathbb E^n}$ is chosen to be the $\ell^\infty$-metric on $\mathbb E^n$ (which is invariant under the Coxeter group action). If $S''$ is of type $\widetilde A_n$, then the Coxeter complex $\mathfrak C_{S''}$ can be identified with the subset $H=\{x\in \mathbb E^{n+1}\mid x_1+x_2+\cdots+x_{n+1}=0\}$ cut out by the hyperplanes $x_i-x_j\in \mathbb Z$ for $1\le i\neq j\le n+1$. Then the metric $d_{\mathbb E_n}$ on $H$ is chosen to be the induced metric from $(\mathbb E^{n+1},d_{\ell^\infty})$. Again $d_{\mathbb E_n}$ is invariant under the action of the type $\widetilde A_n$ Coxeter group.
\end{definition}

In the special case of $S=S'\in \{\widetilde A_n,\widetilde C_n\}$, the metric in Definition~\ref{def:metric2} was defined by Haettel \cite{haettel2021lattices}, which is a source of inspiration for our metric.

Here is a more concrete description of the $\ell^\infty$-metric on $K_{S''}$ when $S''$ is of type $\{\widetilde B_n,\widetilde C_n,\widetilde D_n\}$. In the case of type $\widetilde C_n$, the Coxeter complex $\mathfrak C_{S''}$ is isomorphic to $\mathbb E^n$ cut out by the hyperplanes $x_i\pm x_j\in \mathbb Z$ and $x_i\in \frac{1}{2}\mathbb Z$. Then the fundamental domain $K_{S''}$ can be chosen to be the subset of $\mathbb E^n$ defined by $0\le x_1\le x_2\le\cdots\le x_n\le \frac{1}{2}$, equipped with the $\ell^\infty$ metric. If $S''$ is of type $\widetilde B_n$, then the subdivision of the associated Coxeter complex $\mathfrak C_{S''}$ as in Definition~\ref{def:subdivision} is isomorphic to the Coxeter complex of type $\widetilde C_n$, thus a type $\widetilde B_n$ fundamental domain is the union of two copies of type $\widetilde C_n$ fundamental domain along an appropriate codimension 1 face. If $S''$ is of type $\widetilde D_n$, then a subdivision of the associated Coxeter complex $\mathfrak C_{S''}$ as in Definition~\ref{def:subdivisionD} is isomorphic to the Coxeter complex of type $\widetilde C_n$. Hence a $\widetilde D_n$-fundamental domain is the union of four copies of  $\widetilde C_n$ fundamental domain.

Recall that a \emph{geodesic bicombing} in a metric space $X$, is the assignment of a geodesic segment from $x$ to $y$, for each ordered pair of points $(x,y)$ in $X$. We do not require the geodesic segment from $x$ to $y$ is the same as the geodesic segment from $y$ to $x$. Note that in Definition~\ref{def:metric2}, each top-dimensional simplex has geodesic bicombing coming from the linear structure on each simplex. We expect these geodesic bicombing to fit together to form a geodesic bicombings on the whole space such that the bicombing varies continuously with respect to their endpoints, hence one can show $\Delta_{S,S'}$ is contractible by using the geodesic contraction.

\begin{conj}
	\label{conj:metric2}
	Let $A_S$ be an Artin group and suppose $S$ contains a subset $S'$ which is almost spherical, but its type is not contained in $\{\widetilde F_4,\widetilde E_6,\widetilde E_7,\widetilde E_8,[3,5,3],[5,3,3,3]\}$. Then $\Delta_{S,S'}$ equipped with the metric in Definition~\ref{def:metric2} is a metric space with convex geodesic bicombing in the sense of \cite{descombes2015convex}, hence contractible.
\end{conj}

One naturally asks, in Definition~\ref{def:metric2}, why do not choose $d_{\mathbb E^n}$ to be the standard Euclidean metric in all cases. This is actually one of the most important points in our strategy: if $d_{\mathbb E^n}$ is the standard flat metric, then one would conjecture that $\Delta_{S,S'}$ with the induced metric is CAT$(0)$, hence contractible. Checking  $\Delta_{S,S'}$ is CAT$(0)$ reduces to checking the link $\lk(x,\Delta_{S,S'})$ of each vertex $x\in \Delta_{S,S'}$, endowed with the natural piecewise spherical metric from the ambient space, is CAT$(1)$. Checking $\lk(x,\Delta_{S,S'})$ further reduces to showing for any loop $\omega$ of length $<2\pi$ in $\lk(x,\Delta_{S,S'})$, $\omega$ admits a length non-increasing homotopy to the trivial loop \cite{bowditch1995notes}. However, in high dimensions, there are great difficulties to deal with the large spaces of loops of length $<2\pi$ on the highly singular space $\lk(x,\Delta_{S,S'})$. Actually, two previous approaches to $K(\pi,1)$-conjecture encounter exactly the same difficulty  \cite{CharneyDavis,MR3127810}.

To get around this difficulty, we choose different metric $d_{\mathbb E_n}$ as in Definition~\ref{def:metric2}. With such choice of metric, to verify Conjecture~\ref{conj:metric2}, we still need to check a condition on $\lk(x,\Delta_{S,S'})$ of each vertex $x\in \Delta_{S,S'}$. However, we only need to study loops in the 1-skeleton of $\lk(x,\Delta_{S,S'})$ of length $<2\pi$, which is a much smaller collection of loops. More precisely, we study \emph{edge loops}, which are loops made of a sequence of edges. An \emph{$n$-cycle} in $\lk(x,\Delta_{S,S'})$, is an edge loop made of $n$ edges. 

Still, the collection of all $n$-cycles in $\lk(x,\Delta_{S,S'})$ with length $<2\pi$ is larger than one might have expected. Due to some unusual features of high dimensional round spheres, there are edge loops of length $<2\pi$ in $\lk(x,\Delta_{S,S'})$ such that the number of edges in these loops goes to $\infty$ as the dimension of  $\lk(x,\Delta_{S,S'})$ goes to infinity. However, the good news is that in high dimensions, we only need to study certain $n$-cycles with $n\le 6$ in $\lk(x,\Delta_{S,S'})$. In the next two sections (Section~\ref{subsec:contractible} and Section~\ref{subsec:subdiv}), we explain what are exactly the link conditions one needs to check to prove Conjecture~\ref{conj:metric2}.

For the remaining types of $S'$, we have the following conjecture.
\begin{conj}
	\label{conj:exceptional}
	Let $A_S$ be an Artin group and suppose $S$ contains a subset $S'$ which is almost spherical, and its type is contained in $\{\widetilde F_4,\widetilde E_6,\widetilde E_7,\widetilde E_8,[3,5,3],[5,3,3,3]\}$. Then the associated Coxeter group $W_S$ acts properly and cocompactly by isometries on $X$, with $X=\mathbb E^n$ or $\mathbb H^n$. Let $E_{S'}$ be a fundamental domain of $W_S\curvearrowright X$ equipped with the metric inherited from the standard metric on $X$.
	We metrize the fundamental domain $K_{S'}$ of $\Delta_{S,S'}$ using the metric on $E_{S'}$ as in Definition~\ref{def:metric2}, which gives a piecewise Euclidean or piecewise hyperbolic metric on $\Delta_{S,S'}$.
	
	Then $\Delta_{S,S'}$ equipped with such metric is CAT$(0)$, hence contractible.
\end{conj}

By a combination of \cite[Theorem II.4.1 and Theorem II.5.5]{BridsonHaefliger1999} and Lemma~\ref{lem:sc} in Section~\ref{subsec:subdiv}, verifying Conjecture~\ref{conj:exceptional} reduces to verifying for each vertex $x\in \Delta_{S,S'}$, $\lk(x,\Delta_{S,S'})$ equipped with the induced piecewise spherical metric is CAT$(1)$. Note that in the context of Conjecture~\ref{conj:exceptional}, the dimension of $\lk(x,\Delta_{S,S'})$ is $\le 7$, hence it avoids difficulties appearing in very high dimensions. Though similar to the previous discussion, we firmly believe there is a much easier condition to check on $\lk(x,\Delta_{S,S'})$ compared to CAT$(1)$ which would lead to the contractibility of $\Delta_{S,S'}$.

\subsection{Two contractibility criterion for simplicial complexes}
\label{subsec:contractible}
In this subsection we recall two contractibility criteria for simplicial complexes by Haettel, and explain the relevance to Conjecture~\ref{conj:contractible}.
We start by recalling some terminology about posets.
Let $P$ be a poset, i.e. a partially ordered set.
Let $S\subset P$. An \emph{upper bound} (resp. lower bound) for $S$ is an element $x\in P$ such that $s\le x$ (resp. $s\ge x$) for any $s\in S$. The \emph{join} of $S$ is an upper bound $x$ of $S$ such that $x\le y$ for any other upper bound $y$ of $S$. The \emph{meet} of $S$ is a lower bound of $x$ of $S$ such that $x\ge y$ for any other lower bound $y$ of $S$. We will write $x\vee y$ for the join of two elements $x$ and $y$, and $x\wedge y$ for the meet of two elements (if the join or the meet exists). A poset is \emph{bounded} if it has a maximal element and a minimal element. We say $P$ is \emph{lattice} if $P$ is a poset and any two elements in $P$ have a join and have a meet. 

A chain in $P$ is any totally ordered
subset, subsets of chains are subchains and a maximal chain is one that is not a proper subchain of any other chain. A poset has rank $n$ if it is bounded, every chain is
a subchain of a maximal chain and all maximal chains have length $n$. For $a,b\in P$ with $a\le b$, the \emph{interval}  between $a$ and $b$, denoted by $[a,b]$, is the collection of all elements $x$ of $P$ such that $a\le x$ and $x\le b$. The poset $P$ is \emph{graded} if every interval in $P$ has a rank.

Recall that a \emph{bowtie} in a poset $P$ is a subset $\{x_1,x_2,y_1,y_2\}$ such that $y_i$ covers $x_j$ for $1\le i,j\le 2$. The name comes from that if we draw $y_1,y_2$ above $x_1,x_2$ in the Hasse diagram, then we obtain a bowtie-shaped configuration.

\begin{definition}
	Let $P$ be a poset. We say $P$ is \emph{bowtie free} if any subset $\{x_1,x_2,y_1,y_2\}\subset P$ made of mutually distinct elements with $x_i<y_j$ for $i,j\in\{1,2\}$, there exists $z\in P$ such that $x_i\le z\le y_j$ for any $i,j\in\{1,2\}$.
\end{definition}

The interest in the bowtie-free condition lies in the following observation.

\begin{lem}\cite[Proposition 1.5]{brady2010braids}
	\label{lem:posets}
	If $P$ is a bowtie free graded poset, then any pair of elements in $P$ with a lower bound have a join, and any pair of elements in $P$ with an upper bound have a meet.
	
	Let $P$ be a bounded graded poset. Then $P$ is lattice if it is bowtie free. 
\end{lem}

\begin{definition}
	\label{def:flag}
	A poset $P$ is \emph{upward flag} if any three pairwise upper bounded elements have an upper bound. A poset is \emph{downward flag} if any three pairwise lower bounded elements have a lower bound. A poset is \emph{flag} if it is both upward flag and downward flag.
\end{definition}

We say $x\in P$ is a \emph{maximal element} if there does not exist $y\in P$ such that $x<y$.

\begin{definition}
	\label{def:weak flag}
	A poset $P$ is \emph{weakly upward flag} if for each triple $\{x,y,z\}$ satisfying the following two properties have a common upper bound:
	\begin{enumerate}
		\item $x,y,z$ are not \emph{maximal} element in $P$;
		\item each pair in $\{x,y,z\}$ have a upper bound in $P$ which is not maximal.
	\end{enumerate}
	Similarly, we can define weakly downward flag and weakly flag for posets.
\end{definition}

Let $S=\{s_1,s_2,\ldots,s_n\}$. A simplicial complex $X$ is of \emph{type $S$} if all the maximal simplices of $X$ has dimension $n-1$ and there is a type function $\type$ from the vertex set of $X$ to $\{\hat s_1,\hat s_2,\ldots,\hat s_n\}$ such that $\type(x)\neq\type(y)$ whenever $x$ and $y$ are adjacent vertices of $X$. This labeling induces a bijection between $\{\hat s_1,\hat s_2,\ldots,\hat s_n\}$ and the vertex set of each maximal simplex of $X$. 

Note that if $A_{S'}$ is an Artin group, and $S$ is a subset of the set of generators $S'$, then the relative Artin complex $\Delta_{S',S}$ is a simplicial complex of type $S$. We will be interested in more general simplicial complexes of type $S$, for some $S$ not necessarily made of generators of some Artin groups.

\begin{definition}
	\label{def:order}
	Let $X$ be a simplicial complex of type $S$.
	We put a total order on $S$, and define a relation $<$ on the vertex set $V$ of $X$ induced by this total order as follows: $x<y$ if $x$ and $y$ are adjacent, and $\type(x)<\type(y)$.
\end{definition}

As all maximal simplices of $X$ have the same dimension, we know the following:
\begin{itemize}
	\item for each $x\in V$ of type $\hat s$ such that $s$ is not the smallest element in $S$, there exists $x'\in V$ with $x'<x$;
	\item for each $x\in V$ of type $\hat s$ such that $s$ is not the biggest element in $S$, there exists $x'\in V$ with $x'>x$.
\end{itemize}

Now we discuss a situation when the relation $<$ on $V$ is a poset.

\begin{definition}
	\label{def:admissible}
	An induced subgraph $\Lambda'$ of $\Lambda$ is \emph{admissible} if for any node $x\in \Lambda'$, if $x_1,x_2\in \Lambda'$ are in different connected components of $\Lambda'\setminus\{x\}$, then they are in different components of $\Lambda\setminus\{x\}$.
\end{definition}

\begin{lem}(\cite[Lemma 6.6]{huang2023labeled})
	\label{lem:poset structure}
	Suppose $\Lambda'$ is an admissible linear subgraph of the Coxeter diagram $\Lambda$ of $A_S$ and suppose the consecutive nodes of $\Lambda'$ are $S'=\{s_i\}_{i=1}^n$.  Let $\Delta$ be the $(\Lambda,\Lambda')$-relative Artin complex. Let $V$ be the vertex set of $\Delta$. We order $S'$ such that it is compatible with one of the two linear orders on $\Lambda'$. Then the induced relation $<$ on $V$ is a graded poset.
\end{lem}

Now we are ready to state Haettel's result. Let $S=\{s_1,\ldots,s_n\}$ with a total order $s_1<s_2<\cdots<s_n$. The vertex set of a simplicial complex of type $S$ is endowed with the relation induced from such total order on $S$. The following is a consequence of \cite[Section 4.3, Theorem B]{haettel2022link} and \cite[Theorem 1.15]{haettel2021lattices}.

\begin{thm}	\cite{haettel2021lattices,haettel2022link}
	\label{thm:contractibleII}
	Let $X$ be a simplicial complex of type $S$.  Assume that
	\begin{enumerate}
		\item $X$ is simply connected;
		\item the relation $<$ on the vertex set $V$ of $X$ is a partial order;
		\item for each $x\in V$, let $V_{\ge x}$ be the collection of vertices that is $\ge x$, then $V_{\ge x}$ is bowtie free and upward flag;
		\item for each $x\in V$, let $V_{\le x}$ be the collection of vertices that is $\le x$, then $V_{\le x}$ is bowtie free and downward flag.
	\end{enumerate}
	Then $X$ is contractible, and $X$ admits a piecewise $\ell^\infty$ injective metric.
	
	Moreover, let $Y$ be a graph whose vertex set is the same as the vertex set of $X$, and two vertices $y_1,y_2\in Y$ are adjacent if there exist vertices $z_1\in X$ of type $\hat s_1$ and $z_2\in X$ of type $\hat s_n$ such that $z_1\le y_i\le z_2$ for $i=1,2$. Then $Y$ is a Helly graph.
\end{thm}

The graph $Y$ in the above theorem is called the \emph{thickening} of $X$. Given a simplicial graph $Z$ endowed with the path metric such that each edge has length $1$, a \emph{combinatorial ball} in $Z$ is the collection of vertices in the metric ball $B(x,r)$ of radius $r$ centered at a vertex $x$. The graph $Z$ is \emph{Helly} if whenever a collection of combinatorial balls in $Z$ have non-empty pairwise intersection, then the common intersection of these balls is non-empty. Helly graphs are closely related to the notion of \emph{injective metric spaces}, which are continuous versions of Helly graphs. A geodesic metric space is \emph{injective} if whenever a collection of closed metric balls in the space have non-empty pairwise intersection, then the common intersection of these balls is non-empty.

Now we discuss a variation of Theorem~\ref{thm:contractibleII}.
Put a cyclic order $s_1<s_2<\cdots s_n<s_1$ on $S$. 
For each vertex $x$ in $X$ of type $\hat s_i$, we consider a relation $<_x$ in $\lk(x,X)$ as follows. We identify vertices in $\lk(x,X)$ as vertices in $X$ which are adjacent to $x$. For each $s_i\in S$, this cyclic order induces an order on $S\setminus\{s_i\}$ by declaring $s_{i+1}<\cdots<s_n<s_1<\cdots s_{i-1}$. 
For $y,z\in \lk(x,X)$, define $y<_x z$ if $y$ is adjacent to $z$ and $\type(y)<\type(z)$ in $S\setminus \{s_i\}$. The following is a consequence of \cite[Section 4.2, Theorem A]{haettel2022link}.

\begin{thm}\cite{haettel2021lattices,haettel2022link}
	\label{thm:contractible}
	Let $X$ be a simplicial complex of type $S$, with $S$ being a cyclically ordered set as above. Suppose the following are true.
	\begin{enumerate}
		\item $X$ is simply-connected.
		\item For each vertex $x\in X$, the relation $<_x$ on the vertex set of $\lk(x,X)$ is a partial order.
		\item For each vertex $x\in X$, the set of vertex in $\lk(x,X)$ with this partial order is bowtie free.
	\end{enumerate}
	Then $X$ is contractible. 
	Moreover, if a group $G$ acts on $X$ by type-preserving automorphisms, then $X$ can be equipped with a metric with an $G$-equivariant consistent convex geodesic bicombing $\sigma$ such that each simplex of $X$ is $\sigma$-convex.
\end{thm}

\subsection{Subdivision of some relative Artin complexes}
\label{subsec:subdiv}
In this subsection, we explain the connection between Conjecture~\ref{conj:contractible} and the two contractibility criteria in Section~\ref{subsec:contractible}. This is done by adjusting the procedure of handling Euclidean buildings and certain Artin complexes in Haettel's work \cite{haettel2021lattices} to certain relative Artin complexes.

\begin{definition}
	\label{def:bowtie free}
	Suppose $\Lambda'$ is an admissible linear subgraph of the Coxeter diagram $\Lambda$ of $A_S$ with consecutive nodes of $\Lambda'$ being $\{s_i\}_{i=1}^n$. We say the $(\Lambda,\Lambda')$-relative Artin complex $\Delta_{\Lambda,\Lambda'}$ is \emph{bowtie free} if the poset in Lemma~\ref{lem:poset structure} is bowtie free.  We  $\Delta_{\Lambda,\Lambda'}$ is \emph{flag} or \emph{weakly flag}, if the poset in Lemma~\ref{lem:poset structure} is flag or weakly flag respectively. Note that these definitions do not depend on the choice of the linear order on $\Lambda'$.
\end{definition}

\begin{definition}
	\label{def:subdivision}
	Suppose $A_\Lambda$ is an Artin group whose Coxeter diagram $\Lambda$ contain induced subgraph $\Lambda'$ such that $\Lambda'$ is a copy of Coxeter diagram of type $\widetilde B_n$, though possibly with different edge labeling. Let $\{b_i\}_{i=1}^n$ be nodes of $\Lambda'$ as in Figure~\ref{fig:BD} left.
	
	Let $\Delta=\Delta_{\Lambda,\Lambda'}$ be the associated relative Artin complex. We subdivide each edge of $\Delta$ connecting a vertex of type $\hat b_1$ and a vertex of type $\hat b_2$. We say the middle point of such edge is of type $m$. Cut each top dimensional simplex in $\Delta$ into two simplices along the codimensional 1 simplex spanned by vertices of type $m$ and $\{\hat b_i\}_{i=3}^{n+1}$. This gives a new simplicial complex, which we denoted by $\Delta'$. Define a map $t$ from the vertex set $V\Delta'$ of $\Delta'$ to $\{1,2,\ldots,n,n+1\}$ by sending vertices of type $\hat b_1,\hat b_2$ to $1$, vertices of type $m$ to $2$, vertices of type $\hat b_i$ to $i$ for $i\ge 3$. We will then view $\Delta'$ as a simplicial complex of type $S=\{1,2,\ldots,n+1\}$. We define a relation $<$ on $V\Delta'$ as follows. For $x,y\in V\Delta'$, $x<y$ if $x$ and $y$ are adjacent and $t(x)<t(y)$. The simplicial complex $\Delta'$, together with the relation $<$ on its vertex set, is called the \emph{$(b_1,b_2)$-subdivision of $\Delta_{\Lambda,\Lambda'}$}. 
\end{definition}

\begin{figure}[h]
	\centering
	\includegraphics[scale=1]{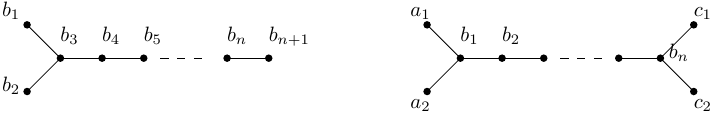}
	\caption{Diagrams of type $\widetilde B$ and $\widetilde D$.}
	\label{fig:BD}
\end{figure}

\begin{lem}
	\label{lem:poset}
	Suppose $\Lambda'$ is an admissible subgraph of $\Lambda$. Then the relation $<$ is a partial order.
\end{lem}

\begin{proof}
	Take a vertex $x\in \Delta'$. Let $\lk_+(x,\Delta')$ (resp. $\lk_-(x,\Delta')$) be the full subcomplex of $\lk(x,\Delta')$ spanned by vertices $y$ of $\Delta'$ with $t(y)>t(x)$ (resp. $t(y)<t(x)$).	
	The lemma follows from the claim that $\lk(x,\Delta')$ is a join of $\lk_+(x,\Delta')$ and $\lk_-(x,\Delta')$. Now we prove this claim.
	It is clear when $t(x)=2$ as in this case as $x$ is the midpoint of an edge of $\Delta$. When $t(x)\ge 3$, the assumption of admissible in this lemma and Lemma~\ref{lem:link} imply the claim.
\end{proof}

We will say the $(b_1,b_2)$-subdivision $\Delta'$ of $\Delta$ is \emph{upward or downward flag or bowtie free}, if $(V\Delta',<)$ is a poset which is upward or downward flat or bowtie free.

\begin{lem}(\cite[Lemma 6.2]{huang2023labeled})
	\label{lem:sc}
	Suppose $|S'|\ge 3$. Then the relative Artin complex $\Delta_{S,S'}$ is simply-connected.
\end{lem} 

\begin{prop}
	\label{prop:ori link0}
	Suppose $A_\Lambda$ is an Artin group whose Coxeter diagram $\Lambda$ contain induced subgraph $\Lambda'$ such that $\Lambda'$ is a copy with Coxeter diagram of type $\widetilde B_n$, though possibly with different edge labeling. Let $\{b_i\}_{i=1}^{n+1}$ be nodes of $\Lambda'$ as in Figure~\ref{fig:BD} left. Suppose $n\ge 3$.
	Suppose $\Lambda'$ is an admissible subgraph of $\Lambda$.
	
	For $i=1,2,n+1$, let $\Lambda_{i}$ (resp. $\Lambda'_{i}$) be the connected component of $\Lambda\setminus\{b_i\}$ (resp. $\Lambda'\setminus\{b_i\}$) that contains $b_3$. 
	Suppose that the following holds:
	\begin{enumerate}
		\item the $(b_1,b_2)$-subdivision of $\Delta_{\Lambda_{n+1},\Lambda'_{n+1}}$ is bowtie free and downward flag;
		\item for $i=1,2$, the vertex set of the relative Artin complex $\Delta_{\Lambda_{i},\Lambda'_{i}}$, endowed with the order induced from $b_i<b_3<b_4<\cdots<b_{n+1}$ is bowtie free and upward flag for $i=1,2$.
	\end{enumerate}
	Then the simplicial complex $\Delta'$, viewed as a simplicial complex of type $S$ with $S=\{1,\ldots,n+1\}$, satisfies all the assumptions of Theorem~\ref{thm:contractibleII}. Hence $\Delta'$ is contractible. Thus $\Delta$ is contractible.
\end{prop}

The proof is very similar to \cite[Section 7]{haettel2021lattices}, we provide details for the convenience of the reader.
\begin{proof}
	Let $P=(V\Delta',\le)$, which is a poset by  Lemma~\ref{lem:poset}. Note that $\Delta'$ is simply connected by Lemma~\ref{lem:sc}. To verify Theorem~\ref{thm:contractibleII} (3), it suffices to consider the case $t(x)=1$. 
	Then the full subcomplex of $\Delta'$ spanned by $P_{>x}$ is $\lk(x,\Delta')$. However, $\lk(x,\Delta')\cong \lk(x,\Delta)\cong \Delta_{\Lambda_i,\Lambda'_i}$ by Lemma~\ref{lem:link}, where $i=1$ or $2$. Moreover, the order of vertices in $\lk(x,\Delta')$ inherited from $P$ and the order of vertices in $\Delta_{\Lambda_i,\Lambda'_i}$ as in Assumption 1 of the proposition, are consistent under the isomorphism. Thus $P_{>x}$ is bowtie free and upward flag by Assumption 2.  
	To verify Theorem~\ref{thm:contractibleII} (4), it suffices to consider the case $t(x)=n+1$. The full subcomplex of $\Delta'$ spanned by $P_{<x}$ is $\lk(x,\Delta')$. By Lemma~\ref{lem:link}, $\lk(x,\Delta')$ is order-preserving isomorphic to the $(b_1,b_2)$-subdivision of $\lk(x,\Delta)\cong \Delta_{\Lambda_{n+1},\Lambda'_{n+1}}$, which finishes the proof by Assumption 1.
\end{proof}

\begin{definition}
	\label{def:subdivisionD}
	Suppose $A_\Lambda$ is an Artin group whose Coxeter diagram $\Lambda$ contain induced subgraph $\Lambda'$ such that $\Lambda'$ is a copy with Coxeter diagram of type $\widetilde D_m$ with $m\ge 4$, though possibly with different edge labeling. Let $a_1,a_2,\{b_i\}_{i=1}^n,c_1,c_2$ be nodes of $\Lambda'$ as in Figure~\ref{fig:BD} right.
	
	Let $\Delta=\Delta_{\Lambda,\Lambda'}$ be the associated relative Artin complex. We subdivide each edge of $\Delta$ connecting a vertex of type $\hat a_1$ (resp. $\hat c_1$) and a vertex of type $\hat a_2$ (resp. $\hat c_2$). We say the middle point of such edge is of type $\hat a$ (resp. $\hat c$).
	Cut each top dimensional simplex in $\Delta$ into four simplices whose vertex set is of type $\{\hat a_i,\hat a,\hat b_1,\ldots,\hat b_n,\hat c,\hat c_j\}_{1\le i,j\le 2}$. This gives a new simplicial complex, which we denoted by $\Delta'$. Define a map $t$ from the vertex set $V\Delta'$ of $\Delta'$ to $\{1,\ldots,n+4\}$ by sending vertices of type $\hat a_i,\hat a,\hat b_1,\ldots,\hat b_n,\hat c,\hat c_j$ to $1,2,3,\ldots,n+2,n+3,n+4$ respectively.  We will then view $\Delta'$ as a simplicial complex of type $S=\{1,\ldots,n+4\}$. We define a relation $<$ on $V\Delta'$ as follows. For $x,y\in V\Delta'$, $x<y$ if $x$ and $y$ are adjacent and $t(x)<t(y)$. Similar to Lemma~\ref{lem:poset}, we know $(V\Delta',\le)$ is a poset, under the additional assumption that $\Lambda'$ is admissible in $\Lambda$.
\end{definition}

\begin{prop}
	\label{prop:ori link2}
	Suppose $A_\Lambda$ is an Artin group whose Coxeter diagram $\Lambda$ contains induced subgraph $\Lambda'$ such that $\Lambda'$ is a copy with Coxeter diagram of type $\widetilde D_m$ with $m\ge 4$, though possibly with different edge labeling. Let $a_1,a_2,\{b_i\}_{i=1}^n,c_1,c_2$ be nodes of $\Lambda'$ as in Figure~\ref{fig:BD} right.	Suppose $\Lambda'$ is an admissible subgraph of $\Lambda$.
	
	Let $\Lambda_{a_i}$ (resp. $\Lambda'_{a_i}$) be the connected component of $\Lambda\setminus\{a_i\}$ (resp. $\Lambda'\setminus\{a_i\}$) that contains $\{b_i\}_{i=1}^n$. Similarly we define $\Lambda_{c_i}$ and $\Lambda'_{c_i}$. 
	Suppose that the following holds:
	\begin{enumerate}
		\item the $(a_1,a_2)$-subdivision of $\Delta_{\Lambda_{c_i},\Lambda'_{c_i}}$ is bowtie free and downward flag for $i=1,2$;
		\item the $(c_1,c_2)$-subdivision of $\Delta_{\Lambda_{a_i},\Lambda'_{a_i}}$ is bowtie free and downward flag for $i=1,2$.
	\end{enumerate}
	Then the simplicial complex $\Delta'$, viewed as a simplicial complex of type $S$ with $S=\{1,\ldots,n+4\}$, satisfies all the assumptions of Theorem~\ref{thm:contractibleII}. Hence $\Delta'$ is contractible. Thus $\Delta$ is contractible.
\end{prop}

The proof is similar to that of Proposition~\ref{prop:ori link0}, and is left to the reader.

\begin{cor}
	\label{cor:link 6cycle}
	Suppose $\Lambda$ is a connected Coxeter diagram, and $\Lambda'$ is an admissible subgraph of $\Lambda$ in the sense of Definition~\ref{def:admissible}. Suppose $\Lambda'$ is almost spherical with vertex set $S'$ and $\Lambda$ has vertex set $S$. Suppose $\Lambda'$ is not of type $\{\widetilde F_4,\widetilde E_6,\widetilde E_7,\widetilde E_8,[3,5,3],[5,3,3,3]\}$.  Then there is a criterion only involving cycles in the 1-skeleton in the link of each vertex of $\Delta_{S,S'}$ such that as long as such link criterion is satisfied, then Conjecture~\ref{conj:contractible} holds true, in particular $\Delta_{S,S'}$ is contractible. More precisely,
	\begin{enumerate}
		\item if $S'$ dominates type $\widetilde A_n$ (in which case the Coxeter diagram of $S'$ must be a cycle, hence elements in $S'$ has a natural cyclic order), then we need to check $\Delta_{S,S'}$, viewed as a simplicial complex of type $S'$, satisfies the assumption 3 of Theorem~\ref{thm:contractible};
		\item if $S'$ dominates type $\widetilde C_n$ (in which case the Coxeter diagram of $S'$ is a line, hence $S'$ has two linear orders), then we need to check if $S'$ is equipped with one of these linear orders, then $\Delta_{S,S'}$ satisfies the assumptions 3 and 4 of Theorem~\ref{thm:contractibleII};
		\item if $S'$ dominates type $\widetilde B_n$ (in which case the Coxeter diagram of $S'$ is isomorphic to Figure~\ref{fig:BD} left), then we need to check assumptions of Proposition~\ref{prop:ori link0} are satisfied;
		\item if $S'$ dominates type $\widetilde D_n$ (in which case the Coxeter diagram of $S'$ is isomorphic to Figure~\ref{fig:BD} right), then we need to check assumptions of Proposition~\ref{prop:ori link2} are satisfied;
		\item if $S'=\{a,b,c\}$ dominates type $\widetilde G_2$ with $m_{ab}\ge 3$ and $m_{bc}\ge 6$, then we need to check $\lk(x,\Delta_{S,S'})$ is a graph with girth $\ge 6$ whenever $x$ is of type $\hat c$ and $\lk(x,\Delta_{S,S'})$ is a graph with girth $\ge 12$ whenever $x$ is of type $\hat a$.
	\end{enumerate}
\end{cor}

\begin{proof}
	Assertions 1 and 2 follow from Lemma~\ref{lem:sc}, Lemma~\ref{lem:poset}, Theorem~\ref{thm:contractible} and Theorem~\ref{thm:contractibleII}. Assertions 3 and 4 follows from Proposition~\ref{prop:ori link0} and Proposition~\ref{prop:ori link2}. For Assertion 5, we metrize triangles with flat metric, with angle $\pi/2$ at vertices of type $\hat b$, angle $\pi/6$ at vertices of type $\hat a$ and angle $\pi/3$ at vertices of type $\hat c$. The assumptions imply $\Delta_{S,S'}$ is locally CAT$(0)$, hence we are done by Lemma~\ref{lem:sc}.
\end{proof}

\begin{remark}
	When $\dim(\Delta_{S,S'})>2$, the link conditions in Corollary~\ref{cor:link 6cycle} amounts to understand certain $n$-cycles in $\Delta_{S,S'}$ with $n\le 6$. More precisely, each bowtie in $\lk(x,\Delta_{S,S'})$ gives rise to a 4-cycle in $\lk(x,\Delta_{S,S'})$. Checking the bowtie free amounts to check if such 4-cycle is embedded and induced, then it has a center in the sense explained in Section~\ref{subsec:fill 6-cycle}. See Lemma~\ref{lem:4wheel} for a precise statement. In particular, each such $4$-cycle can be filled by a disk in the 2-skeleton made of four 2-simplices. To check a flagness condition, say we have vertices $x_1,x_2,x_3$ in $\lk(x,\Delta_{S,S'})$ which are pairwisely upper bounded. Let $y_i$ be an upper bound for $x_i$ and $x_{i+1}$ for $i\in \mathbb Z/3\mathbb Z$. Then $x_1y_1x_2y_2x_3y_3$ forms a 6-cycle in $\lk(x,\Delta_{S,S'})$. Checking upward flagness amounts to find a vertex $z\in \lk(x,\Delta_{S,S'})$ adjacent to each of $\{x_1,x_2,x_3\}$ which is a common upper bound of them. Note that the new vertex $z$ breaks down the 6-cycle into three 4-cycles, namely $zx_1y_1x_2$, $zx_2y_2x_3$ and $zx_3y_3x_1$ such that each of them can be filled in as in the previous paragraph if we know the bowtie free condition. Thus checking the flagness condition is reduced to showing certain 6-cycles can be filled in the 2-skeleton by disks with specific combinatorial types.
\end{remark}

\begin{definition}
	\label{def:labeled 4-wheel}
	Let $\Lambda$ be a Coxeter diagram which is a tree, with its set of nodes $S$. Let $Z$ be a simplicial complex of type $S$.
	Let $X$ be the 1-skeleton of $Z$ with its vertex types as explained above. We say $Z$ satisfies the \emph{labeled 4-wheel condition} if for any induced 4-cycle in $X$ with consecutive vertices being $\{x_i\}_{i=1}^4$ and their types being $\{\hat s_i\}_{i=1}^4$, there exists a vertex $x\in X$ adjacent to each of $x_i$ such that the type $\hat s$ of $x$ satisfies that the node $s$ is in the smallest subtree of $\Lambda'$ containing all of $\{s_i\}_{i=1}^4$.
\end{definition}

The following is a consequence of \cite[Lemma 6.14 and Proposition 6.17]{huang2023labeled}. 
\begin{lem}
	\label{lem:4wheel}
	Suppose $\Lambda'$ is an admissible tree subgraph of $\Lambda$. Then the relative Artin complex $\Delta_{\Lambda,\Lambda'}$ satisfies the labeled 4-wheel condition if and only if for all maximal linear subgraph $\Lambda''\subset \Lambda'$, $\Delta_{\Lambda,\Lambda''}$ is bowtie free.
\end{lem}

Note that Lemma~\ref{cor:link 6cycle} requires $\Lambda'$ to be an admissible subgraph of $\Lambda$. If $\Lambda'$ is not an admissible subgraph of $\Lambda$, then the link criterion in Lemma~\ref{cor:link 6cycle} does not apply directly. For example, in Case 2 of Lemma~\ref{cor:link 6cycle}, if $\Lambda'$ is not admissible, then $\Delta_{S,S'}$ with the relation induced from the linear order from $S'$ is not a poset. So Assumption 2 of Theorem~\ref{thm:contractibleII} is not satisfied. Similar problems happen with the other cases. This leads to us to formulate the following somewhat less precise conjecture.
\begin{conj}
	Under the same assumption of Lemma~\ref{cor:link 6cycle}, even if $\Lambda'$ is not admissible in $\Lambda$, there still exists a criterion only involving cycles in the 1-skeleton in the link of each vertex of $\Delta_{S,S'}$ such that as long as such link criterion is satisfied, then $\Delta_{S,S'}$ is contractible.
\end{conj}

For the remaining types of $\Lambda'$, we ask the following.
\begin{que}
	Suppose $\Lambda$ is a connected Coxeter diagram, and $\Lambda'$ is an induced subgraph of $\Lambda$ in the sense of Definition~\ref{def:admissible}. Suppose the type of $\Lambda'$ is one of  $\{\widetilde F_4,\widetilde E_6,\widetilde E_7,\widetilde E_8$, $[3,5,3],[5,3,3,3]\}$. 
	Is it true that there is a criterion only involving cycles in the 1-skeleton in the link of each vertex of $\Delta_{S,S'}$ such that as long as such link criterion is satisfied, then $\Delta_{S,S'}$ is contractible?
\end{que}

We firmly believe the answer to this question is yes - this is ongoing work of P. Przytycki and the author, and another ongoing work of K. Goldman and P. Przytycki. This question is of independent interests outside $K(\pi,1)$-conjecture, as it would lead to new form of non-positively curvature which can be possibly applied to other classes of groups or spaces.

\subsection{Link conditions at the base case}
\label{sec:cycle}
Let $S'\subset S$ as before with $S'$ almost spherical. We are finally in the position to discuss checking link conditions on  $\lk(x,\Delta_{S,S'})$, which is the core of our approach to $K(\pi,1)$-conjecture. 
Our plan is to reduce the study of cycles in general relative Artin complexes to the \emph{base case} of studying cycles in Artin complexes associated with spherical Artin groups (i.e. cycles in spherical Deligne complexes), so eventually proving $K(\pi,1)$-conjecture for general Artin groups can be reduced to specific properties of spherical Artin groups.

For a vertex $x$ of type $\hat s$ in $\Delta_{S,S'}$, by Lemma~\ref{lem:link}, $\lk(x,\Delta_{S,S'})\cong \Delta_{T,T'}$ with $T=S\setminus\{s\}$ and $T'=S'\setminus\{s\}$. If $A_T$ is  spherical, then we are already in the base case. If $A_T$ is not spherical, then we can find $T''\subset T$ such that $A_{T''}$ is almost spherical. If we understand certain collections of cycles in the vertex links of $\Delta_{T,T''}$, then we can conclude that $\Delta_{T,T''}$ is non-positively curved in an appropriate sense, using some of the criterion in Section~\ref{subsec:contractible} and Section~\ref{subsec:subdiv}. While $T''$ might not equal $T'$, $A_{T}$ acts on both $A_{T,T''}$ and $A_{T,T'}$, this allows us to encode an $n$-cycle in $\Delta_{T,T'}$ as a collection of convex subspaces of $\Delta_{T,T''}$, and use the non-positive curvature of $\Delta_{T,T''}$ to analyze the configuration of these subspaces. This gives a way to use our knowledge of cycles in the vertex links of $\Delta_{T,T''}$ associated to a smaller Artin group $A_T$ to understand cycles in the vertex links of $\Delta_{S,S'}$ associated with a larger Artin group $A_S$. We can keep doing this until $A_T$ is spherical. We refer to \cite[Section 9.1]{huang2023labeled} for an example demonstrating this strategy.

Now we look at the base case $A_T=A_{S\setminus\{s\}}$ being spherical in greater detail.
By translating the meaning of the link conditions in Corollary~\ref{cor:link 6cycle} in the base case, we are led to the following statements about spherical Artin groups.

\begin{conj}
	\label{conj:compareB}
	Suppose $A_S$ is an irreducible spherical Artin group with Coxeter diagram $\Lambda$. Let $\Lambda'$ be a linear subdiagram of $\Lambda$ with consecutive vertices $\{s_i\}_{i=1}^n$ such that the edge between $s_{n-1}$ and $s_n$ has label $\ge 4$.  Then the vertex set of $\Delta_{\Lambda,\Lambda'}$ equipped with the relation induced from $s_1<s_2<\cdots<s_n$ (as in Definition~\ref{def:order}) is a bowtie free and upward flag poset.
\end{conj}

\begin{conj}
	\label{conj:compareD}
	Suppose $A_S$ is an irreducible spherical Artin group. Let $S'\subset S$ such that $A_{S'}$ has Coxeter diagram isomorphic to the type $D_n$ Coxeter diagram for $n\ge 3$ (the isomorphism does not need to preserve edge labels). Let $\{b_i\}_{i=1}^{n+1}$ be vertices in $S'$ as in Figure~\ref{fig:BD} left. Then the $(b_1,b_2)$-subdivision of $\Delta_{S,S'}$ (in the sense of Definition~\ref{def:subdivision}) is a bowtie free and downward flag poset.
\end{conj}

In Conjecture~\ref{conj:compareD}, we allow the case when the Dynkind diagram $A_{S'}$ is isomorphic to the type $D_3$ diagram (though edge labels might not be preserved). While $D_3$ diagram is the same as $A_3$ diagram, but we are considering a subdivision of $\Delta_{S,S'}$ by viewing the Coxeter diagram as $D_3$ rather than $A_3$, as explained in Definition~\ref{def:subdivision}.

Now we summarize previous results on these conjectures. The bowtie free part of both conjectures is already known and is a consequence of Theorem~\ref{thm:bowtie free} and Theorem~\ref{thm:4 wheel}.

\begin{thm}(\cite[Theorem 8.1]{huang2023labeled})
	\label{thm:bowtie free}
	Suppose $A_\Lambda$ is an irreducible spherical Artin group. Then for any linear subgraph $\Lambda'\subset\Lambda$, $\Delta_{\Lambda,\Lambda'}$ is bowtie free. 
\end{thm}

This theorem is a consequence of Theorem~\ref{thm:4 wheel} below and Lemma~\ref{lem:4wheel}.

\begin{thm}(\cite[Proposition 2.8]{huang2023labeled})
	\label{thm:4 wheel}
	Suppose $A_S$ is an irreducible spherical Artin group. Then $\Delta_S$ satisfies the labeled 4-wheel condition.
\end{thm}

For the upward flag part of Conjecture~\ref{conj:compareB}, the case when $A_S=A_{S'}$ is type $B_n$ follows from the work of Haettel \cite[Proposition 6.6]{haettel2021lattices}.
\begin{thm}
	\label{thm:triple}
	Let $A_S$ be the Artin group of type $B_n$. Let $S=\{s_1,s_2,\ldots,s_n\}$ such that $s_i$ and $s_{i+1}$ are adjacent in the Coxeter diagram and $m_{s_{n-1},s_n}=4$. We put a total order on $S$ by $s_1<s_2<\ldots<s_n$. Let the vertex set $V$ of $\Delta_S$ be endowed with the relation $<$ induced from this partial order on $S$. Then $(V,\le)$ is an upward flag poset.
\end{thm}

For Conjecture~\ref{conj:compareD}, the downward flag part is true when $A_S=A_{S'}$ is of type $D_4$, see \cite[Corollary 7.7]{huang2024Dn}. The downward flag part is also known when $A_S$ is of type $D_n$, and $S'=\{b_1,b_2,b_3\}$ in Figure~\ref{fig:BD}, see Theorem~\ref{thm:weakflagD} below and Lemma~\ref{lem:weakly flag equivalent}.

\begin{thm} (\cite[Theorem 1.2]{huang2024Dn})
	\label{thm:weakflagD}
	Let $\Lambda$ be the Coxeter diagram of type $D_n$. Let $\Lambda'\subset\Lambda$ be the subgraph spanned by $\{\delta_1,\delta_2,\delta_3\}$ in Figure~\ref{fig:ad}. Then $\Delta_{\Lambda,\Lambda'}$ is weakly flag.
\end{thm}

\begin{figure}[h]
	\centering
	\includegraphics[scale=1]{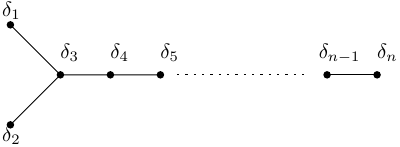}
	\caption{Coxeter diagram of type $D_n$.}
	\label{fig:ad}
\end{figure}

\section{One side flagness in type $F_4$}
\label{sec:F4}

The main goal of this section is Proposition~\ref{prop:F4}. The proof is this proposition is inspired by  \cite[Proposition 6.6]{haettel2021lattices}.

\begin{lem}
	\label{lem:fixed point}
Suppose $A_\Lambda$ is a spherical Artin group with Coxeter diagram $\Lambda$ and suppose $\alpha:A_\Lambda\to A_\Lambda$ be an automorphism induced by an automorphism of $\Lambda$. Let $C$ be a left coset of a standard parabolic subgroup of $A_\Lambda$. If $\alpha(C)=C$, then $\alpha$ fixes an element in $C$.
\end{lem}

\begin{proof}
Let $A^+$ be the positive monoid in $A_\Lambda$. Let $\preccurlyeq_\ell$ the prefix order on $A^+$ and let $\preccurlyeq_r$ be the suffix order on $A^+$. We define a partial order $\preccurlyeq$ on $A^+\times A^+$ by $(a_1,b_1)\preccurlyeq (a_2,b_2)$ if $a_1\preccurlyeq_\ell a_2$ and $a_1\neq a_2$ or $a_1=a_2$ and $b_1\preccurlyeq_r b_2$. Recall that each element $a\in A_\Lambda$ be written uniquely as $bc^{-1}$ with $b,c\in A^+$ and $b$ and $c$ do not have non-trivial common suffix in $A^+$ \cite[Section 2]{MR1314589}. This gives a (set theoretical) embedding $A\to A^+\times A^+$, hence $A$ inherits a partial order $\preccurlyeq_R$ from $A^+\times A^+$. As $\alpha(A^+)=A^+$, $\alpha$ respects the partial orders $\preccurlyeq_\ell$ and $\preccurlyeq_r$ on $A^+$. Hence $\alpha$ induces an automorphism of the poset $(A,\preccurlyeq_R)$. By \cite[Theorem 1]{altobelli1998word}, there is an element $c\in C$ such that $c\preccurlyeq_R c'$ for any $c'\in C$. Such $c$ is unique in $C$ by definition. Thus $\alpha(c)=c$, as desired.
\end{proof}

\begin{prop}
	\label{prop:F4}
Let $S=\{s_1,s_2,s_3,s_4\}$ be consecutive nodes in the Coxeter diagram $\Lambda$ of the Artin group of type $F_4$. Let $S'=\{s_1,s_2,s_3\}$. We consider the relative Artin complex $\Delta_{S,S'}$, whose vertex set is endowed with the order inherited from $s_1<s_2<s_3$. Then the vertex set of $\Delta_{S,S'}$ forms a upward flag poset.
\end{prop}

\begin{proof}
Let $\Lambda_1$ be the Coxeter diagram of type $E_6$. Let $\{t_i\}_{i=1}^5$ be the consecutive nodes of the linear subgraph $\Lambda'_1$ of length $4$ in $\Lambda_1$. Let $t$ be the node in $\Lambda\setminus \Lambda_1$. Consider the homomorphism $\phi: A_\Lambda\to A_{\Lambda_1}$ by sending $s_i$ to $t_it_{6-i}$ for $i=1,2$, $s_3$ to $t_3$ and $s_4$ to $t$. By \cite{crisp2000symmetrical},  $\phi$ is an injective homomorphism, whose image is the fixed subgroup of the automorphism $\sigma$ of $A_{\Lambda_1}$ such that $\sigma(t_i)=t_{6-i}$ for $1\le i\le 5$ and $\sigma(t)=t$. We define $P_i=A_{\hat s_i}$, the subgroup of $A$ generated by all generators except $s_i$, and $Q_i=A_{\hat t_i}$.

Let $\Delta=\Delta_{S,S'}$ and $\Delta_1=\Delta_{\Lambda_1,\Lambda'_1}$. Note that $\sigma$ induces an automorphism $\bar \si$ of $\Delta_1$ by sending $gQ_i$ to $\sigma(g)Q_{6-i}$. The map $\phi$ induced a map $\Delta^{(0)}\to \Delta^{(0)}_1$, by sending a vertex associated with $gP_i$ to a vertex associated with $\phi(g)Q_i$. 
As $\phi(gP_i)\subset \phi(g)Q_i$, if two cosets associated with vertices of $\Delta$ have nonempty intersection, then their associated cosets through $\phi$ in $A_{\Lambda_1}$ have nonempty intersection. Thus $\phi$ respects the order of vertices in $\Delta$ and $\Delta_1$, hence extends to a simplicial map $\bar\phi:\Delta\to \Delta_1$. For any vertex $\bar x$ in the image of $\bar\phi$, we have $\bar x\le\bar\si(\bar x)$, as $\si$ is identity on the image of $\phi$.

Given three vertices $x_1,x_2,x_3$ which are pairwise upper bounded. Let $\bar x_i=\bar\phi(x_i)$. We claim $\bar x_i\le \bar \sigma(\bar x_j)$ for $1\le i,j\le 3$. Indeed, suppose $x_{ij}\in \Delta$ is an upper bound for $x_i$ and $x_j$. Then $\bar x_i\le \bar x_{ij}\le \bar\sigma (\bar x_{ij})\le \bar\sigma (\bar x_j)$ as $\bar\sigma$ is order-reversing. Let $\bar x$ be the join of $\bar x_1,\bar x_2,\bar x_3$ - the existence of $\bar x$ follows from Theorem~\ref{thm:bowtie free} and Lemma~\ref{lem:posets}. Then $\bar\sigma(\bar x)$ is the meet of $\bar\si(\bar x_1),\bar\si(\bar x_1),\bar\si(\bar x_1)$. Note that $\bar x\le \bar\si(\bar x)$. Suppose $\bar x=\gamma Q_i$. Then $\bar \si(\bar x)=\sigma(\gamma) Q_{6-i}$. As $\gamma Q_i\le \sigma(\gamma) Q_{6-i}$, we know $i\le 6-i$, hence $i\le 3$. Let $C=\gamma Q_i\cap \sigma(\gamma) Q_{6-i}\neq \emptyset$. As $\sigma$ is an involution, $\sigma(C)=C$. As intersection of standard parabolic subgroups is a standard parabolic subgroup \cite{lek}, we know $C$ is a left coset of a standard parabolic subgroup. By Lemma~\ref{lem:fixed point}, $\sigma(c)=c$ for an element $c\in C$. By replacing $\gamma$ by $c$, we can assume $\sigma(\gamma)=\gamma$. Thus there exists $x\in \Delta$ such that $\bar x=\bar\phi(x)$. As $\bar\phi$ is an order-preserving embedding, we know $x$ is a common upper bound of $x_1,x_2,x_3$.
\end{proof}

\section{Complexes for hyperplane arrangements}
\label{sec:prelim1}
We recall the definition of several cell-complexes associated with a given hyperplane arrangement, and mention some of their properties, as a preparation for later sections.
\subsection{Real arrangements and their dual polyhedrons}\label{subsec:zonotope}
Recall that a \emph{hyperplane arrangement} in a real vector space $\mathbb R^n$ is a locally finite collection $\mathcal A$ of affine hyperplanes in $\mathbb R^n$. 
Let $\cq(\ca)$ be the collection of affine subspaces coming from intersections of elements in $\ca$ (here $\mathbb R^n$ itself is considered an element of $\cq(\ca)$ as it comes from the empty intersection). Each point $x$ in $\mathbb R^n$ is contained a unique element in $\cq(\ca)$ which is minimal with respect to containment. This element is called the \emph{support} of $x$. A \emph{fan} of $\ca$ is a maximal connected subset of $\mathbb R^n$ which is made of points with the same support. Each fan is convex and polyhedral. Denote that collection of all fans of $\ca$ by $\fan(\ca)$.
Note that $\mathbb R^n$ is a disjoint union of all elements in $\fan(\ca)$.
For each $U\in \fan(\ca)$, choose a point $x_U$ in the relative interior of $U$. The partial order on $\fan(\ca)$ is defined by $U_1<U_2$ if $U_1$ is contained in the closure of $U_2$ and in this case, we also write $x_{U_1}<x_{U_2}$. Let $b\Si_{\ca}$ be a simplicial complex whose vertices are $\{x_U\}_{U\in\fan(\ca)}$, and whose simplices correspond to chains of form $x_{U_1}<x_{U_2}<\cdots<x_{U_k}$. There is a piecewise linear embedding of $b\Si_{\ca}$ as a subset of  $\mathbb R^n$. Now we assemble simplicies of $b\Si_{\ca}$ to form another cell complex, which we denote by $\Si_{\ca}$. The closed cells of $\Si_{\ca}$ are in one-to-one correspondence with vertices of $b\Si_{\ca}$. We identify the face of $\Si_{\ca}$ associated with vertex $x_U\in b\Si_{\ca}$ with the union of all simplices of $b\Si_{\ca}$ corresponding to chains whose smallest element is $x_U$.  In this way each vertex of $b\Si_{\ca}$ can also be regarded as the barycenter of a face of $\Si_{\ca}$.
For $B\in \cq(\ca)$, a face $F$ of $\Si_{\ca}$ is \emph{dual} to $B$ if $F\cap B=\{b_F\}$, where $b_F$ denotes the barycenter of $F$.

The 1-skeleton of $\Si_{\ca}$ is endowed with a path metric $d$ such that each edge has length 1. Given $x,y\in \vertex \Si_{\ca}$ it turns out that $d(x,y)$  is the number of hyperplanes separating $x$ and $y$ (cf.\ \cite[Lemma 1.3]{deligne}).  

\begin{lem}\cite[Lemma 3]{s87}
	\label{lem:gate}
	Let $x$ be a vertex in $\Si_{\ca}$ and $F$ be a face of $\Si_{\ca}$. Then there exists a unique vertex $x_F\in F$ such that $d(x,x_F)\le d(x,y)$ for any vertex $y\in F$. The vertex $x_F$ is called the \emph{projection} of $x$ to $F$, and is denoted $\prj_F(x)$. 
\end{lem}

The proof of the following lemma is left to the reader.

\begin{lem}
	Let $\vertex F$ be the vertex set of a face $F$ of $\Si_\ca$. Let $E$ be another face of $\Si_\ca$. Then $\prj_E(\vertex F)=\vertex E'$ for some face $E'\subset E$. In this case we write $E'=\prj_E(F)$.
\end{lem}

\begin{definition}
	\label{def:projection1}
	Let $F$ be a fact of $\Si_\ca$. Lemma~\ref{lem:gate} gives a map $\pi:\vertex\Si_\ca\to\vertex F$ which extends to a retraction $\Pi_F:\Si_\ca\to F$ as follows. Note that for each face $E$ of $\Si_\ca$, $\pi(\vertex E)$ is the vertex set of a face $E'\subset F$. Then we extends $\pi$ to a map $\pi'$ from the vertex set of $b\Si_\ca$ to the vertex set of $bF$, by sending the barycenter of $E$ to the barycenter of $E'$. As $\pi'$ map vertices in a simplex to vertices in a simplex, it extends linearly to a map $\Pi_F:b\Si_\ca\cong \Si_\ca\to bF\cong F$.
\end{definition}

\subsection{Salvetti complex}
Let $\mathcal A,b\Sigma_{\ca},\Si_\ca$ be as before. Let $\mathcal P$ be the poset of faces of $\Si_\ca$ (under containment), and let $V$ be the vertex set of $\Si_\ca$. We now define the \emph{Salvetti complex} $\widehat\Si_\ca$ associated with $\mathcal A$, as follows.
Consider the set of pairs $(F,v)\in \cp \times V$.  Define  an equivalence relation $\sim$ on this set by $$(F,v)\sim (F,v') \iff F=F' \text{\ and\ } \prj_F(v') = \prj_F(v).$$
Denote the equivalence class of $(F,v')$ by $[F,v']$ and let $\ce(\ca)$ be  the set of equivalence classes.   Note  that each equivalence class $[F,v']$ contains a unique representative of the form $(F,v)$, with $v\in \vertex F$.   In  \cite{s87} the \emph{Salvetti complex} $\widehat\Si_\ca$ of $\ca$ is defined as the regular CW complex given by taking  $\Si_\ca\times V$ (i.e., a disjoint union of copies of $\Si_\ca$) and then identifying faces $F\times v$ and $F\times v'$ whenever $[F,v]=[F,v']$, i.e.,
\begin{equation}
	\widehat\Si_\ca=( \Si_\ca\times V)/ \sim \ .
\end{equation}
For example, for each edge $F$ of $\Si_\ca$ with endpoints $v_0$ and $v_1$, we get two $1$-cells $[F,v_0]$ and $[F,v_1]$ of $\widehat\Si_\ca$ glued together along their endpoints $[v_0,v_0]$ and $[v_1,v_1]$.  So, the $0$-skeleton of $\widehat\Si_\ca$ is equal to the $0$-skeleton of $\widehat\Si_\ca$ while its $1$-skeleton is formed from the $1$-skeleton of $\widehat\Si_\ca$ by doubling each edge.  
There is a natural map $p:\widehat\Si_\ca\to\Si_\ca$ defined by ignoring the second coordinate. 

For each subcomplex $Y$ of $\Si_\ca$, we write $\widehat Y=p^{-1}(Y)$ and call $\widehat Y$ the subcomplex of $\od_\ca$ associated with $Y$.
A \emph{standard subcomplex} of $\widehat\Si_\ca$ is a subcomplex of $\od_\ca$ associated with a face of $\Si_\ca$. In other words, if $F\subset \Si_\ca$ is a face, then $\widehat F$ is the union of faces of form $F\times v$ in $\widehat \Si_\ca$ with $v$ ranging over vertices in $\Si_\ca$.

\begin{lem}
	\label{lem:compactible}
	Take faces $E$ and $F$ of $\Si_\ca$.
	If $[E,v_1]=[E,v_2]$, then $[\prj_F(E),v_1]=[\prj_F(E),v_2]$.
\end{lem}

\begin{proof}
	Note that $[E,v_1]=[E,v_2]$ if and only if for each hyperplane $H\in\ca$ with $H\cap E\neq\emptyset$, we know $v_1$ and $v_2$ are in the same side of $H$. Thus for each hyperplane $H\in\ca$ dual to $\prj_F(E)$, $v_1$ and $v_2$ are in the same side of $H$. Now the lemma follows.
\end{proof}

The following is a construction originated in \cite{godelle2012k}.
\begin{definition}
	\label{def:retraction}
	Let $F$ be a face in $\Si_\ca$. Then there is a retraction map $\Pi_{\widehat F}:\widehat\Si_\ca\to \widehat F$ defined as follows. Recall that $\widehat\Si_\ca=( \Si_\ca\times V)/ \sim$. For each $v\in V$, let $(\Si_\ca)_v$ be the union of all faces in $\widehat\Si_\ca$ of form $E\times v$ with $E$ ranging over faces of $\Si_\ca$. By Definition~\ref{def:projection1}, there is a retraction $(\Pi_F)_v:(\Si_\ca)_v\to F\times v$ for each $v\in V$. It follows from Lemma~\ref{lem:compactible} that these maps $\{(\Pi_F)_v\}_{v\in V}$ are compatible in the intersection of their domains. Thus they fit together to define a retraction $\Pi_{\widehat F}:\widehat\Si_\ca\to \widehat F$.
\end{definition}
The following is a direct consequence of the definition.
\begin{lem}
	\label{lem:retraction property}
	Take faces $E$ and $F$ of $\Si_\ca$. Then $\Pi_{\widehat F}(\widehat E)=\widehat{\Pi_F(E)}$.
\end{lem}

Let $\mathcal A\otimes \mathbb C$ be the complexification of $\mathcal A$, which is a collection of affine complex hyperplanes in $\mathbb C^n$. Define
$$M(\mathcal A\otimes \mathbb C)=\mathbb C^n - \cup_{H\in\mathcal A} (H\otimes \mathbb C).$$
It follows from \cite{s87} that $\widehat\Si_\ca$ is homotopic equivalent to $M(\mathcal A\otimes \mathbb C)$, thus they have isomorphic fundamental groups. 

In the rest of this subsection, we look at the case when $W_S$ is a finite Coxeter group with its canonical representation $\rho:W_S\to GL(n,\mathbb R)$ \cite{davis2012geometry}. Recall a \emph{reflection} in $W_S$ is a conjugate of a standard generator of $W_S$. Each reflection fixes a hyperplane in $\mathbb R^n$, which we call a \emph{reflection hyperplane}. Let $\mathcal A$ be the collection of all reflection hyperplanes in $\mathbb R^n$. The hyperplane arrangement $\ca$ is called the \emph{reflection arrangement} associated with $W_S$. The following facts are standard, see e.g. \cite{paris2014k}.

\begin{enumerate}
	\item The fundamental group $\pi_1(M(\ca\otimes \mathbb C))$ is isomorphic to the pure Artin group $PA_S$ \cite{lek};
	\item As the action of $W_S$ permutes elements in $\mathcal A$, there is an induced action $W_S\act M(\ca\otimes \mathbb C)$ and an induced action $W_S\act \widehat\Si_\ca$ which are free. The quotient of each of these two spaces have $\pi_1$ isomorphic to $A_S$.
	\item The 2-skeleton of $\widehat\Si_\ca/W_S$ is isomorphic to the presentation complex of $A_S$.
\end{enumerate}

\begin{definition}
	\label{def:support}
	When $\ca$ is the reflection arrangement associated with a finite Coxeter group $W_S$, we will also write $\Si_\ca$ and $\od_{\ca}$ as $\Si_S$ and $\od_S$. Note that the 1-skeleton of $\od_\ca$ is isomorphic to the Cayley graph of $W_S$, and the 1-skeleton of $\Si_S$ is isomorphic to the unoriented Cayley graph of $W_S$ (by smashing each double edge of the usual Cayley graph to a single edge). Thus edges of $\od_\ca$ and $\Si_\ca$ are labeled by elements of $S$. Let $K$ be a subset, or an edge path in $\od_\ca$ or $\Si_\ca$. Then $\supp(K)$ is defined to the collection of vertices of $S$ which appear as the label of an edge which is contained in $K$. The \emph{type} of a standard subcomplex of $\Si_S$ or $\od_S$ is defined to be the support of this subcomplex.
\end{definition}

Note that the Coxeter complex $\bC_S$ of $W_S$ is homeomorphic to a sphere. More precisely, elements in $\ca$ cuts the unit sphere of $\mathbb R^n$ into a simplicial complex, which is isomorphic to $\bC_S$. From this, we know that $\bC_S$ and $\Si_\ca$ are dual complexes of each other.
We record the following description of the Artin complex $\Delta_S$ in terms of $\od_\ca$, which will be used later.
\begin{remark}
	\label{rmk:alternative}
	Let $X$ be the universal cover of $\od_\ca$. A \emph{lift} of a standard subcomplex in $\od_\ca$ is a connected component of the inverse image of this subcomplex under the map $X\to\od_\ca$.
	Vertices of $\Delta_S$ are in 1-1 correspondence with lifts standard subcomplexes of $\od_\ca$ of type $\hat s$ for some $s\in S$. A collection of vertices span a simplex if their associated lifts have non-trivial common intersection. 
\end{remark}

\subsection{Falk complexes for affine arrangements}
\label{subsec:deligne complex}
Let $\ca$ be a real affine arrangement. Let $D_\ca$ be the union of all elements in $\fan(\ca)$ which are bounded in $\mathbb R^n$. Note that $D_\ca$ is polyhedron complex, whose open cells are elements in $\fan(\ca)$ that are contained in $D_\ca$. A \emph{face} of $D_\ca$ is defined to be the closure of an open cell of $D_\ca$. Each face of $D_\ca$ is a disjoint union of fans. Top dimensional faces in $D_\ca$ are in 1-1 correspondence with components of $\mathbb R^n\setminus \cup_{H\in \ca} H$ which are bounded. A face $F$ of $D_\ca$ is \emph{dual} to a face $F'$ of $\Si_\ca$ if the barycenter of $F'$ (as defined in the beginning of Section~\ref{subsec:zonotope}) is contained in the interior of $F$. In this case, we will also say $F\subset D_\ca$ is dual to the standard subcomplex $\widehat F'$ of $\od_\ca$.

We define a simple complex of group structure $\mathcal U$ on $D_\ca$ as follows (see \cite[Definition II.12.11]{BridsonHaefliger1999}). For a face $F$ of $D_\ca$, the local group at $F$ is defined to be the fundamental group of the standard subcomplex of $\od_\ca$ which is dual to $F$.
The morphisms between the local groups are induced by inclusions of the associated subcomplexes. By Lemma~\ref{lem:injective}, all the morphisms between the local groups are injective, hence $\mathcal U$ is a simple complex of groups. Moreover, it follows from the retraction in Definition~\ref{def:retraction} that each local group injects into $\pi_1 (\od_\ca)$. Thus $\mathcal U$ is developable with $\pi_1 (\od_\ca)=\pi_1 \mathcal U$. 

Let $\bD_\ca$ be the development complex of $\mathcal U$ (see \cite[Theorem II.12.18]{BridsonHaefliger1999}).  Then $\bD_\ca$ is simply-connected (\cite[Corollary II.12.21]{BridsonHaefliger1999}). The complex $\bD_\ca$ is defined to be the \emph{Falk complex} of the affine arrangement $\ca$, motivated by the work of Falk \cite{falk1995k}.
We now give an alternative description of $\bD_\ca$.  Let $\widetilde K$ be the universal cover of $\od_\ca$. We start with a disjoint union of a collection $\mathcal C$ of multiple copies of faces of $D_\ca$ as follows: for each face $F$ of $D_\ca$, the copies of $F$ in $\mathcal C$ are in 1-1 correspondence with lifts of $\widehat E$ in $\widetilde K$, where $\widehat E$ is dual to $F$. Now we identify an element $F_1$ as a face of another element $F_2$ of $\mathcal C$, if the subcomplex of $\widetilde K$ associated with $F_2$ is contained in the subcomplex of $\widetilde K$ associated with $F_1$. Then we obtained $\bD_\ca$ from $\mathcal C$ after all such identifications. In particular, there is natural map $\bD_\ca\to D_\ca$, coming from quotienting $\bD_\ca$ by the action of $\pi_1\mathcal U$.

When $\ca$ is a central arrangement, there is another associated complex, called the \emph{spherical Deligne complex}, defined as follows. Let $S_\ca$ be the unit sphere, endowed with the polyhedron complex structure coming the the intersection of the unit sphere with $\ca$. Then there is 1-1 correspondence between open cells in $S_\ca$ and elements in $\fan(\ca)$ which are not 0-dimensional. A face $F$ of $S_\ca$ is dual to a face $F'$ of $\Si_\ca$ if the barycenter of $F'$ is contained in the fan associated with $F$. In this case, we also say $F$ is dual to the standard subcomplex $\widehat F'$ of $\od_\ca$. This gives a complex of group structure on $S_\ca$ as before, where the local group on $F$ is the fundamental group of $\widehat F'$. Then the \emph{spherical Deligne complex} for the central arrangement $\ca$, denoted $\bSD_\ca$, is defined to the development complex of this complex of group over $S_\ca$. This gives a natural map $\bSD_\ca\to S_\ca$.
The complex $\bSD_\ca$ also has a similar alternative description in terms of lifts of standard complexes of $\od_\ca$ in the universal cover as in the previous paragraph. The complex $\bSD_\ca$ is simply-connected if $S_\ca$ is simply-connected, i.e. $n\ge 3$.

Let $\ca$ be an arbitrary affine arrangement. Let $x$ be a vertex in $D_\ca$. The \emph{local arrangement of $\ca$ at $x$}, denoted by $\ca_x$, is defined to be the collection of elements of $\ca$ that contain $x$. Note that $\ca_x$ is an central arrangement, and $\lk(x,D_\ca)$ can be naturally identified as a subcomplex of $S_{\ca_x}$. The following is consequence of the description of $\bD_\ca$ and $\bSD_{\ca_x}$.
\begin{lem}
	\label{lem:link deligne}
	Let $x'\in \bD_\ca$ be a vertex which is mapped to $x\in D_\ca$ under $\bD_\ca\to D_\ca$. Let $N$ be the inverse image of $\lk(x,D_\ca)$ (viewed as a subset of $S_{\ca_x}$) under the map $\bSD_{\ca_x}\to S_{\ca_x}$. Then $\lk(x',\bD_\ca)\cong N$.
\end{lem}

\subsection{Collapsing hyperplanes}
\label{ss:col}
Let $\mathcal A$ be an arrangement of affine hyperplanes in $\mathbb R^n$. Let $\mathcal A'\subset \mathcal A$ be a sub-collection of hyperplanes. Then there is a cellular map $\widehat \Si_{\mathcal A}\to \widehat \Si_{\mathcal A'}$ defined as follows. First we define a map $\Si_{\mathcal A}\to \Si_{\mathcal A'}$. Note that each fan of $\ca$ is contained in a unique fan of $\ca'$. As vertices of the barycentric subdivision $b\Si_{\ca}$ of $\Si_\ca$ are in one to one correspondence with fans of $\ca$, this gives a map from the vertex set of $b\Si_\ca$ to the vertex set of $b\Si_\ca'$. One readily checks that vertices of a simplex are mapped to vertices of another simplex. So we can extend linearly to obtain a map $\kappa:b\Si_\ca\to b\Si_{\ca'}$, which can also be viewed as a map $\kappa:\Si_\ca\to\Si_{\ca'}$. Note that 
\begin{enumerate}
	\item $\kappa$ maps a face of $\Si_\ca$ onto a face of $\Si_{\ca'}$;
	\item for a face $F\subset\Si_\ca$ and two vertices $v,v'$ of $\Si_\ca$ satisfying $\prj_F(v')=\prj_F(v)$, we have $\prj_E(\kappa(v'))=\prj_E(\kappa(v))$ where $E=\kappa(F)$.
\end{enumerate}
Thus $\kappa:\Si_\ca\cong b\Si_\ca\to b\Si_{\ca'}\cong \Si_{\ca'}$ induces a continuous map $\hat \kappa:\od_\ca\to \od_{\ca'}$.

The map $\hat \kappa$ restricted to $\widehat \Si_{\mathcal A}^{(1)}$ has a more straightforward description. 
As each component of $\mathbb R^n - \cup_{H\in\mathcal A} H$ lies in a unique component of $\mathbb R^n - \cup_{H\in\mathcal A'} H$, this gives $\hat \kappa:\widehat \Si_{\mathcal A}^{(0)}\to \widehat \Si_{\mathcal A'}^{(0)}$. Moreover, two adjacent components of $\mathbb R^n - \cup_{H\in\mathcal A} H$ are either contained in the same component of $\mathbb R^n - \cup_{H\in\mathcal A'} H$ or correspond to two adjacent components of $\mathbb R^n - \cup_{H\in\mathcal A'} H$. This gives $\hat \kappa:\widehat \Si_{\mathcal A}^{(1)}\to \widehat \Si_{\mathcal A'}^{(1)}$, where an edge of $\widehat \Si_{\mathcal A}^{(1)}$ is collapsed to a single point if its endpoints are sent to the same point of $\widehat \Si_{\mathcal A'}^{(1)}$.

\subsection{Some properties of central arrangements}
Let $\ca$ be a finite arrangement of affine hyperplanes in $\mathbb R^n$. We say $\ca$ is \emph{central} if the intersection of all hyperplanes in $\ca$ is non-empty, in which case we can assume all hyperplanes in $\ca$ pass through the origin. Now we assume $\ca$ is central.
Take $H\in \ca$. We define the \emph{deconing} of $\ca$ with respect to $H$ to an affine hyperplane arrangement in $\mathbb R^{n-1}$ as follows. Note that the collection $\ca$ of hyperplanes in $\mathbb R^n$ give rise to a collection $\ca_p$ of $\mathbb R\mathbb P^{n-2}$ in $\mathbb R\mathbb P^{n-1}$. Let $H_p\in \ca_p$ be the copy of $\mathbb R\mathbb P^{n-2}$ corresponding to $H$. Then $\mathbb R\mathbb P^{n-1}\setminus H_p=\mathbb R^{n-1}$, and the intersection of elements in $\ca_p-\{H_p\}$ with $\mathbb R^{n-1}$ is a collection $\ca_H$ of affine hyperplanes in $\mathbb R^{n-1}$. This affine hyperplane arrangement $\ca_H$ in $\mathbb R^{n-1}$ is defined to be the deconing of $\ca$ with respect to $H$.

It is a well-known fact that $M(\mathcal A\otimes \mathbb C)$ is homeomorphic to $M(\ca_H\otimes \mathbb C)\times \mathbb C^*$ where $\mathbb C^*=\mathbb C-\{0\}$, see e.g. \cite{orlik2013arrangements}. Thus $\pi_1M(\mathcal A\otimes \mathbb C)\cong\pi_1M(\ca_H\otimes \mathbb C)\oplus\mathbb Z$. It is also possible to see this isomorphism on the level of Salvetti complex, see the following lemma, where $\widehat U_H$ in the lemma is isomorphic to the Salvetti complex of the deconing of $\ca$ with respect to $H$. 

\begin{lem}
	\label{lem:injective}
	Let $\mathcal A$ be a central arrangement in $\mathbb R^n$. Take a hyperplane $H$ of $\mathcal A$. Let $U_H$ be a maximal subcomplex of $\Si_{\mathcal A}$ contained in one side of $H$. Let $\widehat U_H$ be the associated subcomplex of $\widehat \Si_{\mathcal A}$. Then the inclusion $i:\widehat U_H\to \widehat \Si_{\mathcal A}$ is $\pi_1$-injective. Moreover, there is a $\mathbb Z$-subgroup of $Z\le \pi_1\Si_{\mathcal A}$ such that $\pi_1 \Si_{\mathcal A}= i_*(\pi_1\widehat U_H)\oplus Z$.
\end{lem}

\section{Weakly flagness in type $A_n$}
\label{sec:AD}
The goal of this section is to prove weakly flagness for certain relative Artin complexes associated to Artin groups of type $A_n$, see Theorem~\ref{thm:weaklyflagA}.
\subsection{A lattice theoretical lemma}
\begin{lem}
	\label{lem:big lattice}
	Suppose $X$ and $(V,\le)$ satisfy the assumptions of Theorem~\ref{thm:contractibleII}.  
	Then $(V,\le)$ is bowtie free and flag.
\end{lem}

\begin{proof}
	First we show $(V,\le)$ is bowtie free.	
	We verify the assumptions of Lemma~\ref{lem:bowtie free criterion}. Assumption 1 is a direct consequence of Assumptions 2 and 3 of Theorem~\ref{thm:contractibleII} (note that if $x$ is of type $\hat s_1$, then $V_\ge x$ is exactly the vertex set of $\lk(x,X)$). Let $x_1y_1x_2y_2$ be a 4-cycle in $X$ as in Assumption 2 of Lemma~\ref{lem:bowtie free criterion}. Let $Y$ be the Helly graph as defined in Theorem~\ref{thm:contractibleII}.
	As any Helly graph satisfies the 4-wheel condition (see e.g. \cite[Proposition 3.25]{weaklymodular}), there exists $z\in Y$ such that $z$ is adjacent to each of $\{x_1,y_1,x_2,y_2\}$ in $Y$. As $x_1,x_2$ are of type $\hat s_1$ and $y_1,y_2$ are of type $\hat s_n$, by the definition of edges in $Y$ we know $z$ is also adjacent to each of $\{x_1,y_1,x_2,y_2\}$ in $X$.
	
	Now we show $(V,\le)$ is downward flag. Let $\{x_1,x_2,x_3\}$ be three pairwise distinct elements in $V$ such that $x_i$ and $x_{i+1}$ has a lower bound $y_i$ for $i\in \mathbb Z/3\mathbb Z$. We assume $\{y_1,y_2,y_3\}$ are pairwise distinct, otherwise we can clearly find a lower bound for $\{x_1,x_2,x_3\}$. It follows from the assumption that we can assume $y_i$ is type $\hat s_1$ for $1\le i\le 3$. For simplicity, we will say $x\in V$ is of type $i$ if it is of type $\hat s_i$. We will use a downward induction on $\type (x_1)+\type (x_2)+\type (x_3)$.
	
	The base case of the induction is that each of $\{x_1,x_2,x_3\}$ is type $\hat s_n$. As $y_i$ is adjacent to both $x_i$ and $x_{i+1}$ in $X$, the same is true in $Y$. Thus $d(x_i,x_j)\le 2$ for $i\neq j$ in $Y$. As $x_i\neq x_j$ for $i\neq j$, it follows from the definition of $Y$ that $d(x_i,x_j)=2$. Thus the combinatorial balls of radius 1 centered at $x_1,x_2,x_3$ pairwise intersect. It follows from the Helly property of combinatorial balls that there exists a vertex $z\in Y$ such that $z$ is adjacent to $x_i$ in $Y$ for $i=1,2,3$. As $x_i$ is of type $\hat s_n$, it follows from the definition of $Y$ that $z$ is adjacent to each $x_i$ in $X$. Thus $z$ is a common lower bound for $\{x_1,x_2,x_3\}$.
	
	Now we assume the downward flagness is verified for any $\{x_1,x_2,x_3\}$ with $$\sum_{i=1}^3\type(x_i)\ge k.$$ Take $\{x_1,x_2,x_3\}$ with $\sum_{i=1}^3\type(x_i)=k-1$. By assumption, there exists $1\le i\le 3$ and $x'_i$ such that $x'_i>x_i$. We assume without loss of generality that $i=1$. Then by induction, $\{x'_1,x_2,x_3\}$ has a common lower bound in $V$, denoted by $z'$. As $y_1$ and $z'$ have at least one common upper bound (e.g. $x'_1$ and $x_2$) and $(V,\le)$ is bowtie free, by Lemma~\ref{lem:posets}, $y_1,z'$ have a join, denoted $z_1$. Then $z_1\le x'_1$ and $z_1\le x_2$. Similarly, $y_3$ and $z'$ have a join $z_3$ and $z_3\le x'_1$, $z_2\le x_3$. Note that each pair from $\{z_1,z_3,x_1\}$ have a lower bound, and $\{z_1,z_3,x_1\}\subset V_{\le x'_1}$. By assumption 4 of Theorem~\ref{thm:contractibleII}, there is a common lower bound $z$ for $\{z_1,z_3,x_1\}$. Note that $z\le z_1$ and $z_1\le x_2$, thus $z\le x_2$. Similarly, $z\le x_3$. Thus $z$ is a lower bound for $\{x_1,x_2,x_3\}$, as desired. 
	
	The proof of upward flagness of $(V,\le)$ is similar. 
\end{proof}

The following is an immediate consequence of Lemma~\ref{lem:big lattice} and Theorem~\ref{thm:triple}. 
\begin{cor}
	\label{cor:widetildeC_n}
	Suppose $\Lambda$ is a Coxeter diagram of type $\widetilde C_n$ with consecutive nodes being $\{s_i\}_{i=1}^n$. Then the vertex set of $\Delta_\Lambda$, endowed with the order from $s_1<s_2<\cdots<s_n$, is a bowtie free and flag poset. 
\end{cor}

\subsection{An injective simplicial complex for type $A_n$}
\label{subsec:injective}
In the rest of this subsection, $\Lambda$ will be the Coxeter diagram of type $A_n$.
Let consecutive nodes in $\Lambda$ be $\{s_1,\ldots, s_n\}$.
Let $\ca$ be the reflection arrangement of type $A_n$. Up to a linear transformation, elements in $\ca$ are $x_i=0$ for $1\le i\le n$, and $x_i=x_j$ for $1\le i\neq j\le n$.
Let $H$ be the hyperplane $x_1=0$. Then a simple calculation implies that deconing of $\ca$ with respect to $H$ gives the affine arrangement $\cb$ in $\mathbb R^{n-1}$ made of $y_i=1$ for $1\le i\le n-1$, $y_i=0$ for $1\le i\le n-1$ and $y_i=y_j$ for $1\le i\neq j\le n-1$.

Let $S_\ca$ and $D_\cb$ be defined in Section~\ref{subsec:deligne complex}. Then $S_\ca$ is isomorphic to the Coxeter complex of the associated Coxeter group, hence each vertex of $S_\ca$ has a type $\hat s_i$ for some $i$. The complex $D_\cb$ is a unit cube subdivided into orthoschemes (\cite[Definition 4.1]{brady2010braids}).
Note that $D_\cb$ can be realized as the maximal subcomplex of $S_\ca$ which is contained in the interior of a hemisphere bounded by $H\cap S_\ca$. Thus it makes sense to talk about types of vertices of $D_\cb$, using this embedding $D_\cb\to S_\ca$. We assume without loss of generality that the vertex of $D_\cb$ with coordinate $(0,\ldots,0)$ is of type $\hat s_1$. Then a vertex of $D_\cb$ is of type $\hat s_i$ if and only if the coordinate of this vertex has $i-1$ nonzero entries.

Let $\Si_\cb$ and $\Si_\ca$ be as in Section~\ref{subsec:zonotope}.
Note that $\Si_\cb$ can be identified as a maximal subcomplex of $\Si_\ca$ contained in one side of $H$. This embedding $\Si_\cb\to \Si_\ca$ is dual to $D_\cb\to S_\ca$. We record the following consequence of Lemma~\ref{lem:injective}.
\begin{lem}
	\label{lem:injective1}
The embedding $D_\cb\to S_\ca$ induces an embedding $\od_\cb\to \od_\ca$ which is $\pi_1$-injective.
\end{lem}

Let $\bD_\cb$ and $\bSD_\ca$ be the Falk complex and spherical Deligne complex defined in Section~\ref{subsec:deligne complex}. As $\od_\cb\to \od_\ca$ is $\pi_1$-injective, by the description of $\bD_\cb$ and $\bSD_\ca$ in terms of certain subcomplexes of the universal covers of $\od_\cb$ and $\od_\ca$ in Section~\ref{subsec:deligne complex}, we know there is an embedding $\bD_\ca\to \bSD_\ca$ which is equivariant with respect to $\pi_1\od_\cb\curvearrowright \bD_\cb$, $\pi_1\od_\ca\curvearrowright \bSD_{\ca}$ and $\pi_1\od_\cb\hookrightarrow \pi_1\od_\ca$. Hence we can view $\bD_\ca$ as a subcomplex of $\bSD_\ca$, moreover, this subcomplex can be alternatively described as a connected component of the inverse image of $D_\cb$ (viewed as a subcomplex of $S_\ca$) under the map $\bSD_\ca\to S_\ca$.

A vertex of $\bD_\cb$ is of type $\hat s_i$ if it maps to a vertex of type $\hat s_i$ under $\bD_\cb\to D_\cb$. Similarly, we define types of vertices in $\bSD_\ca$. Note that $\bD_\cb$ and $\bSD_\ca$ are simplicial complexes of type $S=\{s_1,\ldots,s_n\}$. Moreover, $\bSD_\ca$ is isomorphic to the Artin complex of type $A_n$ which preserves the types of vertices.

\begin{prop}
	\label{prop:An Helly}
The vertex set of $\bD_\cb$, endowed with the relation $<$ induced from $s_1<s_2<\cdots<s_n$ as in Definition~\ref{def:order}, is a poset satisfying all the assumptions of Theorem~\ref{thm:contractibleII}.
\end{prop}

\begin{proof}
Since the vertex set of $\bSD_\ca$ (which is the Artin complex of type $A_n$) with the induced order is a poset by Lemma~\ref{lem:poset structure}, so is the vertex set of $\bD_\cb$. 

As each vertex of $\bD_\cb$ is lower bounded by a vertex of type $\hat s_1$ and upper bounded by a vertex of type $\hat s_n$, it suffices to verify Theorem~\ref{thm:contractibleII} (3) for type $\hat s_1$ vertices and Theorem~\ref{thm:contractibleII} (4) for type $\hat s_n$ vertices. Take a vertex $x$ of $\bD_\cb$ of type $\hat s_1$. Then $x$ maps to $\bar x=(0,0,\ldots,0)$ under $\bD_\cb\to D_\cb$. Note that the local arrangement of $\cb$ at $\bar x$ (as defined in Section~\ref{subsec:deligne complex}) is an arrangement of type $A_{n-1}$. By Lemma~\ref{lem:link An} below and Lemma~\ref{lem:link deligne}, there is a vertex type preserving isomorphism between $\lk(x,\bD_\cb)$ and the Artin complex of type $B_{n-1}$, and we are done by Theorem~\ref{thm:triple}. The verification of Theorem~\ref{thm:contractibleII} (4) for type $\hat s_n$ vertices is similar.
\end{proof}

Now we consider the Artin group of type $B_n$ with consecutive nodes in the Coxeter diagram being $\{s_1,\ldots,s_n\}$ and the edge between $s_{n-1}$ and $s_n$ is labeled by $4$.
Consider the reflection arrangement $\ca'$ of type $B_n$ made of the following hyperplanes: $x_i=0$ for $1\le i\le n$, and $x_i\pm x_j=0$ for $1\le i\neq j\le n$. Then vertices of the complexes $S_{\ca'}$ and $\bSD_{\ca'}$ (cf. Section~\ref{subsec:deligne complex}) are labeled by $\hat s_i$ for $1\le i\le n$.

\begin{lem}
	\label{lem:link An}
Let $S'_\ca$ be the subcomplex of $S_\ca$ made of points with non-negative coordinates. We assume the vertex of $S'_\ca$ in the first octant corresponding to the line $x_1=x_2=\cdots=x_n$ is of type $\hat s_n$. Let $\bSD'_{\ca}$ be the inverse image of $S'_\ca$ under the map $\bSD_{\ca}\to S_{\ca}$. Then there is a simplicial isomorphism between $\bSD'_{\ca}$ and $\bSD_{\ca'}$ which preserve types of vertices.
\end{lem}

\begin{proof}
For the proof, we need an alternative description of $\bSD_{\ca}$ and $\bSD_{\ca'}$, due to Allcock \cite[Theorem 7.3]{MR3127810}.
Let $M(\ca)=\mathbb C^n\setminus (\cup_{H\in \ca} H\otimes \mathbb C)$ be the complement of complexified arrangement associated with $\ca$. Similarly we define $M(\ca')$. We endow $M(\ca)$ with the Euclidean metric induced from $\mathbb C^n$, which makes $M(\ca)$ an incomplete metric space whose metric completion (denoted $\bar M(\ca)$) is $\mathbb C^n$. Let $\bar M_{\mathbb R}(\ca)$ be the real part of   $\bar M(\ca)\cong \mathbb C^n$.
Let $\widetilde M(\ca)$ be the universal cover of $M(\ca)$, with the induced length metric from $M(\ca)$; and let $\widehat M(\ca)$ be the metric completion of $\widetilde M(\ca)$. Note that the covering map $\widetilde M(\ca)\to M(\ca)$ induced a map $p_\ca:\widehat M(\ca)\to \bar M(\ca)$. We view $S_\ca$ as the unit sphere in $\bar M_{\mathbb R}(\ca)$, hence $S_\ca$ is a subset of $\bar M(\ca)$. Then $\bSD_\ca$ is isomorphic to $p^{-1}_\ca(S_\ca)$. A similar discussion applies if we replace $\ca$ by $\ca'$.
There is a finite sheeted covering map $f:M(\ca')\to M(\ca)$ induced by $$(z_1,\ldots, z_n)\to (z^2_1,\ldots,z^2_n)$$ see \cite[Section 4]{allcock2002braid}. This gives a homeomorphism $\widetilde f: \widetilde M(\ca')\to \widetilde M(\ca)$ with the following diagram commutes:
\[ \begin{tikzcd}
	\widetilde M(\ca')\arrow{r}{\widetilde f} \arrow[swap]{d}{p_{\ca'}} &\widetilde M(\ca) \arrow{d}{p_\ca} \\%
	 M(\ca') \arrow{r}{f}&  M(\ca)
\end{tikzcd}
\]
As there is $L>0$ such that each point in $M(\ca')$ has an open neighborhood where $f$ restricts to an $L$-biLipschitz map, thus the same property holds for $\widetilde f$. It follows that $\widetilde f$ is a biLipschitz homeomorphism with respect to the induced length metric, hence it extends to a biLipschitz homeomorphism $\widehat f:\widehat M(\ca')\to \widehat M(\ca)$ which fits into
the following commutative diagram:
\[ \begin{tikzcd}
\widehat M(\ca')\arrow{r}{\widehat f} \arrow[swap]{d}{p_{\ca'}} &\widehat M(\ca) \arrow{d}{p_\ca} \\%
	\bar M(\ca') \arrow{r}{\bar f}& \bar M(\ca)
\end{tikzcd}
\]
Let $K=\bar f(S_{\ca'})$. As each point in $K$ has non-negative real coordinates, $S_{\ca'}=\bar f^{-1}(K)$. Then the above commutative diagram implies that $\widehat f$ induces a homeomorphism $p^{-1}_{\ca'}(S_{\ca'})\to p^{-1}_{\ca}(K)$. Now we consider the map $\alpha: \bar M_{\mathbb R}(\ca)-\{0\}\to S_{\ca}$ by sending $x$ to $\frac{x}{||x||}$. Note that $\alpha$ induces a homeomorphism $\alpha_{|K}:K\to S'_{\ca}$, which gives a homeomorphism $p^{-1}_{\ca}(K)\to p^{-1}_\ca(S'_\ca)$. Thus
$$
\bSD_{\ca'}\cong p^{-1}_{\ca'}(S_{\ca'})\cong p^{-1}_{\ca}(K)\cong p^{-1}_\ca(S'_\ca)\cong \bSD'_\ca.
$$
One readily verifies that the type of vertices are preserved, as $\bar f$ is type-preserving.
\end{proof}

\begin{thm}
	\label{thm:weaklyflagA}
	Suppose $\Lambda$ is a Coxeter diagram of type $A_n$. Let $\Lambda'\subset \Lambda$ be a linear subgraph made of three nodes. Then $\Delta_{\Lambda,\Lambda'}$ is weakly flag.
\end{thm}

\begin{proof}
In the following proof, we assume $V\Delta_\Lambda$ is endowed with the order induced from $s_1<s_2<\cdots<s_n$.	
Assume nodes of $\Lambda'$ are $t_1=s_i,t_2=s_{i+1}$ and $t_3=s_{i+3}$. We will only prove upward weakly flagness here,  as the proof of downward weakly flagness is almost identical. Take vertices $\{x_i\}_{i=1}^3$ in $\Delta_{\Lambda,\Lambda'}$ of type $\hat t_1$ and $\{y_i\}_{i=1}^3$ of type $\hat t_2$. Suppose $y_i$ is a common upper bound for $x_i$ and $x_{i+1}$. Then we need to show $\{x_1,x_2,x_3\}$ have a common upper bound in $\Delta_{\Lambda,\Lambda'}$.

Let $\omega$ be the 6-cycle $x_1y_1x_2y_2x_3y_3$ in $\Delta_{\Lambda,\Lambda'}$. Let $\pi:\Delta_\Lambda\to\mathfrak C_\Lambda\cong S_\ca$ be the map induced by quotienting $\Delta_\Lambda$ by the action of the pure Artin group. By \cite[Proposition 3.1]{huang2024Dn}, it suffices to consider the case when $\pi(\omega)$ is a single edge.

We can assume without loss of generality that $\pi(\omega)\subset D_\cb\subset S_\ca$. As $\bD_\cb$ can be identified with the subcomplex of $\bSD_\ca\cong \Delta_\Lambda$ which is a component of the inverse image of $D_\cb$ under the map $\bSD_\ca\to S_\ca$, we can assume $\omega\subset \bD_\cb$. By Proposition~\ref{prop:An Helly} and Lemma~\ref{lem:big lattice}, there is a vertex $z\in \bD_\cb$ which is the common upper bound for $\{x_1,x_2,x_3\}$ in $(V\bD_\cb,\le)$. As $(V\bD_\cb,\le)$ is bowtie free, we know $y_i$ is the join of $x_i$ and $x_{i+1}$. Thus $y_i<z$ in $(V\bD_\cb,\le)$. Then $z$ is of type $s_j$ with $j\ge i+2$. If $j=i+2$, then $z\in \Delta_{\Lambda,\Lambda'}$ and we are done. Now assume $j>i+2$. Let $\lk^-(z,\Delta_\Lambda)$ be the full subcomplex of $\lk(z,\Delta_\Lambda)$ spanned by vertices $<z$. Then $\omega\subset \lk^-(z,\Delta_\Lambda)$. Note that $\lk^-(z,\Delta_\Lambda)\cong \Delta_{\Lambda_j}$ where $\Lambda_j$ is a Coxeter diagram of type $A_{j-1}$. By the same  argument as before, we find $z'\in \lk^-(z,\Delta_\Lambda)$ such that $\omega\subset \lk^-(z',\Delta_\Lambda)$. Note that $z'<z$. Repeating this procedure finitely many times will eventually give $z_0\in \Delta_\Lambda$ of type $\hat t_3$ such that $\omega\subset \lk^-(z_0,\Delta_\Lambda)$. Then $z_0\in \Delta_{\Lambda,\Lambda'}$, as desired.
\end{proof}

The $n=3$ case of Theorem~\ref{thm:weaklyflagA} was previously known \cite[Lemma 4.2]{charney2004deligne}.

By Theorem~\ref{thm:weaklyflagA}, Theorem~\ref{thm:bowtie free} and \cite[Theorem 5.2]{goldman2023cat}, we have the following.
\begin{cor}
	\label{cor:cat(1)}
		Suppose $\Lambda$ is a Coxeter diagram of type $A_n$. Let $\Lambda'\subset \Lambda$ be a linear subgraph made of three nodes. Then $\Delta_{\Lambda,\Lambda'}$ with its simplices equipped with $A_3$-shape is CAT$(1)$.
\end{cor}

\section{Some sub-arrangements of the $H_3$-arrangement}
\label{sec:subarrangement}
It is natural to ask whether Falk complexes associated deconing of arrangement of type $H_3$ can be equipped with an equivariant non-positive curved metric, similar to what happens in Section~\ref{subsec:injective} for type $A_n$. This is unfortunately not clear. However, if we are willing to consider the deconing of sub-arrangements of type $H_3$, then it is possible to arrange non-positive curved metric on the associated Falk subcomplex. However, it is rather subtle to decide which sub-arrangements we should use - if the sub-arrangement is too small, then it is useless for our ultimate goal, namely to understand 6-cycles in the Artin complex of type $H_3$; if the sub-arrangement is too large, then there might be no way to metrize the associated Falk complex to make it non-positively curved. 

This section is a preparation of Section~\ref{sec:H3}. More precisely, we discuss two sub-arrangements of the $H_3$-arrangements and prove some useful properties on the structure of their Falk complexes, which will be used in Section~\ref{sec:H3}. The reader can start with Section~\ref{sec:H3}, and refer back to this section if necessary.

\subsection{Auxiliary arrangement I}
\label{subsec:aug1}
Let $W_S$ be the Coxeter group of type $H_3$. Suppose $S=\{a,b,c\}$, with $m_{ab}=5$, $m_{bc}=3$ and $m_{ac}=2$.
Let $\ca$ be the collection of reflection hyperplanes in $\mathbb R^3$ arising from the canonical representation $W_S\to GL(3,\mathbb R)$.

Let $\bC$ be the simplicial complex obtained by intersecting the unit sphere of $\mathbb R^3$ with elements in $\ca$. Then $\bC$ is the Coxeter complex associated with $W_S$. Each element in $\ca$ gives a \emph{wall} of $\bC$. 

\begin{definition}[Auxiliary sub-arrangement I]
	\label{def:arrangment}
	We define a sub-collection of walls in $\bC$ as follows. Takes three consecutive vertices $\{\theta_i\}_{i=1}^3$ in a wall of $\bC$ such that $\theta_1$ and $\theta_3$ are of type $\hat c$, and $\theta_2$ is type $\hat b$. Let $\mathcal H$ be the collection of walls of $\bC$ which passes at least one of $\theta_1,\theta_2$ or $\theta_3$. See Figure~\ref{fig:1} left for $\mathcal H$. We also think $\mathcal H$ as a central arrangement in $\mathbb R^3$.
	
	Let $H\subset\mathcal H$ be a wall passing through $\theta_1$. We consider the deconing of the arrangement $\mathcal H$ with respect to $H$. This gives arrangement $\mathcal H'$ of affine hyperplanes in $\mathbb R^2$, in Figure~\ref{fig:1} right.
\end{definition}

\begin{figure}[h]
	\centering
	\includegraphics[scale=0.7]{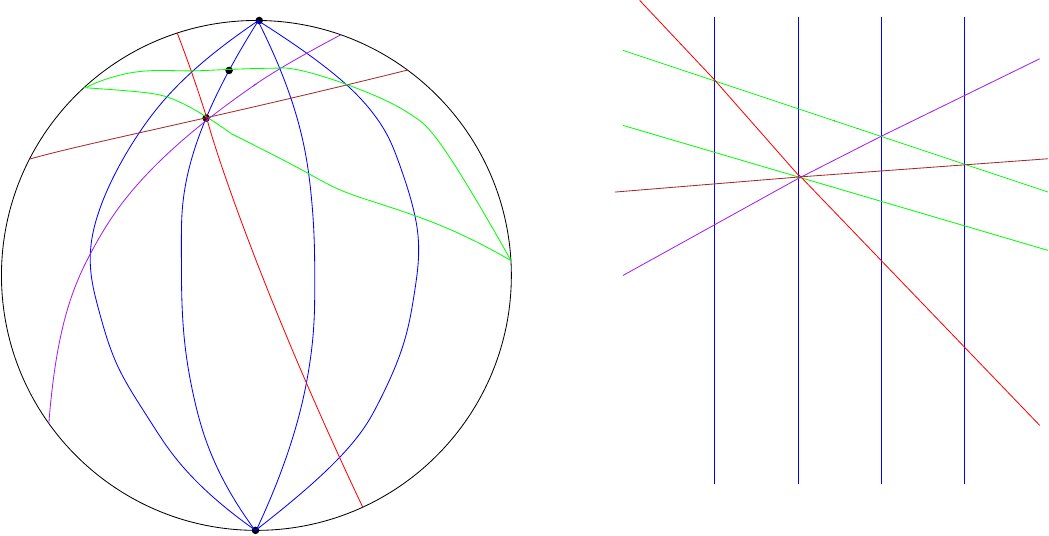}
	\caption{Auxiliary sub-arrangement I.}
		\label{fig:1}
\end{figure}	

Let $\Si_{\mathcal H'}$ (resp. $\Si_{\mathcal H}$) be the dual polyhedron (cf. Section~\ref{subsec:zonotope}) associated with $\mathcal H'$ (resp. $\mathcal H$), see Figure~\ref{fig:2} left.
Recall that there is an embedding $i_H:\Si_{\ch'}\to \Si_\ch$ whose image is a maximal subcomplex of $\Si_\ca$ which is contained in one side of $H$.
Let $\widehat \Si_{\mathcal H'}$ (resp. $\widehat \Si_{\mathcal H}$) be the associated Salvetti complex. Let $X=\Sigma_{\mathcal H'}$ and $\wX=\widehat \Sigma_{\mathcal H'}$. Then the embedding $i_H:\Si_{\ch'}\to \Si_\ch$ induces an embedding $\hat i_H:\od_{\ch'}\to\od_\ch$ which is $\pi_1$-injective by Lemma~\ref{lem:injective}.

We now define a collection of subcomplexes of $X$ and $\wX$.
Denote the four vertical walls of $\mathcal H'$ by $h_1,h_2,h_3,h_4$ (from left to right). Let $X_i$ be the union of all closed cell of $X$ which has non-trivial intersection with $h_i$. Let $\widehat X_i$ be the subcomplex of $\wX$ associated with $X_i$. For $i=1,2,3$, let $\widehat Y_i=\widehat X_i\cap \widehat X_{i+1}$.

We define a family of subcomplexes $\{X_{ij}\}_{1\le i\le 4,1\le j\le 2}$ of $X$ as follows. For $i=1,3$, $X_{i1}$ is the subspace in $X_i$ colored white in Figure~\ref{fig:2} (i.e. the hexagonal face), and $X_{i2}$ is the subspace of $X_i$ colored gray in Figure~\ref{fig:2} (i.e. the union of three squares). For $i=2,4$, $X_{i1}$ is the subspace in $X_i$ colored black in $X_i$ (i.e. the square on top), and $X_{i2}$ is the subspace in $X_i$ colored white. Let $\wX_{ij}$ be the subcomplex of $\wX$ associated with $X_{ij}$.

\begin{figure}[h]
	\centering
	\includegraphics[scale=0.85]{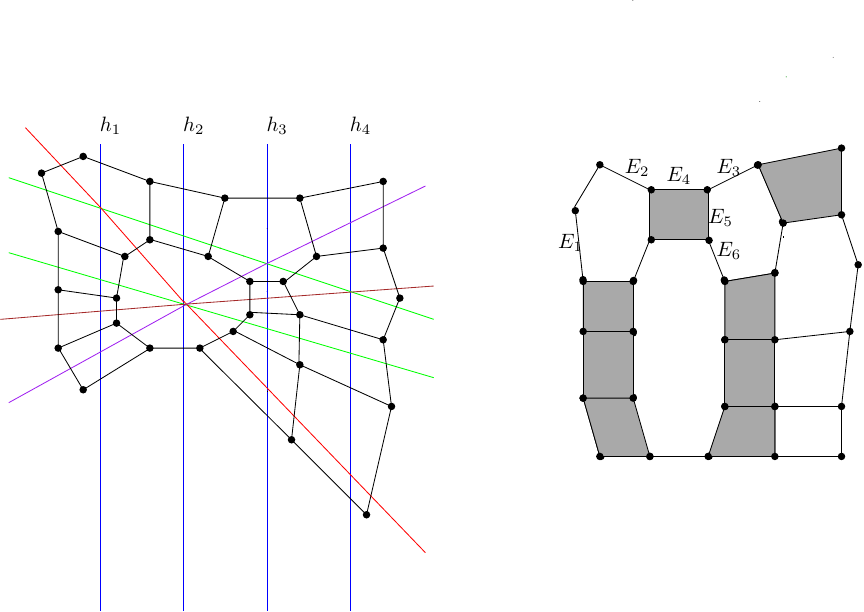}
	\caption{Dual complex.}
		\label{fig:2}
\end{figure}	

\begin{lem}
	\label{lem:BS}
	The group $\pi_1 \wX$ splits as a graph of groups whose underlying graph is a linear graph with four vertices. The vertex groups (from left to right) are $\pi_1 \wX_1,\pi_1 \wX_2$, $\pi_1 \wX_3$ and $\pi_1 \wX_4$; and the edge groups (from left to right) are $\pi_1 \wY_1$, $\pi_1\wY_2$ and $\pi_1 \wY_3$. Moreover, the inclusion $\whX_{ij}\to \whX$ is $\pi_1$-injective for $1\le i\le 4$ and $1\le j\le 2$.
\end{lem}

\begin{proof}
	For the first assertion of the lemma, it suffices to show the maps $\pi_1\wY_i\to \pi_1\wX_i$ and $\pi_1\wY_i\to \pi_1\wX_{i+1}$ induced by inclusion are injective for $i=1,2,3$. We only show this for $\pi_1\wY_1\to \pi_1\wX_1$, as the other inclusions are similar. Note that $\pi_1 \wX_1$ splits as amalgamation $\pi_1 \wX_{11}*_{\pi_1 A}\pi_1\wX_{12}$ where $A=\wX_{11}\cap \wX_{12}$, as $\pi_1 A\to \pi_1 \wX_{11}$ is injective (because of the retraction $\wX_{11}\to A$ as in Definition~\ref{def:retraction}) and $\pi_1 A\to \pi_1\wX_{12}$ is injective (because $\wX_{12}$ is a product of $A$ and $\wX_{12}\cap \wY_1$).
	By \cite[Page 6, Proposition 3]{serre2002trees}, the injectivity of $\pi_1\wY_1\to \pi_1\wX_1$ would follow if we can show $\pi_1 A \cap \pi_1(\wX_{11}\cap \wY_1)$ is trivial in $\pi_1 \wX_{11}$, and $\pi_1 A \cap \pi_1(\wX_{12}\cap \wY_1)$ is trivial in $\pi_1 \wX_{12}$. The first statement follows by considering the retraction $\wX_{11}\to A$ which maps $\wX_{12}\cap \wY_1$ to a single point, and the second statement follows from the product structure of $\wX_{12}$. 
	
	For the second assertion of the lemma, the previous paragraph already implies that $\wX_{11}\to \wX_1$ and $\wX_1\to \wX$ is $\pi_1$-injective, hence the same holds for $\wX_{11}\to \wX$. The case of $\wX_{ij}$ is similar.
\end{proof}

\begin{lem}
	\label{lem:connected}
	Let $\widetilde K$ be the universal cover of $\whX_1\cup \whX_2$. Let $T_1$ be a lift of $\whX_{ij}$ in $\widetilde K$, and let $T_2$ be a lift $\whX_{i'j'}$, with $1\le i,i',j,j'\le 2$. If $T_1\cap T_2\neq\emptyset$, then $T_1\cap T_2$ is connected.
\end{lem}

\begin{proof}
	We assume $\whX_{ij}=\whX_{11}$ and $\whX_{i'j'}=\whX_{22}$. The other cases are similar. Suppose $T_1\cap T_2$ is not connected. Then we take two connected components $S_1$ and $S_2$ of $T_1\cap T_2$. Then $S_1$ and $S_2$ are two different lifts of $\whX_{11}\cap \whX_{22}$. For $i=1,2$, let $\wtP_i\subset T_i$ be a path from $x_1\in S_1$ to $x_2\in S_2$. Let $P_i\subset \whX_1\cup \whX_2$ be the image of $\wtP_i$ under the covering map. Then $P_1$ and $P_2$ are homotopic rel endpoints in $\whX_1\cup \whX_2$. Now consider the retraction map $r:\whX\to \whX_{11}$ (cf. Definition~\ref{def:retraction}) which restricts to a retraction map $r:\whX_1\cup \whX_2\to \whX_{11}$. Then $r(P_1)=P_1$ and $r(P_2)$ are homotopic rel endpoints in $\whX_{11}$. As $r(\whX_{22})=\whX_{11}\cap \whX_{22}$, we know $r(P_2)\subset \whX_{11}\cap \whX_{22}$. Thus $P_1$ is homotopic rel endpoints in $\whX_{11}$ to a path in $\whX_{11}\cap \whX_{22}$. Hence $S_1=S_2$, which is a contradiction, and the lemma is proved.
\end{proof}

We now define a simple complex of group structure $\mathcal U$ on $\pi_1 (\wX_1\cup \wX_2)$ as in Figure~\ref{fig:3}, where the underlying complex $U$ is a union of two triangles and all the groups over 2-faces are trivial. Local groups over vertices and edges are fundamental groups of subcomplexes of $\wX$ as labeled in Figure~\ref{fig:3}. The morphisms between the local groups are induced by inclusions of the associated subcomplexes. By Lemma~\ref{lem:injective}, all the morphisms between the local groups are injective, hence $\mathcal U$ is a simple complex of groups. It follows from Lemma~\ref{lem:BS} that each local group injects into $\pi_1 (\wX_1\cup \wX_2)$. Thus $\mathcal U$ is developable with $\pi_1 (\wX_1\cup \wX_2)=\pi_1 \mathcal U$. 

\begin{figure}[h]
	\label{fig:3}
	\centering
	\includegraphics[scale=0.7]{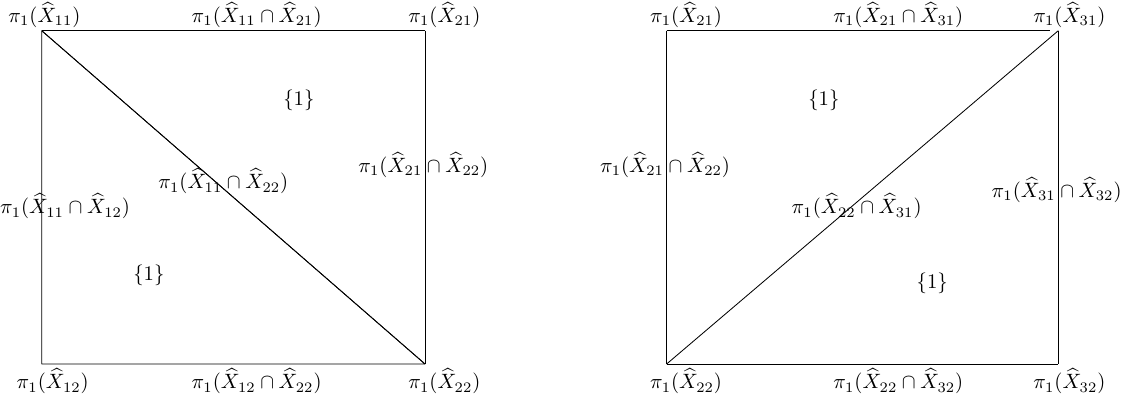}
\end{figure}

Let $\bU$ be the development complex of $\mathcal U$ (see \cite[Theorem II.12.18]{BridsonHaefliger1999}). We now give an alternative description of $\bU$. Let $\widetilde K$ be the universal cover of $\whX_1\cup \whX_2$. Vertices of $\bU$ are in 1-1 correspondence with lifts of $\whX_{ij}$ in $\widetilde K$ ($1\le i\le 2,1\le j\le 2$). These lifts are called \emph{standard subcomplexes} of $\widetilde K$.
Each vertex of $\bU$ has a well-defined \emph{type}, which is one of $\{\whX_{ij}\}_{1\le i\le 2,1\le j\le 2}$. Two vertices of $\bU$ are adjacent in $\bU$ if their associated subcomplexes of $\widetilde K$ have non-empty intersection, and their types correspond to adjacent vertices in $U$. Three mutually adjacent vertices of $\bU$ form a triangle, if their associated subcomplexes have non-empty common intersection.

There is a natural action $\pi_1(\whX_1\cup \whX_2)\act \bU$, with the quotient complex being $U$. We metrize $U$ such that it is made of two flat right-angled isosceles triangles of equal size, with right-angles at the vertices labeled by $\pi_1(\whX_{21})$ and $\pi_1(\whX_{12})$. This pulls back to a piecewise Euclidean metric on $\bU$.

\begin{lem}
	$\bU$ is CAT$(0)$.
\end{lem}

\begin{proof}
	As $\bU$ is simply-connected \cite[Corollary II.12.21]{BridsonHaefliger1999}, we need to show the link of each vertex is CAT$(1)$. As $\whX_{12}$ splits as a product of $\whX_{11}\cap \whX_{12}$ and $\whX_{12}\cap \whX_{22}$, we know the link of a vertex of type $\whX_{22}$ is a complete bipartite graph with each edge having length $=\pi/2$. Thus this link is CAT$(1)$. Similarly the link at any vertex of type $\whX_{21}$ is CAT$(1)$.
	
	Now we look at a vertex $v\in\bU$ of type $\wX_{22}$. Let $\Gamma$ be its link.
	Let 
	\begin{center}
	$D_1=\wX_{22}\cap \wX_{21}$, $D_2=\wX_{22}\cap \wX_{11}$ and $D_3=\wX_{22}\cap \wX_{12}$.
	\end{center}
 Then vertices of $\Gamma$ are in one to one correspondence with lifts of $D_i$ in $\wtX_{22}$ for $1\le i\le 3$. Two vertices are adjacent if and only if the associated lifts have non-trivial intersection. We say a vertex of $\Gamma$ is of type $D_i$ if it corresponds to a lift of $D_i$. Each edge of $\Gamma$ has length $=\pi/4$.
	
	Take an embedded cycle in $\Gamma$ with consecutive vertices  $\{w_i\}_{i\in \mathbb Z/n\mathbb Z}$. Let $\wtX_i$ be the subcomplex corresponding to $w_i$. Let $\wtP_i$ be an edge path from $\wtX_{i-1}\cap \wtX_i$ to $\wtX_i\cap \wtX_{i+1}$. Let $\wtP$ be the loop obtained by the concatenation of $\{\wtP_i\}_{i=1}^n$, which projects to a loop $P=P_1P_2\cdots P_n$ in $\wX_{22}$ which is null-homotopic in $\wX_{22}$. We can assume $P$ does not backtrack locally. Given a hyperplane $H\in \mathcal H'$ as in Definition~\ref{def:arrangment1}, an $H$-segment of $P$ is a maximal subpath of $Q\subset P$ such that the image of $Q$ under the projection to the dual polyhedron $\wX\to X$ is a single edge dual to $H$. By \cite[Lemma 3.6]{falk1995k}, for each hyperplane $H$ of $\mathcal H'$ dual to $X_{22}$, $P$ contains at least two disjoint $H$-segments.
	
	It follows that $\{w_i\}_{i\in \mathbb Z/n\mathbb Z}$ contains at least two different members of type $D_1$, two different members of type $D_2$ and one member of type $D_3$. As each edge of $\Gamma$ has length $\pi/4$ and $\Gamma$ is bipartite, the only possibility of such cycle having length $<2\pi$ is the case of the cycle has 6 vertices, with $w_1,w_5$ being of type $D_1$, $w_2,w_4,w_6$ being of type $D_2$ and $w_3$ being of type $D_3$. In this case $P_1,P_5$ are loops based at $D_1\cap D_2$, $P_3$ is a loop based at $D_2\cap D_3$ and $P_6$ is a loop based at $D_1\cap D_2$. 
	
	By consider the image of $P$ under the retraction (cf. Definition~\ref{def:retraction}) $\Pi_{D_1}:\wX_{22}\to D_1$, we know the concatenation $P_1P_5$ is null-homotopic in $D_1$. Thus we can assume $P_5=P^{-1}_1$ (i.e. $P_5$ is the inverse path of $P_1$). As $P$ is null-homotopic in $\wX_{22}$, we know $[P_2P_3P_4]=[P^{-1}_1P^{-1}_6P_1]$ represents the same element in $\pi_1(\wX_{22},x_0)$ where $x_0=D_1\cap D_2$ (we use $[\cdots]$ to denotes homotopy class of paths in $\wX_{22}$ rel end points). By Lemma~\ref{lem:injective}, $[P_1]=[QR]$ where $Q,R$ are loops, $Q\subset D_2\cup D_3$ and $[R]$ is in the center of $\pi_1(\wX_{22},x_0)$. Thus $[P^{-1}_1P^{-1}_6P_1]=[Q^{-1}P^{-1}_6Q]$.  
	
	We decompose $Q$ into subsegments alternating between $D_2$ and $D_3$ as $Q=Q_1Q_2\cdots Q_n$ such that each $Q_i$ is a maximal sub-segment of $Q$ in $D_2$ or $D_3$. We can assume each $Q_i$ is either a homotopically non-trivial loop in the associated $D_2$ or $D_3$, or a path with two different endpoints. As $Q$ is a loop, we know that $n$ is always an odd number. Moreover $n\ge 3$ as $Q$ is not homotopic rel endpoints in $\wX_{22}$ to a loop contained in $D_2$ - otherwise $P_1$ and $P_6$ give commuting elements in $\pi_1(\wX_{22})$, which is a contradiction as Lemma~\ref{lem:injective} implies that $P_1$ and $P_6$ generate a free group in $\pi_1(\wX_{22})$.

	Now we can write $$[Q^{-1}_nQ^{-1}_{n-1}\cdots Q^{-1}_1P^{-1}_6 Q_1\cdots Q_n]=[P_2P_3P_4].$$ By Lemma~\ref{lem:injective},
	\begin{center}
		$Q^{-1}_nQ^{-1}_{n-1}\cdots Q^{-1}_1P^{-1}_6 Q^{-1}_1\cdots Q_n$ and $P_2P_3P_4$
	\end{center}  represent the same element in $\pi_1(D_2\cup D_3)$. As $n\ge 3$, we know $$Q^{-1}_nQ^{-1}_{n-1}\cdots Q^{-1}_1P^{-1}_6 Q_1\cdots Q_n$$ is a concatenation of at least 5 sub-segments alternating in $D_2$ and $D_3$ such that each of them is either a homotopically non-trivial loop in the associated $D_2$ or $D_3$ (note that $Q^{-1}_1P^{-1}_6 Q_1$ is homotopically non-trivial in $D_2$ as $P^{-1}_6$ is), or a path with two different endpoints. On the other hand, $P_2P_3P_4$ only has 3 such sub-segments. 
	
	We view $D_2\cup D_3$ as a graph of spaces, with the underlying graph being a single edge. The two vertex spaces are $D_2$ and $D_3$, and the edge space is a single point $D_2\cap D_3$. Let $T$ be the associated Bass-Serre tree. Then $Q^{-1}_nQ^{-1}_{n-1}\cdots Q_1P^{-1}_6 Q_1\cdots Q_n$ gives a geodesic segment in $T$ with at least 5 vertices, while $P_2P_3P_4$ gives a geodesic segment with 3 vertices. Thus we can not have $[Q^{-1}_nQ^{-1}_{n-1}\cdots Q_1P^{-1}_6 Q_1\cdots Q_n]=[P_2P_3P_4]$. It follows that $\Gamma$ is CAT$(1)$.
	
	It remains to verify the link at a vertex of type $\whX_{11}$. Each edge in the link again has length $=\pi/4$. The girth of the link is $\ge 8$ by \cite[Lemma 3.6]{falk1995k} and the argument as before.
\end{proof}

Let $\wZ=\wX_{11}\cup \wX_{22}\cup \wX_{21}$. Then there is an embedding $\wZ\to \od_\ca$ defined as follows. Consider $\wZ\subset \wX=\od_{\ch'}\to \od_{\ch}$ where the second map is the map $\hat i_H$ after Definition~\ref{def:arrangment}. Let $\kappa:\Si_\ca\to \Si_\ch$ and $\hat \kappa:\od_{\ca}\to \od_{\ch}$ be as in Section~\ref{ss:col}. Note that there is a unique face $F$ of $\Si_\ca$ such that $\kappa(F)=i_H(X_{11})$. Moreover, the collection of elements of $\ca$ dual to $F$ is identical to the collection of elements of $\ch$ dual to $i_H(X_{11})$. Thus $\kappa$ maps $F$ homeomorphically to $i_H(X_{11})$, and consequently $\hat \kappa$ map $\widehat F$ homeomorphically to $\hat i_H (\wX_{11})$. Similarly discussion applies to $\hat i_H (\wX_{22})$ and $\hat i_H (\wX_{21})$, which gives a lift $\hat i_H(\wZ)$ to a subcomplex of $\od_{\ca}$. Let $i:\wZ\to \od_{\ca}$ be such embedding.

\begin{lem}
	\label{lem:transfer1}
	Let $P\subset \wZ$ be a loop. If $i(P)$ is null-homotopic in $\od_\ca$, then $P$ is null-homotopic in $\wX_1\cup\wX_2$.
\end{lem}

\begin{proof}
	Note that if $i(P)$ is null-homotopic in $\od_\ca$, then $\hat \kappa\circ i(P)=\hat i_H(P)$ is null-homotopic in $\od_\ch$. Thus $P$ is null-homotopic in $\wX=\od_{\ch'}$ by Lemma~\ref{lem:injective}. Now Lemma~\ref{lem:BS} implies $P$ is null-homotopic in $\wX_1\cup \wX_2$.
\end{proof}

We end this section by recording the following lemma for later use. For three points $x,y,z$ in a CAT$(0)$ space, we use $\angle_y(x,z)$ to denote the Alexandrov between the geodesic segments $\overline{yx}$ and $\overline{yz}$ (cf. \cite[Chapter II.3.1]{BridsonHaefliger1999}). Let $E_1,E_2$ be two edges of $X_{11}$ as in Figure~\ref{fig:2}, and let $E_3,E_5,E_6$ be the edges of $X_{13}$ in Figure~\ref{fig:2}. Let $E_4$ be the edge in $X_{12}$ in Figure~\ref{fig:2}.
For $1\le i\le 6$, let $\wE_i$ be the associated subcomplexes of $\wX$.
\begin{lem}
	\label{lem:piangle}
	Let $\widetilde K$ be the universal cover of $\wX_{1}\cup \wX_{2}$.
	Let $\{z_i\}_{i=1}^3$ be three consecutive vertices in $\mathbb U$. Let $\{T_i\}_{i=1}^3$ be the associated standard subcomplexes in $\widetilde K$. We assume one of the following situations holds:
	\begin{enumerate}
		\item $z_1,z_2,z_3$ are of type $\wX_{12},\wX_{11},\wX_{12}$ respectively and there is a path from $T_1\cap T_2$ to $T_3\cap T_2$ which projects to a homotopically nontrivial loop in $\wX_{11}\cap \wX_{21}$ or $\wE_1$;
		\item  $z_1,z_2,z_3$ are of type $\wX_{21},\wX_{22},\wX_{21}$ respectively and there is a path from $T_1\cap T_2$ to $T_3\cap T_2$ which projects to a homotopically nontrivial loop in $\hat E_6$;
		\item  $z_1,z_2,z_3$ are of type $\wX_{22},\wX_{11},\wX_{22}$ or type $\wX_{21},\wX_{11},\wX_{21}$  respectively and there is a path from $T_1\cap T_2$ to $T_3\cap T_2$ which projects to a homotopically nontrivial loop in $\wE_2$.
	\end{enumerate}
	Then $\angle_{z_2}(z_1,z_3)=\pi$.
\end{lem}

%$z_1,z_2,z_3$ are of type $\wX_{21},\wX_{11},\wX_{21}$ respectively and there is a path from $T_1\cap T_2$ to $T_3\cap T_2$ which projects to a homotopically nontrivial loop in $\wX_{11}\cap \wX_{12}$;

\begin{proof}
	Let $\bar z_i$ be the projection of $z_i$ under the map $p:\bU\to U$ induced by the action of $\pi_1(\whX_1\cup\whX_2)$. We view $z_1$ and $z_3$ as vertices in the link $\lk(z_2,\bU)$ of $z$ in $\bU$. Then any path in $\lk(z_2,\bU)$ joining $z_1$ and $z_3$ projects to path in $\lk(\bar z_2,U)$ joining $\bar z_1$ and $\bar z_3$. 
	For (1), note that $\bar z_1=\bar z_3$, thus $\angle_{z_2}(z_1,z_3)$ is a multiple of $\pi/2$. Now look at the $\widehat E_1$ case of (1). If $\angle_{z_2}(z_1,z_3)=\pi/2$, then there is a vertex $z\in\bU$ such that it is adjacent to $z_1,z_2,z_3$. Let $T$ be the subcomplex of $\widetilde K$ associated with $z$. Note that $z$ has type $\wX_{22}$. As $z_1,z_2,z$ span a triangle, then $T_1\cap T_2\cap T\neq \emptyset$. Hence this triple intersection is a vertex, which we denoted by $x_1$. Similarly, let $x_2=T_2\cap T\cap T_3$.

	Let $\wtP$ be the path from a point in $y_1\in T_1\cap T_2$ to a point in $y_2\in T_2\cap T_3$ as in (1). Let $\wtP'$ be the concatenation of a path $\wtP'_1\subset T_1\cap T_2$ from $y_1$ to $x_1$, a path $\wtP'_2\subset T_2\cap T$ from $x_1$ to $x_2$ and a path $\wtP'_3\subset T_2\cap T_3$ from $x_2$ to $y_2$. These paths exist by Lemma~\ref{lem:connected}. Let $P$ and $P'$ be the image of $\wtP$ and $\wtP'$ in $\wX_1\cup \wX_2$ under the covering map. Then $P\subset \widehat E_1$ and $P'\subset (\wX_{11}\cap \wX_{12})\cup (\wX_{11}\cap \wX_{22})$. Now we consider the retraction $r:\wX\to \widehat E$ as in Definition~\ref{def:retraction} which restricts to $r:\wX_{11}\to \widehat E_{1}$. As $P$ and $P'$ are homotopic rel endpoints in $\wX_1\cup \wX_2$, we know $r(P)=P$ and $r(P')$ are homotopic rel endpoints in $\widehat E_1$. However, this is a contradiction as $r(P')$ is a point. This finishes the proof of the $\widehat E_1$ case of (1).
	The $\whX_{11}\cap \whX_{21}$ case of (1), as well as (2) and (3) follows from a  similar argument.
\end{proof}

\subsection{Auxiliary arrangement II}
\label{subsec:AuII}
Let $\bC$ and $\ca$ be as before, defined from the Coxeter group of type $H_3$.
\begin{definition}[Auxiliary sub-arrangement II]
	\label{def:arrangment1}
	We define a sub-collection of walls in $\bC$ as follows. Takes four consecutive vertices $\{\theta_i\}_{i=1}^4$ in a wall of $\bC$ such that $\theta_1$ is of type $\hat c$, $\theta_2,\theta_4$ are of type $\hat a$ and $\theta_3$ is of type $\hat b$. Let $\mathcal K$ be the collection of walls of $\bC$ which passes at least one of $\{\theta_i\}_{i=1}^4$. See Figure~\ref{fig:5} left for $\mathcal K$. We also think $\mathcal K$ as a central arrangement in $\mathbb R^3$.
	
	Let $H\subset\mathcal K$ be the wall passing through $\theta_1$ as in the boundary circle of Figure~\ref{fig:5} left. We consider the deconing of the arrangement $\mathcal K$ with respect to $H$. This gives us a planar arrangement $\mathcal K'$, depicted in Figure~\ref{fig:5} right.
\end{definition}

\begin{figure}[h]
	\centering
	\includegraphics[scale=0.7]{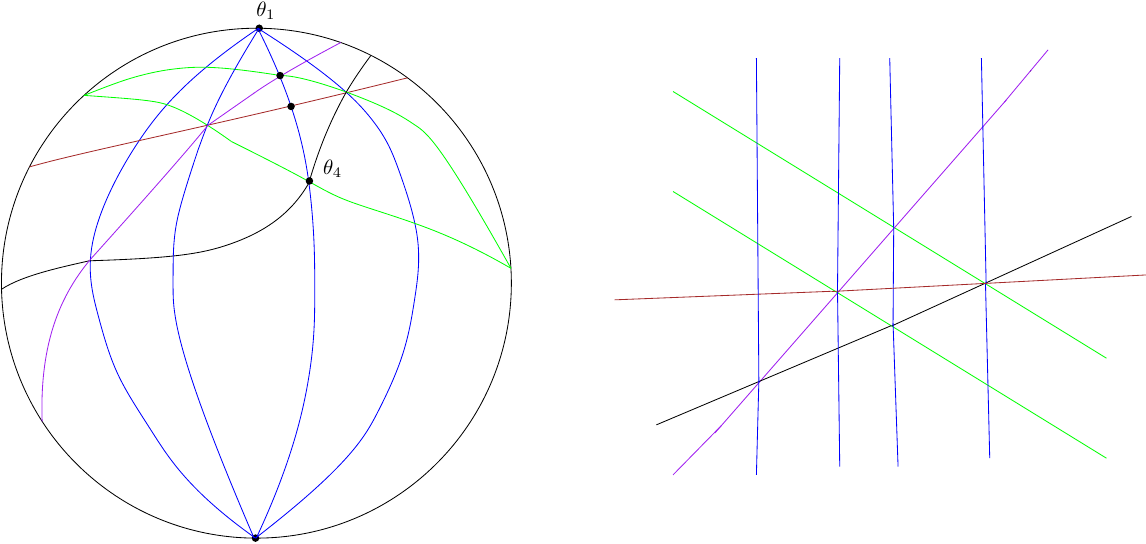}
	\caption{Auxiliary sub-arrangement II.}
		\label{fig:5}
\end{figure}	

Let $\Si_{\mathcal K'}$ (resp. $\Si_{\mathcal K}$) be the cell complex dual to $\mathcal K'$ (resp. $\mathcal K$), see Figure~\ref{fig:2} left. Let $\widehat \Si_{\mathcal K'}$ (resp. $\widehat \Si_{\mathcal K}$) be the associated Salvetti complex. Let $\sX=\Sigma_{\mathcal K'}$ and $\wsX=\widehat \Sigma_{\mathcal K'}$. We now define a collection of subcomplexes of $\sX$ and $\wsX$.

We refer to Figure~\ref{fig:6}. Denote the four vertical walls of $\mathcal K'$ by $h_1,h_2,h_3,h_4$ (from left to right). Let $\sX_i$ be the union of all closed cell of $X$ which has non-trivial intersection with $h_i$. Let $\sX_{31}$ be the hexagonal face in the top of $\sX_3$. We also define $\sX_{22}, \sX_{33}$ and $\sX_{42}$ be the face as in Figure~\ref{fig:6}.

Let $\sY$ be the union of $\sX_{31}$ together with three square faces sharing an edge with $\sX_{31}$ (see the shaded part in Figure~\ref{fig:6}). Define $\sX'=\cup_{i=2}^4\sX_i$.
Let $\wsX_i$, $\wsX_{31}$, $\wsY$, $\wsX'$ be the associated subcomplex of $\wsX$. Note that there is an embedding $\sY\to \Si_{\ca}$ by first embedding $\sY\subset \sX$ into $\Si_{\mathcal K}$ as a maximal subcomplex in one side of $H\subset \mathcal K$, and lifting the image of $\sY$ in $\Si_{\mathcal K}$ with respect to $\kappa:\Si_\ca\to \Si_{\mathcal K}$ in the same way as in Section~\ref{subsec:aug1}. Note that the face $\sX_{31}$ is sent to the face of $\Si_\mathcal{K}$ dual to $\theta_2$. This gives an embedding $i:\wsY\to \od_\ca$. The next lemma can be proved in the same way as Lemma~\ref{lem:BS}.

\begin{lem}
	\label{lem:BSinjective}
The group $\pi_1 \wsX$ splits as a graph of groups whose underlying graph is a linear graph with 4 vertices. The vertex groups (from left to right) are $\pi_1 \wsX_1,\pi_1 \wsX_2$, $\pi_1 \wsX_3$ and $\pi_1 \wsX_4$. In particular, the inclusion $\wsX'\to\wsX$ is $\pi_1$-injective.
\end{lem}

The proof of the next lemma is  identical to Lemma~\ref{lem:transfer1}, using Lemma~\ref{lem:BSinjective}.
\begin{lem}
	\label{lem:embed}
	Given a loop $\omega$ in $\wsY$ such that $i(\omega)$ is null-homotopic in $\od_\ca$, then $\omega$ is null-homotopic in $\wsX'$.
\end{lem}

\begin{figure}[h]
	\centering
	\includegraphics[scale=0.68]{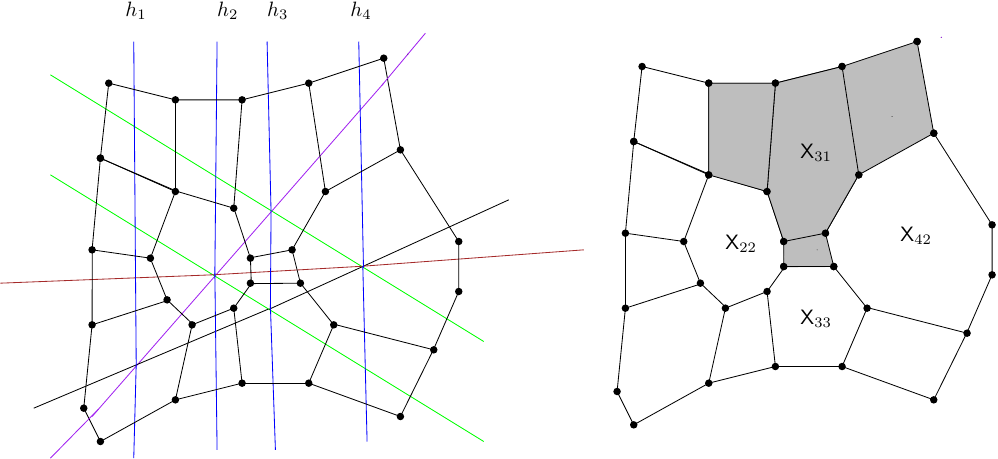}
	\caption{Dual complex.}
		\label{fig:6}
\end{figure}

Let $W_{234}$ be the Coxeter group of type $B_3$, and let $\ca_{234}$ be the arrangement of reflection hyperplanes in $\mathbb R^3$ coming from the canonical representation of $W_{234}$. Namely $\ca_{234}$ is made of the following hyperplanes: $x_i=0$ for $1\le i\le 3$, and $x_i\pm x_j=0$ for $1\le i\neq j\le 3$.
Let $\Si_{234}$ and $\od_{234}$ be the associated dual polyhedron and Salvetti complex. Note that there is an embedding $e:\sX'\to\Si_{234}$ whose image is the maximal subcomplex of $\Si_{234}$ in one side of the hyperplane $x_1=0$ of $\ca_{234}$. Thus there is an induced embedding $\hat e:\wsX'\to \widehat \Si_{234}$ be the associated embedding.

\section{One side flagness in $H_3$}
\label{sec:H3}

The main goal of this section is to prove the following.
\begin{thm}
	\label{thm:tripleH}
	Let $\Lambda$ be the Coxeter diagram of type $H_3$ with the set of nodes $\{a,b,c\}$ such that $m_{ab}=5$, $m_{bc}=3$.
	Let $\Delta$ be the Artin complex of the Artin group $A_\Lambda$ with the order on its vertex induced from $a<b<c$ (see Definition~\ref{def:order}). Then $\Delta$ is downward flag.
\end{thm}

\begin{proof}
	For $i\in\mathbb Z/3\mathbb Z$, let $x_i$ be a lower bound of $y_i$ and $y_{i+1}$. We assume $\{x_1,x_2,x_3\}$ are pairwise distinct, otherwise the theorem is clear.
	If one of $\{y_1,y_2,y_3\}$, say $y_1$, is of type $\hat a$, then we must have $y_1\le y_2$ and $y_1\le y_3$, and the theorem is clear. So we assume none of	$\{y_1,y_2,y_3\}$ is of type $\hat a$. 
	
	If each of $\{y_1,y_2,y_3\}$ is of type $\hat c$, then we can each pair of them has a lower bound which is a type $\hat a$ vertex (if a pair has a type $\hat b$ vertex as lower bound, then we can always find a type $\hat a$ vertex which is $\le$ any given type $\hat b$ vertex). Then the theorem follows immediately from Lemma~\ref{lem:triple} below.
	
	If exactly two of $\{y_1,y_2,y_3\}$ is of type $\hat c$. We assume without loss of generality that $y_1$ is of type $\hat b$. As before, we assume each $x_i$ is of type $\hat a$. 
	Let $y'_1$ be a vertex of type $\hat c$ such that $y_1\le y'_1$.  Lemma~\ref{lem:triple} gives a vertex $s$ which is a common lower bound of $\{y'_1,y_2,y_3\}$. We can assume $s$ is of type $\hat a$. If $s$ is one of $x_1,x_3$, then $s$ is common lower bound of $\{y_1,y_2,y_3\}$. It remains to consider the case $s\neq x_1$ and $s\neq x_3$. Note that $y'_1,s,y_3,x_3$ form a 4-cycle in $\widehat\Sigma$ made of type $\hat a$ and type $\hat c$ vertices. By Theorem~\ref{thm:bowtie free}, there exists a type $\hat b$ vertex $z_3$ such that $z_3$ is adjacent to each of $\{y'_1,s,y_3,x_3\}$. Similarly, there exists a type $\hat b$ vertex $z_1$ such that $z_1$ is adjacent to each of $\{x_1,y_2,s,y'_1\}$. Now we consider the 6-cycle $$x_3\to z_3\to s\to z_1\to x_1\to y_1\to x_3.$$ Each vertex in this 6-cycle is adjacent to $y'_1$. However, as $y'_1$ is of type $\hat c$, we know from \cite[Lemma 39]{crisp2005automorphisms} that the girth of Lk$(y'_1,\Delta)$ is $10$. Thus the image of this 6-cycle in Lk$(y'_1,\Delta)$ is a tree. On the other hand, as we assume $\{x_1,s,x_3\}$ is pairwise distinct, the only possibility $z_1=z_3=y_1$. Thus $s$ is adjacent to $y_1$, hence $s\le y_1$ and $s$ is a lower bound for $\{y_1,y_2,y_3\}$.
	
	If exactly one of $\{y_1,y_2,y_3\}$ is of type $\hat c$. We assume without loss of generality that $y_1$ and $y_2$ are of type $\hat b$. As before, we assume each $x_i$ is of type $\hat a$. 
	For $i=1,2$, let $y'_i$ be a vertex of type $\hat c$ such that $y_i\le y'_i$.  Lemma~\ref{lem:triple} gives a vertex $s$ which is a common lower bound of $\{y'_1,y'_2,y_3\}$. We can assume $s$ is of type $\hat a$. If $s=x_1$, then $s$ is common lower bound of $\{y_1,y_2,y_3\}$. If $s\neq x_i$ for $1\le i\le 3$, then the argument in the previous paragraph implies that $s$ is adjacent to both $y_1$ and $y_2$, hence $s$ is a lower bound for $\{y_1,y_2,y_3\}$.
	It remains to consider the case $s\neq x_1,x_2$ but $s=x_3$, and the case $s\neq x_1,x_3$ but $s=x_2$. This two cases are symmetric, so we only look at the case $s\neq x_1,x_2$ but $s=x_3$. By the argument in the previous paragraph, $s$ is adjacent to $y_2$, so $s\le y_2$. As $x_3\le y_1$, we know $s\le y_1$. Thus $s$ is common lower bound of $\{y_1,y_2,y_3\}$.
	
	If each of $\{y_1,y_2,y_3\}$ is of type $\hat b$. Let $y'_i$ be a vertex of type $\hat c$ such that $y_i\le y'_i$. Lemma~\ref{lem:triple} gives a vertex $s$ which is a common lower bound of $\{y'_1,y'_2,y'_3\}$. We can assume $s$ is of type $\hat a$. If $s\notin \{x_1,x_2,x_3\}$, as $\{x_1,x_2,x_3\}$ is pairwise distinct, then the previous argument implies that $s$ is adjacent to each of $\{y_1,y_2,y_3\}$, hence $s$ is a lower bound for them. If $s\in \{x_1,x_2,x_3\}$, then we assume without loss of generality that $s=x_3$. As $x_3$ is a lower bound for $y_3$ and $y_1$, we know $s\le y_1$ and $s\le y_3$. As $s\neq x_1$ and $s\neq x_2$, the previous argument implies that $s$ is adjacent to $y_2$, hence $s\le y_2$. Thus $s$ is a lower bound for $\{y_1,y_2,y_3\}$.
\end{proof}

We need a procedure of converting a $n$-cycle in the Artin complex $\Delta=\Delta_\Lambda$ of $A_\Lambda$ to a concatenation of $n$ words in $A_\Lambda$ (cf. \cite[Definition 6.14]{huang2023labeled}). These $n$ words are well-defined up to an appropriate notion of equivalence.

\begin{definition}
	\label{def:ncycle}
	A \emph{chamber} in $\Delta$ is a top-dimensional simplex in $\Delta$. There is a 1-1 correspondence between chambers in $\Delta$ and elements in $A_\Lambda$.	
	Let $\{x_n\}_{i=1}^4$ be consecutive vertices of an $n$-cycle $\omega$ in $\Delta$ and suppose $x_i$ has type $\hat a_i$ with $a_i\in \Lambda$. 
	For each edge of $\omega$, take a chamber of $\Delta$ containing this edge. We name these chambers by $\{\Theta_i\}_{i=1}^n$ with $\Theta_1$ containing the edge $\overline{x_1x_2}$. Each $\Theta_i$ gives an element $g_i\in A_\Lambda$. Then for $i\in \mathbb Z/n\mathbb Z$, $g_i=g_{i-1}w_{i}$ for $w_i\in A_{\hat a_i}$ (recall that $A_{\hat a_i}$ is defined to be $A_{S\setminus\{a_i\}}$). Thus $w_1w_2\cdots w_n=1$.
	The word $w_1\cdots w_n$ depends on the choice of $\{\Theta_i\}_{i=1}^n$. A different choice would lead to a word of form $u_1\cdots u_n$ such that there exist elements $q_i\in A_{S\setminus\{a_i,a_{i+1}\}}$ such that $u_i=q^{-1}_{i-1}w_i q_i$ for $i\in\mathbb Z/n\mathbb Z$. In this case we will say the words $u_1\cdots u_n$ and $w_1\cdots w_n$ are equivalent. If in addition there exists a parabolic subgroup $A'$ of $A_\Lambda$ such that $q_i\in A_{S\setminus\{a_i,a_{i+1}\}}\cap A'$, then we say $w_1\cdots w_n$ is \emph{equivalent} to $u_1\cdots u_n$ in $A'$.
\end{definition}

\begin{lem}
	\label{lem:triple}
	Let $\Delta$ be as in Theorem~\ref{thm:tripleH}.
	Take vertices $\{x_1,x_2,x_3\}\in \Delta$ of type $\hat a$ and $\{y_1,y_2,y_3\}\in\Delta$ of type $\hat c$ such that $x_i\le y_{i}$ and $x_{i+1}\le y_{i}$ for all $i\in\mathbb Z/3\mathbb Z$. We assume the edge loop in $\Delta$ formed by $\omega=y_1x_1y_2x_2y_3x_3y_1$ is a local embedding. Then there is a vertex $s\in \Delta$ such that $s$ is adjacent to each of $\{y_1,y_2,y_3\}$.
\end{lem}

In the following proof we will need a simple version of combinatorial Gauss-Bonnet formula. We refer to \cite[Section 3.6]{huang2023labeled} for a quick summary.
\begin{proof}
	From Definition~\ref{def:ncycle}, the 6-cycle $\omega$ gives a word $$w=w_{ab}w_{bc}w'_{ab}w'_{bc}w''_{ab}w''_{bc},$$ where the subword $w_{ab}$ is a word only using $a$ and $b$ (similar constraints applies to other subwords of $w$). Moreover, $w$ represents the trivial element in $A_\Lambda$. Let $\Si$ be the polyhedron dual to the reflection arrangement of type $H_3$ in $\mathbb R^3$, and $\od$ be the associated Salvetti complex.
	This word gives a null-homotopic path in $\od$ as follows: $$P=P_{ab}P_{bc}P'_{ab}P'_{bc}P''_{ab}P''_{bc}.$$  We denote the sub-segments of this path by $\{P_i\}_{i=1}^6$. Let $\bC$ be the Coxeter complex of type $H_3$.
	Let $\pi:\Delta\to\bC$ be the map induced by the action of the pure Artin group on $\Delta$.
	
	\medskip
	
	\noindent	
	\underline{Case 1: the $\pi$-image of the 6-cycle $\omega$ is a single edge $\bar x\bar y$ in $\bC$.} Let $C_{\bar x}$ (resp. $C_{\bar y}$) be the 2-cell in $\Si$ dual to $\bar x$ (resp. $\bar y$). Let $\widehat C_{\bar x}$ and $\widehat C_{\bar y}$ be the associated standard subcomplexes of $\od$. Up to replacing $w=\Pi_{i=1}^6w_i$ by an equivalent word (in the sense of Definition~\ref{def:ncycle}), we can assume $P_i$ is a loop in $\widehat C_{\bar x}$ (resp. $\widehat C_{\bar y}$) for $i$ odd (resp. for $i$ even). Let $\wZ$ be the complex defined before Lemma~\ref{lem:transfer1}. Recall from Section~\ref{subsec:aug1} that there is an embedding $\wZ\to \widehat\Sigma$. Thus we can also view $\wZ$ as a subcomplex of $\od$. Recall that edges of $\od$ are oriented and labeled by $\{a,b,c\}$ (Definition~\ref{def:support}), this also gives label and orientation of edges in $\wZ$ via the inclusion of $\wZ$ into $\widehat\Sigma$.
	Up to a symmetry of $\widehat\Sigma$, we can assume $$P\subset\widehat C_{\bar x}\cup\whC_{\bar y}\subset\wZ.$$ Moreover, we can assume $\whC_{\bar x}=\whX_{11}$ and $\whC_{\bar y}=\whX_{22}$. As we are also viewing $\whZ$ as a subcomplex of $\wX_1\cup\wX_2$, by Lemma~\ref{lem:transfer1}, $P$ is null-homotopic in $\wX_1\cup\wX_2$. Assume
	\begin{center}
	 $P_1,P_3,P_5\subset \wX_{22}$ and $P_2,P_4,P_6\subset \wX_{11}$.
	\end{center}
	  As $P=\cup_{i=1}^6 P_i$ is null-homotopic in $X$, it lifts to a loop $\wtP=\cup_{i=1}^6 \wtP_i$ in the universal cover $\wtX$ of $\wX_1\cup\wX_2$. A \emph{standard subcomplex} of $\wX_1\cup\wX_2$ is an intersection of $\{\whX_{ij}\}_{1\le i,j\le 2}$. A \emph{standard subcomplex} of $\wtX$ (of type $\whX_{ij}$) is a lift of $\whX_{ij}$ in $\wtX$. 
	
	Let $\mathbb U$ be the development of the complex of group $\mathcal U$, endowed with the CAT(0) metric as in Section~\ref{subsec:aug1}. For a vertex $z\in \mathbb U$, denote the subcomplex of $\wtX$ associated with $z$ by $\wtX_z$.
	The loop $\wtP$ gives a loop $\omega_{\mathbb U}$ with consecutive vertices denoted by $\{z_i\}_{i\in \mathbb Z/6\mathbb Z}$ in $\mathbb U$ where $\wtX_{z_1},\wtX_{z_3},\wtX_{z_5}$ (resp. $\wtX_{z_2},\wtX_{z_4},\wtX_{z_6}$) correspond to the lifts of $\wX_{22}$ (resp. $\wX_{11}$) in $\wtX$ that contain $\wtP_1,\wtP_3$ and $\wtP_5$ respectively (resp. $\wtP_2,\wtP_4$ and $\wtP_6$). 
	
	Let $\bD\to \bU$ be a minimal area singular disk diagram (\cite[Section 3.6]{huang2023labeled}) for the loop $\omega_\bU$. We will slightly abuse notation and use $z_i$ to denote the point in the boundary cycle of $\bD$ mapping to $z_i\in\bU$. The metric on $\bU$ induces a metric on $\bD$, which is CAT$(0)$. For three points $x,y,z$ in a CAT$(0)$ space, we use $\angle_y(x,z)$ to denote the Alexandrov between the geodesic segments $\overline{yx}$ and $\overline{yz}$ (cf. \cite[Chapter II.3.1]{BridsonHaefliger1999}).
	
	In $\bU$ we have $\angle_{z_i}(z_{i-1},z_{i+1})\ge \pi/2$ for each $i\in\mathbb Z/6\mathbb Z$. This is because that $\omega$ is a local embedding, hence $\omega_{\mathbb U}$ is a local embedding. Thus 
	\begin{center}
$\angle_{z_i}(z_{i-1},z_{i+1})$ is a multiple of $\pi/2$ in $\bD$.
	\end{center}
For each vertex $z_i$ on the boundary cycle $\omega_\bU$ of $\bD$, the \emph{interval angle} at $z_i$ is the quantity $\alpha(z_i)$ defined in \cite[Section 3.6]{huang2023labeled}.
	By the combinatorial Gauss-Bonnet formula in \cite[Equation 3.17]{huang2023labeled}, at least four internal angles of $\omega_{\mathbb U}$ in $\mathbb D$ is $\pi/2$.

	If at least five internal angles of $\omega_{\mathbb U}$ in $\mathbb D$ is $\pi/2$, then $\mathbb D$ has exactly one interior vertex of type $\wX_{12}$ or $\wX_{21}$. 
	We only discuss the case where the internal vertex is of type $\wX_{12}$, as the other case is similar and easier. Let $T\subset\wtX$ be the lift of $\wX_{12}$ corresponding to this internal vertex. Then each $\wtX_{z_i}$ has non-empty intersection with $T$. For each $i\in\mathbb Z/6\mathbb Z$, let $t_i$ be the terminal point of $\wtP_i$ (which is the starting point of $\wtP_{i+1}$) and let $$t'_i=\wtX_{z_i}\cap \wtX_{z_{i+1}}\cap T.$$ Such triple intersection is nonempty because they corresponds to three vertices of 2-face in $\mathbb U$.
	Set $\wtQ_i$ to be an edge path in $T\cap \wtX_{z_i}$ from $t_{i-1}$ to $t_i$. Let $\Theta_i$ be a path in $\wtX_{z_i}\cap \wtX_{z_{i+1}}$ from $t_i$ to $t'_i$. Note that $\wtQ_i$ and $\Theta_i$ exist by Lemma~\ref{lem:connected}.
	Then the loop $$(\Theta^{-1}_6\wtP_1\Theta_1)(\Theta^{-1}_1\wtP_2\Theta_2)\cdots (\Theta^{-1}_5\wtP_6\Theta_6)$$ gives a word which is equivalent to $w$. As $\Theta^{-1}_{i-1}\wtP_i\Theta_i$ is homotopic to $\wtQ_i$ rel endpoints in $\wtX_{z_i}$, we know the word traced out by $\Theta^{-1}_{i-1}\wtP_i\Theta_i$ and the word traced out by $\wtQ_i$ represent the same element in the group $A_{ab}$ (i.e. the subgroup of $A$ generated by $a$ and $b$) or $A_{bc}$. Thus we assume instead that $\wtQ_i$ traces out the word $w_i$, and replace $\widetilde P$ by $\wtQ$, which is the concatenation of all the $\wtQ_i$.
	As $\wtQ_i\subset T\cap \wtX_{z_i}$ and by Lemma~\ref{lem:connected}, $T\cap \wtX_{z_i}$ is a lift of $\whX_{12}\cap \whX_{11}$ or $\whX_{12}\cap \whX_{22}$, we deduce that $w_i\in A_{ab}$ for $i$ odd, and $w_i$ is a power of $c$ for $i$ even. Moreover, $w$ still represents the trivial element in $A_{\Lambda}$.
	
	The new loop $\wtQ$ has the advantage that its projection $Q$ to $\whZ$ lies inside a smaller subcomplex, i.e. $\wtQ\subset \wX_{12}\cap (\wtX_{11}\cup \wtX_{22})$. By construction, $Q\subset \wZ$ starts and ends at $\whX_{11}\cap \whX_{12}\cap \whX_{22}$. Note that there is a unique loop $R\subset \wZ$ starting and ending at $\whX_{11}\cap \whX_{21}\cap \whX_{22}$ such that $R$ traces out the same word as $Q$. As $w$ represents the trivial element in $A_\Lambda$, Lemma~\ref{lem:transfer1} implies that $R$ is null-homotopic in $\wX_1\cup \wX_2$. Thus we will work with $R$ instead of $Q$ now. Note that $R_i\subset \wX_{22}$ for $i=1,3,5$ and $R_i\subset \wtX_{11}\cap \wtX_{21}$ for $i=2,4,6$.
	
	We lift $R$ to be a loop $\wtR\subset \wtX$. 
Let $T'_i$ be the standard subcomplex of type $\wX_{22}$ (resp. $\wX_{21}$) containing $\wtR_i$ for $i=1,3,5$ (resp. $i=2,4,6$). Let $z'_i\in\mathbb D$ be the vertex associated with $T'_i$. Then $\{z'_i\}_{i=1}^6$ forms consecutive vertices in a cycle $\omega'_{\mathbb U}\subset \mathbb U$. As each $w_i$ is a nonzero power of $c$ for $i$ even, we know for $i$ even $$\angle_{z'_i}(z'_{i-1},z'_{i+1})=\pi.$$ So the $\omega'_{\mathbb U}$ is a geodesic triangle with vertices at $z'_1,z'_3,z'_5$ with three sides having equal length. Moreover, $\angle_{z'_i}(z'_{i-1},z'_{i+1})$ is a multiple of $\pi/2$. As $\mathbb U$ is CAT(0), the only possibility is that the triangle $\Delta(z'_1,z'_3,z'_5)$ is degenerate such that the three sides share a common midpoint, i.e. $z'_2=z'_4=z'_6$. This means $\widetilde R_3$ is a path in $T'_3$ from $T'_2\cap T'_3$ to $T'_4\cap T'_3=T'_2\cap T'_3$. Thus $\wtR_3$ is homotopic in $T'_3$ to a path inside $T'_2\cap T'_3$, which means $w'_{ab}$ and $a^{k}$ represent the same element in $A_{ab}$.  The same is true for $w_{ab}$ and $w''_{ab}$ by a similar argument. Thus each $w_i$ is either a power of $a$ or a power of $c$, and the lemma follows.
	
	\begin{figure}
		\centering
		\includegraphics[scale=1.2]{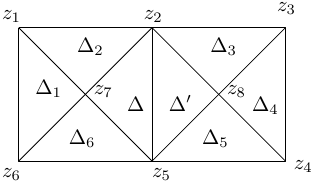}
		\caption{Disk diagram in Case 1.}
			\label{fig:case1}
	\end{figure}
	
	It remains to consider the case exactly four internal angles of $\omega_{\mathbb U}$ in $\mathbb D$ is $\pi/2$. In this case, $\mathbb D$ is a flat rectangle, cf. \cite[Theorem II.2.11]{BridsonHaefliger1999}. Thus $\mathbb D$ must be as in Figure~\ref{fig:case1}. Note that one of the four corners must be of type $\wX_{22}$. Thus up to a symmetry of $\mathbb D$ and a cyclic permutation of $\{z_i\}_{i\in\mathbb Z/6\mathbb Z}$, we can assume the vertices of $\mathbb D$ are labeled as in Figure~\ref{fig:case1}. Let $z_7$ and $z_8$ the two internal vertices of $\mathbb D$. Let $\Delta,\Delta'$ and $\{\Delta_i\}_{i=1}^6$ be the 8 triangles in Figure~\ref{fig:case1}. Each of these triangles corresponds to a vertex in $\wtX$ which is the intersection of the standard subcomplexes corresponding to the three vertices of this triangle. Let $q,q'$ and $\{q_i\}_{i=1}^6$ be the 8 vertices of $\wtX$ corresponding these 8 triangles. By a similar argument as above, up to replacing $w$ by an equivalent word, we can assume each $\wtP_i$ starts with $q_{i}$ and ends with $q_{i+1}$. Note that $q_1,q_2\in \wtX_{z_1}\cap \wtX_{z_7}$, thus $\wtP_1$ is homotopic rel its boundary points in $\wtX_{z_1}$ to a path in $\wtX_{z_1}\cap \wtX_{z_7}$ by Lemma~\ref{lem:connected}. Thus by replacing $w_{ab}$ by another word which represents the same element in $A_{ab}$, we can assume $\wtP_1\subset \wtX_{z_1}\cap \wtX_{z_7}$. By a similar argument, we can assume
	\begin{itemize}
		\item $\wtP_2$ (resp. $\wtP_5$) is a concatenation of three subpaths, the first subpath is in $\wtX_{z_2}\cap \wtX_{z_7}$ (resp. $\wtX_{z_5}\cap \wtX_{z_7}$), the second subpath is in $\wtX_{z_2}\cap \wtX_{z_5}$ (resp. $\wtX_{z_5}\cap \wtX_{z_2}$), and the third subpath is in $\wtX_{z_2}\cap \wtX_{z_8}$ (resp. $\wtX_{z_5}\cap \wtX_{z_8}$);
		\item for $i=3,4$, $\wtP_i\subset\wtX_{z_i}\cap \wtX_{z_8}$;
		\item $\wtP_6\subset\wtX_{z_6}\cap \wtX_{z_7}$.
	\end{itemize}
	Suppose one of the internal vertices of $\mathbb D$, say $z_8$, is type $\wX_{21}$. Then the above properties of $\wtP$ implies $w_4$ is a power of $c$ and $w_3$ is a power of $a$. Thus by switching the position of $w_4$ and $w_3$ (this is possible as they commute), combining powers of $c$ of $w_4$ with $w_2$, and combining powers of $a$ of $w_2$ with $w_5$, we deduce that $x_1$ and $y_3$ are adjacent. Thus the lemma follows by taking $s=z_1$.
	
	Suppose both of the internal vertices of $\mathbb D$ are of type $\wX_{12}$. Then the above properties of $\wtP$ implies
	\begin{center}
	 $w_{bc}=c^{k_1}b^{k_2}c^{k_3}$, $w'_{bc}=c^{k_4}$ and $w''_{bc}=c^{k_5}$.
	\end{center}
	 Moreover, 
	\begin{center}
	$P_1,P_3,P_5\subset \whX_{22}$, $P_2\subset \whX_{11}\cap (\whX_{12}\cup \whX_{22})$, and $P_4,P_6\subset \whX_{11}\cap \whX_{12}$. 
	\end{center}
	Also we know $P$ is a concatenation of six loops based at $\whX_{11}\cap \whX_{12}\cap \whX_{22}$. Then there is a unique loop $\mathsf P$ in $\wZ$ based at $\whX_{11}\cap \whX_{21}\cap \whX_{22}$ such that this loop trace out exactly the same word $w$ as $P$. More precisely, 
	\begin{center}
	$\mathsf P_1,\mathsf P_3,\mathsf P_5\subset \whX_{22}$, $\mathsf P_2\subset \whX_{11}\cap (\whX_{21}\cup \whX_{22})$, and $\mathsf P_4,\mathsf P_6\subset \whX_{11}\cap \whX_{21}$.
	\end{center}
	As $w$ corresponds to the trivial element in $A_\Lambda$, we know from Lemma~\ref{lem:transfer1} that $\mathsf P$ is null-homotopic in $\wX_1\cup\wX_2$. We can construct a cycle $\omega''_{\mathbb U}$ in $\mathbb U$ from $\mathsf P$ as before, with its vertices denoted by $\{z''_i\}_{i=1}^6$. Moreover, because of the specific form of $w$, we know
	\begin{center}
$\angle_{z''_6}(z''_1,z''_5)=\angle_{z''_4}(z''_5,z''_3)=\pi/2$ and $\angle_{z''_2}(z''_1,z''_3)=\pi$. 
	\end{center} 
It suffices to consider the case when $\omega''_{\mathbb U}$ bounds a flat rectangle in $\mathbb U$ as in Figure~\ref{fig:case1}. If $z_7$ is of type $\whX_{12}$, then by a similar argument as before, we know $\mathsf P_6$ is homotopic rel endpoints in $\whX_{11}$ to $\whX_{11}\cap (\whX_{12}\cup\whX_{22})$, this contradicts that $\mathcal P_6\subset \whX_{11}\cap\whX_{21}$. Thus $z_7$ of type $\whX_{21}$. Similarly we know $z_8$ is of type $\whX_{21}$. This reduces to one of the situations we studied before.
	
	%This implies that $\omega''_{\mathbb U}$ is 
	
	\medskip
	
	\noindent
	\underline{Case 2: the $\pi$-image of $\omega$ is two edges $\bar x\bar y$ and $\bar y\bar z$ such that $\bar y$ is of type $\hat a$.} Then $\bar x$ and $\bar z$ is of type $\hat c$. Up to applying a cyclic permutation to the index $i$ and possibly changing the role of $\bar x$ and $\bar z$, we can assume that $P_1,P_5$ are loops in $\whC_{\bar x}$, $P_2,P_4$ are paths in $\whC_{\bar y}$, $P_3$ is a loop in $\whC_{\bar z}$ and $P_6$ is a loop in $\whC_{\bar y}$.
	
	Note that $\p_{\widehat C_{\bar z}}(\widehat C_{\bar y})$ is a standard subcomplex of type $b$ (see Definition~\ref{def:support}) and $\p_{\widehat C_{\bar z}}(\widehat C_{\bar x})$ is a standard subcomplex of type $a$. Thus if we consider the loop $\p_{\widehat C_{\bar z}}(P)$, then we can read of a word of form $$a^{k_1}b^{k_2}w'_{ab}b^{k_3}a^{k_4}b^{b_5}$$ which represents the trivial element in the group. Thus 
	$$
	w'_{ab}=b^{-k_2}a^{-k_1}b^{-k_5}a^{-k_4}b^{-k_3}.
	$$
	By combining powers of $b$'s at the beginning and end of $w'_{ab}$ with $w_{bc}$ and $w'_{bc}$, we can replace $w$ by an equivalent word such that $$w'_{ab}=a^{-k_1}b^{-k_5}a^{-k_4},$$ moreover, $P_3$ is a loop in  $\p_{\widehat C_{\bar z}}(\widehat C_{\bar y})\cup \p_{\widehat C_{\bar z}}(\widehat C_{\bar x})$. Thus up to applying a symmetry of $\widehat \Si$, we can assume $$P\subset \p_{\widehat C_{\bar z}}(\widehat C_{\bar y})\cup \p_{\widehat C_{\bar z}}(\widehat C_{\bar x})\cup \widehat C_{\bar x}\cup \widehat C_{\bar y}\subset \whZ.$$ Moreover, we can assume 
	\begin{center}
		$\whX_{11}=\whC_{\bar y}$, $\whX_{22}=\whC_{\bar x}$, $\p_{\widehat C_{\bar z}}(\widehat C_{\bar y})=\wE_2$ and $\p_{\widehat C_{\bar z}}(\widehat C_{\bar x})=\wE_4$
	\end{center}
 (see Figure~\ref{fig:2} for the definition of $\wE_2$ and $\wE_4$). As before, we deduce from Lemma~\ref{lem:transfer1} that $P$ is null-homotopic in $\wX_1\cup \wX_2$.
	
	For we assume $k_1\neq 0, k_5\neq 0$ and $k_4\neq 0$.
	We produce a cycle in $\mathbb U$ from $P$ as follows. First we lift $P$ to a loop $\wtP$ in $\wtX$, which is the universal cover of $\wX_1\cup \wX_2$. Let $T_i$ be the unique standard subcomplex of type $\wtX_{22}$ (resp. $\wtX_{11}$) containing $\wtP_i$ for $i=1,5$ (resp. for $i=2,4,6$). Note that $\wtP_3$ is a concatenation of three subsegments $$\wtP_3=\wtP_{31}\wtP_{32}\wtP_{33}$$ such that $\wtP_{31}$ and $\wtP_{33}$ arise from of powers of $a$'s in $w'_{ab}$ and $\wtP_{32}$ corresponds to powers of $b$'s in $w'_{ab}$. Let $T_{3i}$ (resp. $T_{32}$) be the unique standard subcomplex of type $\wtX_{21}$ (resp. $\wtX_{11}$) containing $\wtP_{3i}$ for $i=1,3$ (resp. for $i=2$). This gives a cycle in $\omega_{\mathbb U}$  of form
	$$
	z_1\to z_2\to z_{31}\to z_{32}\to z_{33}\to z_4\to z_5\to z_6\to z_1
	$$
	where $z_i$ (resp. $z_{3i}$) corresponds to $T_i$ (resp. $T_{3i}$).
	Let $\mathbb D$ be a minimal area singular disk diagram with boundary $\omega_{\mathbb U}$. As before we will slightly abuse notation and use $z_i$ to denote the point in the boundary cycle of $\bD$ mapping to $z_i\in\bU$. 
Since
\begin{center}
 $k_1\neq 0$, $k_5\neq 0$ and $k_4\neq 0$,
\end{center}
 we have
 \begin{center}
  $\angle_{z_{31}}(z_2,z_{32})=\angle_{z_{32}}(z_{31},z_{33})=\angle_{z_{33}}(z_{32},z_4)=\pi$ in $\mathbb U$,
 \end{center}
 so the same holds in $\mathbb D$. 
	\begin{enumerate}
		\item If $\angle_{z_2}(z_1,z_{31})\ge 3\pi/4$ and $\angle_{z_4}(z_5,z_{33})\ge 3\pi/4$ in $\mathbb D$, then from \cite[Equation 3.17]{huang2023labeled} we know all the other angles at the vertices of $\omega_{\mathbb U}$ in $\mathbb D$ has to be $\pi/2$. As $\mathbb D$ is made of right-angled isosceles triangles, by consider the sequence of right-angled turns in the edge path $z_4\to z_5\to z_6\to z_1\to z_2$, we must have $z_2=z_4$, which is impossible as $z_2\to z_{31}\to z_{32}\to z_{33}\to z_4$ is a geodesic in $\mathbb U$.
		\item If $\angle_{z_2}(z_1,z_{31})= \pi/4$ and $\angle_{z_4}(z_5,z_{33})\ge 3\pi/4$ in $\mathbb D$, then $z_2,z_{31},z_1$ form the vertices of a triangle with right-angled at $z_{31}$. Thus $z_1$ is adjacent to $z_{32}$. Now we are reduced to consider a new cycle $\omega'_{\mathbb U}$ of form $z_1\to z_{32}\to z_{33}\to z_4\to z_5\to z_6\to z_1$ and a disk diagram $\mathbb D'\subset\mathbb D$ bounded by $\omega'_{\mathbb U}$. Note that $\angle_{z_{32}}(z_1,z_{33})=3\pi/4$ in $\mathbb D'$. Thus by \cite[Equation 3.17]{huang2023labeled}, all other angles at the vertices of $\omega'_{\mathbb U}$ in $\mathbb D'$ is $\pi/2$. We deduce as before that $z_{32}=z_4$, which is impossible as $z_{32}\to z_{33}\to z_4$ is a geodesic in $\mathbb U$.
		\item Suppose $\angle_{z_2}(z_1,z_{31})= \pi/4$ and $\angle_{z_4}(z_5,z_{33})=\pi/4$ in $\mathbb D$. From $\angle_{z_2}(z_1,z_{31})= \pi/4$ we deduce that $T_1,T_2$ and $T_{31}$ has pairwise nonempty intersection. Then $T_1\cap T_2\cap T_{31}\neq\emptyset$ as $z_1,z_2,z_{31}$ form a triangle which bounds a 2-cell in $\mathbb U$. Then $T_1\cap T_2\cap T_{31}$ is a single point, which we denote by $z$.	As $T_1\cap T_2$ is a $b$-line (i.e. a copy of $\mathbb R$ made of edges which are mapped to some $b$-labeled edges in $\whZ$ under the covering map) containing $z$, and $T_2\cap T_{31}$ is a $c$-line containing $z$. Thus $P_2$ is a path in $T_2$ from a point in the $b$-line containing $z$ to a point in the $c$-lines containing $z$. Thus we can assume $w_{bc}=b^{k_7}c^{k_8}$. Similarly, we can assume $w'_{bc}=c^{k_9}b^{k_{10}}$. Then the word $w$ becomes:
	\begin{equation}
		\label{eq:word1}
		w_{ab}\cdot b^{k_7}c^{k_8}\cdot a^{-k_1}b^{-k_5}a^{-k_4}\cdot c^{k_9}b^{k_{10}}\cdot w''_{ab}w''_{bc},
	\end{equation}
		which can be rearranged as 
		\begin{equation}
			\label{eq:word2}
		w_{ab}b^{k_7}a^{-k_1}\cdot c^{k_8}b^{-k_5}c^{k_9}\cdot a^{-k_4}b^{k_{10}}w''_{ab}\cdot w''_{bc}.
		\end{equation}
	Consider a lift $\widetilde P'$ of $P$ to the universal cover of $\od$, then this rearrangement gives a way of replacing  $\widetilde P'$ by another loop $\widetilde P''$ with the same endpoint. Let $\widetilde Z$ be the standard subcomplex of type $\{b,c\}$ (cf. Definition~\ref{def:support}) containing subpath of $\widetilde P''$ corresponding to the subword $c^{k_8}b^{-k_5}c^{k_9}$ of the word \eqref{eq:word2}.
	By keeping track of the replacement, we deduce that for $i=1,3,5$, $\widetilde Z$ has nonempty intersection with the standard subcomplex of type $\{a,b\}$ containing $\widetilde P'_i$. This implies that in $\omega=y_1x_1y_2x_2y_3x_3$ which we started with, $y_1,y_3$ and $y_5$ are adjacent to a common vertex of type $\hat c$.
	\end{enumerate}
	
	Now we assume at least one of $k_1,k_5$ or $k_4$ is $0$. Then we can assume $w'_{ab}$ is a power of $a$, up to combining powers of $b$'s in $w'_{ab}$ with $w_{bc}$ or $w'_{bc}$. This is quite similar to the previous case, except in the cycle $\omega_{\mathbb U}$, we have $z_{31}=z_{33}$. Thus we set $z_3=z_{31}$, which is the vertex corresponding to the standard subcomplex of type $\wtX_{21}$ containing $\wtP_3$, and we define $\omega_{\mathbb U}$ to be $z_1z_2z_3z_4z_5z_6z_1$ in this case. Let $\mathbb D$ be a minimal area disk diagram for $\omega_{\mathbb U}$. Then $\angle_{z_3}(z_2,z_4)=\pi$ in $\mathbb D$. Note that $\angle_{z_2}(z_1,z_3)$ is either $\pi/4$ or $\ge 3\pi/4$. The same is true for $\angle_{z_4}(z_3,z_5)$. If both $\angle_{z_2}(z_1,z_3)$ and $\angle_{z_4}(z_3,z_5)$ is $\ge 3\pi/4$, then we can deduce a contradiction as before. If $\angle_{z_2}(z_1,z_3)=\pi/4$, then we deduce as before that $w_{bc}=b^{k_7}c^{k_8}$. Thus $w$ is of form $$w_{ab}\cdot b^{k_7}c^{k_8}\cdot a^{k_9}\cdot w'_{bc}w''_{ab}w''_{bc},$$ which can be rearranged as $$w_{ab}b^{k_7}a^{k_9}\cdot c^{k_8}w'_{bc}\cdot w''_{ab}w''_{bc}.$$ 
	By comparing the two loops in the universal cover of $\od$ corresponding to these two words 
	and reasoning as in the previous paragraph, we know $x_2$ is adjacent to $y_1$ in $\Delta$ and we can take $s=x_2$ in the lemma.
	
	\medskip
	
	\noindent
	\underline{Case 3: the $\pi$-image of $\omega$ is two edges $\bar x\bar y$ and $\bar y\bar z$ such that $\bar y$ is of type $\hat c$.} Then $\bar x$ and $\bar z$ are of type $\hat a$. Up to applying a cyclic permutation to the index $i$ and possibly changing the role of $\bar x$ and $\bar z$, we can assume that $P_6,P_4$ are loops in $\whC_{\bar x}$, $P_1,P_3$ are paths in $\whC_{\bar y}$, $P_2$ is a loop in $\whC_{\bar z}$ and $P_5$ is a loop in $\whC_{\bar y}$.
	
	\smallskip
	
	\noindent
	\underline{Case 3.1: there is a vertex in $\whC_{\bar x}\cap \whC_{\bar y}$ and a vertex in $\whC_{\bar y}\cap \whC_{\bar z}$ that are adjacent.}\\
	Note that $\p_{\widehat C_{\bar z}}(\widehat C_{\bar y})$ is a standard subcomplex of type $b$ (see Definition~\ref{def:support}) and $\p_{\widehat C_{\bar z}}(\widehat C_{\bar x})$ is a standard subcomplex of type $c$. Thus if we consider the loop $\p_{\widehat C_{\bar z}}(P)$, then we can read of a word of form $b^{k_2}w_{bc}b^{k_3}c^{k_4}b^{b_5}c^{k_1}$ which represents the trivial element in $A_{bc}$. Thus 
	$$
	w_{bc}=b^{-k_2}c^{-k_1}b^{-k_5}c^{-k_4}b^{-k_3}.
	$$
	By combining powers of $b$'s at the beginning and end of $w'_{bc}$ with $w_{ab}$ and $w'_{ab}$, we can replace $w$ by an equivalent word such that $$w_{bc}=c^{-k_1}b^{-k_5}c^{-k_4},$$ moreover, $P_2$ is a loop in  $\p_{\widehat C_{\bar z}}(\widehat C_{\bar y})\cup \p_{\widehat C_{\bar z}}(\widehat C_{\bar x})$. Thus up to applying a symmetry of $\widehat \Si$, we can assume $$P\subset \p_{\widehat C_{\bar z}}(\widehat C_{\bar y})\cup \p_{\widehat C_{\bar z}}(\widehat C_{\bar x})\cup \widehat C_{\bar x}\cup \widehat C_{\bar y}\subset \whZ.$$ Moreover, we can assume $\whX_{11}=\whC_{\bar x}$, $\whX_{22}=\whC_{\bar y}$, $\p_{\widehat C_{\bar z}}(\widehat C_{\bar y})=\wE_6$ and $\p_{\widehat C_{\bar z}}(\widehat C_{\bar x})=\wE_5$. 
	
First we assume $k_1\neq 0, k_5\neq 0$ and $k_4\neq 0$.
	We produce a cycle in $\mathbb U$ from $P$ as follows. First we lift $P$ to a loop $\wtP$ in $\wtX$. Let $T_i$ be the unique standard subcomplex of type $\wtX_{22}$ (resp. $\wtX_{11}$) containing $\wtP_i$ for $i=1,3,5$ (resp. for $i=6,4$). Note that $\wtP_2$ is a concatenation of three subsegments $\wtP_2=\wtP_{21}\wtP_{22}\wtP_{23}$ such that $\wtP_{21}$ and $\wtP_{23}$ arise from of powers of $c$'s in $w_{bc}$ and $\wtP_{22}$ corresponds to powers of $b$'s in $w_{bc}$. Let $T_{2i}$ (resp. $T_{22}$) be the unique standard subcomplex of type $\wtX_{21}$ (resp. $\wtX_{11}$) containing $\wtP_{2i}$ for $i=1,3$ (resp. for $i=2$). This gives a cycle in $\omega_{\mathbb U}$  of form
	$$
	z_6\to z_1\to z_{21}\to z_{22}\to z_{23}\to z_3\to z_4\to z_5\to z_6
	$$
	where $z_i$ (resp. $z_{2i}$) corresponds to $T_i$ (resp. $T_{2i}$). Now the rest of the argument is quite similar to Case 2.
	
	\smallskip
	\noindent
	\underline{Case 3.2: $\whC_{\bar x}\cap \whC_{\bar y}$ and $\whC_{\bar y}\cap \whC_{\bar z}$ do not contain adjacent vertices.}
	Then $\p_{\widehat C_{\bar z}}(\widehat C_{\bar y})$ is still a standard subcomplex of type $b$, but $\p_{\widehat C_{\bar z}}(\widehat C_{\bar x})$ is a single point, which is the main difference from Case 3.1. Thus if we consider the loop $\p_{\widehat C_{\bar z}}(P)$, then we can read of a word of form $b^{k_1}w'_{bc}b^{k_2}b^{b_3}$ which represents the trivial element in $A_{bc}$. Thus $w_{bc}=b^{k_4}$ for some integer $k_4$, which means $y_1=y_2$ and it contradicts the assumption that $\omega$ is a local embedding in $\od$.
	
	\medskip
	\noindent
	\underline{Case 4: the $\pi$-image of $\omega$ is a linear subgraph with three edges.} 
	Let $\overline x_1,\overline x_2,\overline x_3,\overline x_4$ be consecutive vertices of $\pi(\omega)$. Let $\whC_i$ be the standard subcomplex of $\widehat \Sigma$ associated with $\overline x_i$.
	We can assume without loss of generality that 
	\begin{center}
		$P_i\subset \whC_i$ for $1\le i\le 4$ and $P_i\subset \whC_{8-i}$ for $i=5,6$.
	\end{center}
 Moreover, $P_1$ and $P_4$ are loops. We define $\whC_{ij}=\Pi_{\whC_i}\whC_j$.
	We assume without loss of generality that $\bar x_1$ is of type $\hat c$ (otherwise we can switch the role of $\bar x_1$ and $\bar x_4$). There are four subcases to consider depending on the shape of $\pi(\omega)$, see Figure~\ref{fig:4}.
	
	\begin{figure}[h]
		\centering
		\includegraphics[scale=0.6]{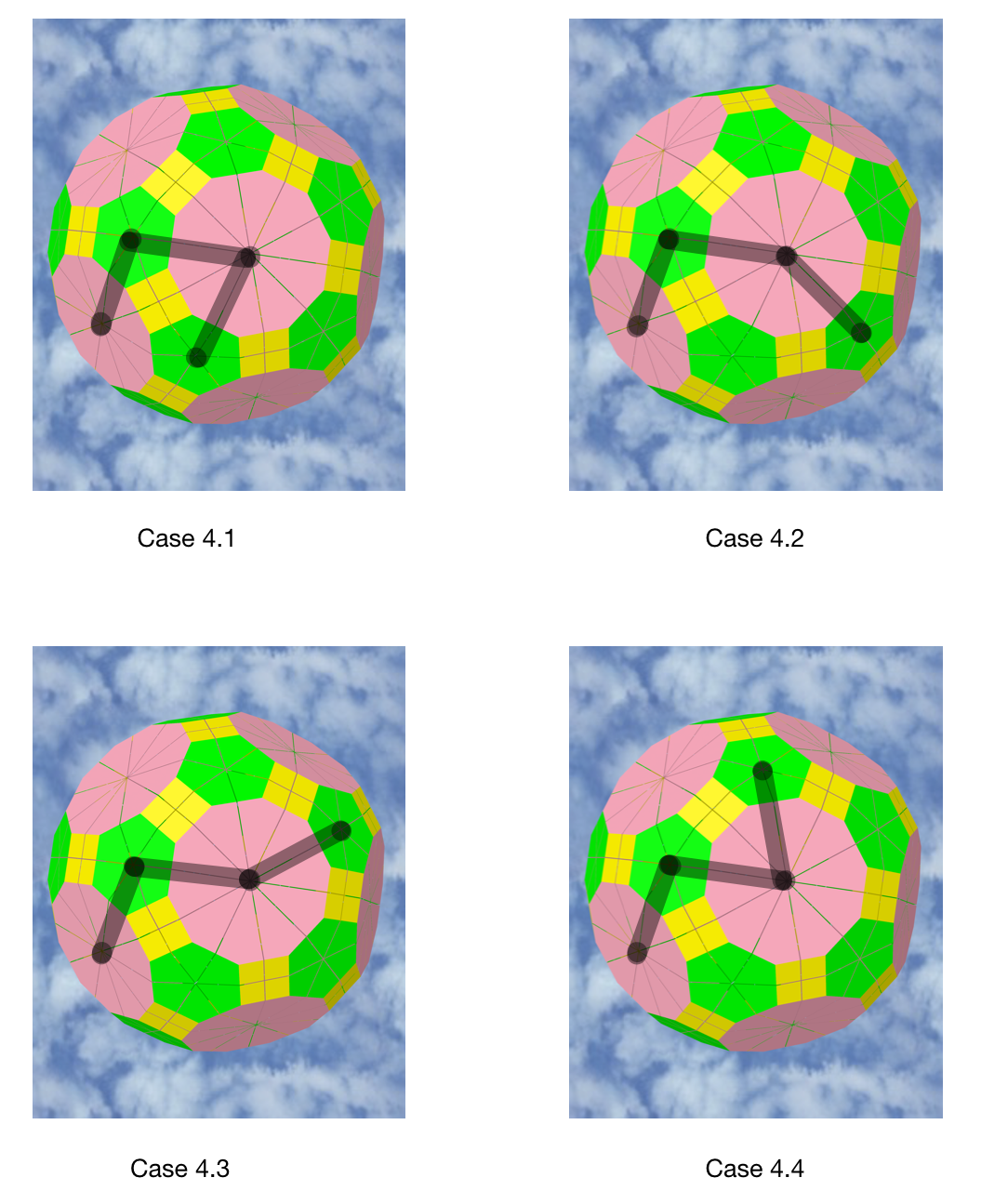}
		\caption{Subcases of Case 4. Picture made using KaleidoTile \cite{Kaleidotile}.}
			\label{fig:4}
	\end{figure}
	
	\smallskip
	\noindent
	\underline{Case 4.1.}\   In this case $\whC_{12}$ and $\whC_{14}$ are standard subcomplexes of type $b$, and $\whC_{13}$ is a standard subcomplex of type $a$. Thus by considering the loop $\Pi_{\whC_{1}}(P)$, we read of a word of form $$w_{ab}b^{k'_1}a^{k'_2}b^{k'_3}a^{k'_4}b^{k'_5}$$ which represents the trivial element in $A_{ab}$, moreover, the subwords $a^{k'_2},a^{k'_4}$ correspond to paths in $\whC_{13}$, and the subword $b^{k'_3}$ corresponds to a path $\whC_{14}$.
	Thus up to passing to a word which is equivalent to $w$, we can assume $w_{ab}=a^{k_1}b^{k_2}a^{k_3}$, and it corresponds to a path in $\whC_{13}\cup \whC_{14}$. Thus $$P\subset (\whC_{13}\cup \whC_{14})\cup \whC_2\cup \whC_3\cup\whC_4.$$ Similar by considering $\Pi_{\whC_{4}}(P)$, we deduce that
	\begin{center}
	$w'_{bc}=c^{k_4}b^{k_5}c^{k_6}$ and it corresponds to a path in $\whC_{41}\cup\whC_{42}$.
	\end{center}
	  Thus $$P\subset (\whC_{13}\cup \whC_{14})\cup \whC_2\cup \whC_3\cup(\whC_{41}\cup\whC_{42}).$$ By considering $\Pi_{\whC_{41}}(P)$ and noting that $\Pi_{\whC_{41}}(\whC_i)$ is a point for $i=1,2$, we deduce that $k_2+k_5=0$. 
	  
	  If $k_2=k_5=0$, then $w'_{bc}=c^*$.
	   we can modify $P_4$ such that it is in $\whC_2$ as follows. More precisely, we write $w'_{bc}$ as $ac^*a^{-1}$ (this is possible as $a$ and $c$ commute). This has the effect of replacing $P_4$ by a homotopic path such that it is a concatenation of a subpath in $\whC_3$ (corresponding to $a$), a subpath in $\whC_2$ (corresponding to $c^*$) and a subpath in $\whC_3$ (corresponding to $a^{-1}$). By combining the first and third subpath of $P_4$ with $P_3$ and $P_5$ respectively, we can assume $P_4\subset \whC_2$. Then $P\subset \whC_1\cup \whC_2\cup\whC_3$. In the level of the cycle $\omega$ in Lemma~\ref{lem:triple}, this has the effect of replace $x_2$ by another type $\hat a$ vertex which is adjacent to $y_2$ and $y_3$. Thus we are reduced to Case 2 or 3.

We now assume $k_2\neq 0$, hence $k_5\neq 0$. 	Let 
\begin{center}
$\wsX,\wsX'$, $\{\wsX_i\}_{i=1}^4$, $\wsX_{ij}$, $\wsY$ and $\kappa:\Si_{\ca}\to \Si_{\mathcal K}$
\end{center}
 be as in Section~\ref{subsec:AuII}. We view $\wsX$ as a subcomplex of $\od_{\mathcal K}$. Let $\hat \kappa:\od_{\ca}\to \od_{\mathcal K}$ be the map induced by $\kappa$. We will assume 
	\begin{center}
	$\hat \kappa(\whC_1)=\wsX_{22}$,	$\hat \kappa(\whC_2)=\wsX_{31}$, $\hat \kappa(\whC_3)=\wsX_{42}$ and $\hat \kappa(\whC_4)=\wsX_{33}$.
	\end{center}
Then $\hat \kappa(P)\subset \wsX'$ and $\hat \kappa(P)$ is null-homotopic in $\wsX'$ by Lemma~\ref{lem:injective} and Lemma~\ref{lem:BSinjective}. 
Let $W_{234}$ be the Coxeter group of type $B_3$ with its generating set $S=\{r,s,t\}$ such that $m_{rs}=4$ and $m_{st}=3$. Let $\od_{234}$ and $\hat e:\wsX'\to \od_{234}$ be as in Section~\ref{subsec:AuII}. Let $\Delta_{234}$ be the Artin complex of the Artin group of type $B_3$.
Then $Q=\hat e\circ \hat \kappa(P)$ is null-homotopic in $\widehat \Si_{234}$.  
Let $\widetilde \Si_{234}$ be the universal cover of $\widehat \Si_{234}$. We define a \emph{subcomplex of type $\hat r$} in $\widetilde \Si_{234}$ to be a lift of a standard subcomplex of type $\hat r$ in $\od_{234}$.

Let $Q_i=\hat e\circ \hat \kappa(P_i)$. Let $\wtQ$ be a lift of $Q$ to $\widetilde \Si_{234}$. Let $\whC'_i=\hat e\circ \hat \kappa(\whC_i)$, and $\whC'_{ij}=\hat e\circ \hat \kappa(\whC_{ij})$.
For $i=1,3,5$, let $\widetilde\Phi_i$ be the subcomplex of type $\hat t$ that contains $\wtQ_i$, and let $y'_i$ be the vertex in $\Delta_{234}$ corresponding to $\widetilde\Phi_i$. For $i=2,4,6$, let $\widetilde\Phi_i$ be the subcomplex of type $\hat r$ that contains $\wtQ_i$, and let $x'_i$ be the vertex in $\Delta_{234}$ corresponding to $\widetilde\Phi_i$.
Let $\omega'$ be the following 6-cycle in $\Delta_{234}$: $$y'_1x'_1y'_2x'_2y'_3x'_3.$$ 

We now show $\omega'$ is locally embedded. Clearly $y'_1\neq y'_i$ for $i=2,3$ as $\widetilde\Phi_1$ and $\widetilde\Phi_i$ are mapped to different standard subcomplexes in $\od_{234}$. Similar reason yields $x'_2\neq x'_i$ for $i=1,3$. From the description of $P_1$ and $P_4$, we know $\hat \kappa$ does not collapse any edges in $P_1$ and $P_4$. Thus the loop $Q_1$ represents an element in $\pi_1(\whC'_1)$ which is not contained in $\pi_1(\whC'_1\cap \whC'_2)$, which implies $x'_1\neq x'_3$. Similarly $y'_2\neq y'_3$.

By Theorem~\ref{thm:triple}, there is a vertex $z\in \Delta_{234}$ of type $\hat t$ such that $z$ is adjacent to $y'_i$ for $1\le i\le 3$.

\begin{claim}
	\label{claim:nonequal}
We have $z\neq x'_i$ for $i=1,3$.
\end{claim}

\begin{proof}
We first rule out $z=x'_1$. If $z=x'_1$, then $y'_1x'_1y'_2x'_3$ form an embedded 4-cycle in $\Delta_{234}$. By Theorem~\ref{thm:bowtie free}, there is a vertex $w$ of type $\hat s$ such that it is adjacent to each of the vertices of this 4-cycle. Let $\widetilde\Phi_w$ be the standard subcomplex in $\widetilde \Si_{234}$ associated with $w$. Then $\widetilde \Phi_1\cap \widetilde \Phi_2\cap\widetilde \Phi_w$ is exactly a vertex (as $y'_1,x'_1$ and $w$ span a triangle), denoted by $q$. Similar, define $q'=\widetilde \Phi_1\cap \widetilde \Phi_w\cap\widetilde \Phi_6$. We consider $\widetilde Q'_1$ which is a concatenation of the following path:
\begin{enumerate}
	\item $\widetilde Q'_{11}\subset \widetilde \Phi_6\cap \widetilde \Phi_1$ from the starting point of $\wtQ_1$ to $q'$;
	\item $\widetilde Q'_{12}\subset \widetilde\Phi_1\cap \widetilde\Phi_w$ from $q'$ to $q$;
	\item $\widetilde Q'_{13}\subset \widetilde\Phi_1\cap \widetilde \Phi_2$ from $q$ to the endpoint of $\wtQ_1$. 
\end{enumerate}
Then $\wtQ'_1$ is homotopic to $\wtQ_1$ rel endpoints in $\widetilde\Phi_1$. Let $\Phi_i$ be the image of $\widetilde \Phi_i$ under $\widetilde \Si_{234}\to \od_{234}$. We define $\Phi_w$ and $Q'_3$ similarly. Let $\bar w, \bar x'_i$ be the image of $w,x'_i$ under $\Delta_{234}\to \bC_{234}$ respectively. Then $\bar w$ is adjacent $\bar y'_1,\bar y'_2$ and $\bar x'_1$. This determines the position of $\bar w$. In particular, $$Q'_{12}\subset \Phi_w\cap \Phi_1=\whC'_{13}.$$ Thus $Q'_1\subset \whC'_{13}\cup \whC'_{12}$. 
Let $Q_{1i}=\hat e\circ \hat \kappa(P_{1i})$. Then $Q_{11}\cup Q_{13}\subset \whC'_{12}$ and $Q_{12}\subset \whC'_{14}$. Consider the retraction $r:\whC'_1\to \whC'_{14}$. As $Q_1$ and $Q'_1$ are homotopic rel endpoints in $\whC'_1=\Phi_1$, we know $r(Q_1)$ and $r(Q'_1)$ are homotopic rel endpoint in $\whC'_{14}$. However, this is a contradiction as $r(Q'_1)$ is a point, and $r(Q_1)=Q_{12}$ is null-homotopic as we are assuming $k_2\neq 0$ and $k_5\neq 0$. Thus $z\neq x'_1$. The proof of $z\neq x'_3$ is similar.
\end{proof}

\begin{claim}
	\label{claim:trim}
We have $w_{bc}=b^*c^*b^*$ and $w''_{bc}=b^*c^*b^*$.
\end{claim}

\begin{proof}
We only prove the claim for $w_{bc}$, as the other case is similar. Consider the 4-cycle $y'_1x'_1zx'_3$, which is embedded by Claim~\ref{claim:nonequal}. By Theorem~\ref{thm:bowtie free}, there exists a vertex $w'\in \Delta_{234}$ of type $\hat s$ such that it is adjacent of the each of the vertices of this 4-cycle. Let $\bar w'$ be the image of $w'$ under $\Delta_{234}\to \bC_{234}$. Then $\bar w'$ is the unique vertex which is adjacent to $\bar y'_1,\bar y'_2$ and $\bar x'_1$. By the same argument as in the proof of Claim 2, we know $Q_2$ is homotopic rel endpoints in $\whC'_2$ to a path $Q'_2$ which is a concatenation of $Q'_{21}\subset \whC'_{21}$, $Q'_{22}\subset \whC'_{24}$ and $Q'_{23}\subset \whC'_{23}$. As $\hat e\circ \hat \kappa$ map $\whC_2$ homeomorphically onto $\whC'_2$, we know $P_2$ is homotopic rel endpoints in $\whC_2$ to the concatenation of a path in $\whC_{21}$, a path in $\whC_{24}$ and a path in $\whC_{23}$. Hence $w_{bc}=b^*c^*b^*$.
\end{proof}

By Claim~\ref{claim:trim}, up to passing to a word equivalent to $w$, we can assume $(P_2\cup P_6) \subset \whC_{24}$, $w_{bc}=c^*$, and $w''_{bc}=c^*$. We now modify $P_2$ and $P_6$ as follows. Write $w_{bc}$ as $ac^*a^{-1}$. This has the effect of replacing $P_2$ by a homotopic path such that it is a concatenation of $P_{21}\subset \whC_1$ (corresponding to $a$), $P_{22}\subset \whC_4$ (corresponding to $c^*$), and $P_{23}\subset \whC_3$ (corresponding to $a^{-1}$). Similarly we write $w''_{bc}$ as $ac^*c^{-1}$ and replacing $P_6$ by $P_{61}P_{62}P_{63}$ with 
\begin{center}
$P_{61}\subset \whC_3$, $P_{62}\subset \whC_2$ and $P_{63}\subset \whC_1$.
\end{center}
By combining $P_{63}$ and $P_{21}$ with $P_1$, combining $P_{23}$ with $P_3$, and combining $P_{61}$ with $P_5$; we can assume $P_2\cup P_6\subset \whC_4$. Then $P\subset \whC_1\cup \whC_4\cup \whC_3$. In the level of the cycle $\omega$, this has the effect of replacing $x_1$ by another type $\hat a$ vertex that is adjacent to $y_1$ and $y_2$, and replacing $x_3$ by another type $\hat a$ vertex that is adjacent to $y_3$ and $y_1$. Thus we are reduced to Case 2.

	\smallskip
	\noindent
	\underline{Case 4.2.}\   In this case  $\whC_{41}$ is a standard subcomplex of type $c$, $\whC_{42}$ is a point and $\whC_{43}$ is a standard subcomplex of type $b$. By considering the loop $\Pi_{\whC_{4}}(P)$, we read of a word of form $$c^*b^*w'_{bc}b^*c^*$$ which represents the trivial element in $A_{bc}$.  Thus up to passing to a word which is equivalent to $w$, we can assume $w'_{bc}=c^*$ and $P_4\subset \whC_{41}$.
	
	%$\whC_{12}$ and $\whC_{14}$ are standard subcomplexes of type $b$, and $\whC_{13}$ is a standard subcomplex of type $a$. Moreover, for $i=2,3,4$, the subcomplex $\whC_{1i}$ remained unchanged compared to in Case 4.1.  By considering the loop $\Pi_{\whC_{1}}(P)$, we read of a word of form $$w_{ab}b^{k_1}a^{k_2}b^{k_3}a^{k_4}b^{k_5}$$ which represents the trivial element in $A_{ab}$. Thus up to passing to a word which is equivalent to $w$, we can assume $$w_{ab}=a^{k_1}b^{k_2}a^{k_3}$$ (for possibly different value of $k_1,k_2,k_3$). Note that
	
	Let $\whC'_4$ be the subcomplex $\whC_4$ in Case 4.1. Our next goal is to modify $P_4$ such that it is in $\whC'_4$. More precisely, we write $w'_{bc}$ as $ac^*a^{-1}$ (this is possible as $a$ and $c$ commute). This has the effect of replacing $P_4$ by a homotopic path such that it is a concatenation of a subpath in $\whC_3$ (corresponding to $a$), a subpath in $\whC'_4$ (corresponding to $c^*$) and a subpath in $\whC_3$ (corresponding to $a^{-1}$). By combining the first and third subpath of $P_4$ with $P_3$ and $P_5$ respectively, we can assume $P_4\subset \whC'_4$. Then $$P\subset \whC_2\cup \whC_3\cup\whC'_4\cup (\cup_{i=2}^4\whC_{1i}).$$ In the level of the cycle $\omega$ in Lemma~\ref{lem:triple}, this has the effect of replace $x_2$ by another type $\hat a$ vertex which is adjacent to $y_2$ and $y_3$. Thus we are reduced to Case 4.1.
	
	\smallskip
	\noindent
	\underline{Case 4.3.}\  In this case, $\whC_{41}=\whC_{43}$ is a standard subcomplex of type $a$, and $\whC_{42}$ is a point. Thus by considering the loop $\Pi_{\whC_{4}}(P)$, we deduce that $w'_{bc}$ is a power of $b$. This implies that $y_3=y_5$ in $\omega$, which contradicts that $\omega$ is a local embedding.
	
	\smallskip
	\noindent
	\underline{Case 4.4.}\  In this case $\whC_{41}=\whC_{42}$ is a standard subcomplex of type $c$, and $\whC_{43}$ is a standard subcomplex of type $b$. By considering the loop $\Pi_{\whC_{4}}(P)$, we read of a word of form $$c^*b^*w'_{bc}b^*c^*$$ which represents the trivial element in $A_{bc}$.  Thus up to passing to a word which is equivalent to $w$, we can assume $w'_{bc}=c^*$ (for possibly different value of $k_4$) and $P_4\subset \whC_{41}$. By a similar argument as in Case 4.2 (i.e. writing $w'_{bc}=ac^{k_4}a^{-1}$), we can assume $P_4\subset \whC_2$. This reduces to Case 2.

	\medskip
	\noindent
	\underline{Case 5: the $\pi$-image of $\omega$ is three edges sharing a common vertex $\bar x$.} Suppose $\pi(\omega)$ $=\bar x\bar y_1\cup\bar x\bar y_2\cup\bar x\bar y_3$. Let $\whC_0=\whC_{\bar x}$ and $\whC_i=\whC_{y_i}$ for $i=1,2,3$. Let $\whC_{ij}=\Pi_{\whC_i}(\whC_j)$.
	
	\smallskip
	\noindent
	\underline{Case 5.1: $\bar x$ is of type $\hat a$.} Then up to a cyclic permutation of the index $i$, we assume $P_i\subset \whC_{\bar x}$ for $i=2,4,6$, and $P_{2i-1}\subset \whC_{y_i}$ for $i=1,2,3$. 
	By considering the retraction $\Pi_{\whC_1}(P)$, we can assume $w_{ab}=a^*b^*a^*$ and it corresponds to a path in $\cup_{i\neq 1}\whC_{1i}$. Similarly, $w'_{ab}=a^*b^*a^*$ with the associated path in $\cup_{i\neq 2}\whC_{1i}$ and $w''_{ab}=a^*b^*a^*$ with the associated path in $\cup_{i\neq 3}\whC_{1i}$. Now we  consider $\Pi_{\whC_{12}}(P)$. Note that $\Pi_{\whC_{12}}(P_i)$ is not a constant path if and only if $i=2,4$. Thus the power of $a$s at the end of $w_{ab}$ cancels with the power of $a$s at the beginning of $w'_{ab}$. There are similar cancellations between $w'_{ab}$ and $w''_{ab}$, and between $w''_{ab}$ and $w_{ab}$. Thus we write 
	\begin{center}
		$w_{ab}=a^{k_1}b^{k_2}a^{k_3}$, $w'_{ab}=a^{-k_3}b^{k_4}a^{k_5}$ and $w''_{ab}=a^{-k_5}b^{k_6}a^{-k_1}$.
	\end{center}
	If one of $k_1,k_3,k_5$ is zero, say $k_5=0$, then we write $P_3=P'_3P''_3$ with $P'_3\subset \whC_{21}$ and $P''_3\subset \whC_0$. Note that $P'_3$ is homotopic rel endpoints to a path $P'_{31}P'_{32}P'_{33}$ with $P'_{31}\cup P'_{33}\subset \whC_0$ and $P'_{32}\subset \whC_1$. We replace $P_3$ by $P'_{31}P'_{32}P'_{33}P''_3$, combine $P'_{31}$ with $P_2$ and combine $P'_{33}P''_3$ with $P_4$. Then $$
	P\subset \whC_0\cup\whC_1\cup\whC_3.
	$$
	In the level of $\omega$, this has the effect of replacing $y_2$ by another vertex of type $\hat c$ that is adjacent to both $x_1$ and $x_2$. Now we are done by Case 2 and Claim~\ref{claim:change} below.

	\begin{claim}
	\label{claim:change}
	Let $\{x_i\}$ and $\{y_i\}$ be as in the lemma. Let $y'_2\neq y_2$ be a vertex of type $\hat c$ in $\Delta$ which is adjacent to both $x_1$ and $x_2$. Suppose there is a vertex $s'\in \Delta$ such that $s'$ is adjacent to each of $\{y_1,y'_2,y_3\}$. Then $s'$ is adjacent to each of $\{y_1,y_2,y_3\}$.
	\end{claim}

\begin{proof}
First we consider the case $s'\neq x_1$ and $s'\neq x_2$. By Theorem~\ref{thm:bowtie free}, there are vertices $q_1,q_2,q_3$ of type $\hat b$ such that
\begin{itemize}
	\item $q_1$ is adjacent to each of $y_1,x_1,y'_2,s'$;
	\item $q_2$ is adjacent to each of $y'_2,x_2,y_3,s'$;
	\item $q_3$ is adjacent to each of $x_1,y_2,x_2,y'_2$.
\end{itemize}
Then $s'q_1x_1q_3x_2q_2$ form a 6-cycle in $\lk(y'_2,\Delta)$. However, as $lk(y'_2,\Delta)$ is a copy of the Artin complex of $A_{ab}$, which has girth $\ge 10$ by \cite[Lemma 6]{appel1983artin}. Thus the 6-cycle is degenerate. As $s'\neq x_1$, $s'\neq x_2$ and $x_1\neq x_2$, we know the back-tracking vertices of this 6-cycle is a subset of $\{x_1,x_2,s'\}$, which is impossible.
Thus either $s'=x_1$ or $s'=x_2$, and the claim follows.
\end{proof}

In the rest of Case 5.1, we assume $k_1\neq 0,k_3\neq 0$ and $k_5\neq 0$. 
	
	Let $\wsX,\wsX'$, $\{\wsX_i\}_{i=1}^4$, $\wsY$ and $i:\wsY\to \od$ be as in Section~\ref{subsec:AuII}. We can assume without loss of generality that $i(\wsX_{31})=\whC_{\bar x}$. Then the discussion in the previous paragraph implies that $P\subset i(\wsY)$. Thus we can view $P$ as an edge loop in $\wsX'$, moreover, Lemma~\ref{lem:embed} implies that $P$ is null-homotopic in $\wsX'$. 
	
	Let $W_{234}$ be the Coxeter group of type $B_3$ with its generating set $S=\{r,s,t\}$ such that $m_{rs}=4$ and $m_{st}=3$. Let $\od_{234}$ and $\hat e:\wsX'\to \od_{234}$ be as in Section~\ref{subsec:AuII}. 
	Then $Q=\hat e(P)$ is null-homotopic in $\widehat \Si_{234}$.  
Recall that $\widetilde \Si_{234}$ is the universal cover of $\widehat \Si_{234}$. 
	Note that $Q$ lifts to a loop $\wtQ$ in $\widetilde \Sigma_{234}$. For $i=1,3,5$, let $\widetilde V_i$ be the subcomplex of type $\hat t$ in $\widetilde \Si_{234}$ containing $\wtQ_i$, and let $y'_i$ be the vertex  in $\Delta_{234}$ associated with $\widetilde V_i$. For $i=2,4,6$, let $\widetilde V_i$ be the subcomplex of type $\hat r$ in $\widetilde \Si_{234}$ containing $\wtQ_i$, and let $x'_i\in \Delta_{234}$ be the associated vertex. Then we have a cycle $$\omega'=y'_1x'_1y'_2x'_2y'_3x'_3$$ in $\Delta_{234}$. Note that $\omega'$ is a local embedding at $y'_i$ as we are assuming $k_1\neq 0$, $k_3\neq 0$ and $k_5\neq 0$. Moreover, $\omega'$ is local embedding at $x'_i$ as $\omega$ is a local embedding at $x_i$. 
	
	By Theorem~\ref{thm:triple}, there is a vertex $z\in\Delta_{234}$ of type $\hat t$ such that it is adjacent to $y'_i$ for $1\le i\le 3$. 
	Let $\bC_{234}$ be the Coxeter complex for $W_{234}$. Recall that there is a map $\Delta_{234}\to \bC_{234}$ induced by the quotient of the action of the pure Artin group $PA_{234}$. Let $\bar z$ be the image of $z$ under the map $\Delta_{234}\to \mathbb \bC_{234}$. We define $\bar x'_i$ and $\bar y'_i$ similarly. As $\bar z$ is adjacent to $\bar y'_1$, $\bar y'_2$ and $\bar y'_3$, we know $\bar x'_1=\bar x'_2=\bar x'_3=\bar z$.
	
	We have  $z\neq x'_i$ for $1\le i\le 3$. This can be proved in the same way as Claim~\ref{claim:nonequal}, using $k_1\neq 0,k_3\neq 0$ and $k_5\neq 0$.  Thus for $j=1,2,3$, the vertices $x'_j,y'_j,y'_{j+1},z$ form an embedded 4-cycle in $\Delta_{234}$.
	By Theorem~\ref{thm:bowtie free}, there is a vertex $w_j$ of type $\hat s$ such that $w_j$ is adjacent to each vertex of this 4-cycle (see Figure~\ref{fig:case5} below). Let $\widetilde V_{w_j}$ and $\widetilde V_z$ be the standard subcomplexes of $\widetilde \Sigma_{234}$ corresponding to $w_j$ and $z$ respectively. As $x'_1,y'_1,w_1$ span a triangle in $\Delta_{234}$, we know $\widetilde V_1\cap \widetilde V_2\cap \widetilde V_{w_1}$ is a single vertex, denoted by $q_1$. Similarly,  for $i=2,4,6$, let 
	\begin{center}
		$q_i=\widetilde V_i\cap \widetilde V_{i+1}\cap \widetilde V_{w_{\frac{i+1}{2}}}$ for $i=1,3,5$ and $q_i=\widetilde V_i\cap \widetilde V_{i+1}\cap \widetilde V_{w_{i/2}}$.
	\end{center}

	\begin{figure}
		\centering
		\includegraphics[scale=1]{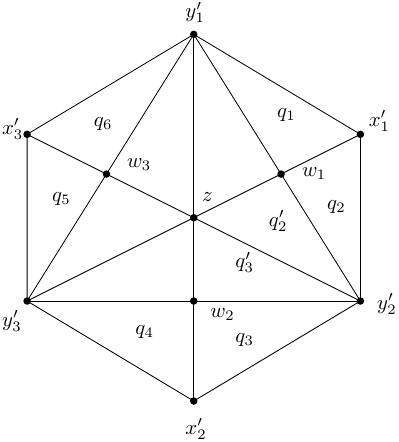}
		\caption{Case 5.}
			\label{fig:case5}
	\end{figure}
	
	\begin{claim}
		\label{claim:endpoint}
	The path $\wtQ_i$ starts at $q_{i-1}$ and ends at $q_{i}$ for $1\le i\le 6$. 
	\end{claim}

Assume Claim~\ref{claim:endpoint} for the moment, we explain how to finish Case 5.1.
  Let $\wtQ'_2$ be a shortest path in $\widetilde V_{w_1}\cap \widetilde V_2$ from $q_1$ to $q_2$. By the previous paragraph, $\wtQ'_2$ is homotopic rel endpoints to $\wtQ_2$ in $\widetilde V_2$. Let $V_i$ be the image of $\widetilde V_i$ under $\widetilde \Si_{234}\to \widehat\Si_{234}$. We define $V_{w_j},V_z$ and $Q'_2$ similarly.
  Thus $Q_2$ and $Q'_2$ are homotopic rel endpoints in $V_2$. Then $P_2$ and $\hat e^{-1}(Q'_2)$ are homotopic rel endpoints in $\wsX_{31}$. It follows that $P_2$ is homotopic rel endpoints in $\whC_0$ to a path in $i\circ \hat e^{-1}(V_{w_1}\cap V_2)$. Thus $w_{bc}=c^*$. Similarly, we know $w'_{bc}=c^*$ and $w''_{bc}=c^*$.
  Thus the word $w$ is of form
  $$
a^{k_1}b^{k_2}a^{k_3}\cdot c^*\cdot a^{-k_3}b^{k_4}a^{k_5}\cdot c^*\cdot  a^{-k_5}b^{k_6}a^{-k_1}\cdot c^*.
  $$
  For $i=1,3,5$, we write $P_i=P_{i1}P_{i2}P_{i3}$ corresponding to  $a^*b^*a^*$. As $a$ and $c$ commute, we know $P_{12}P_{13}P_2P_{31}P_{32}$ is homotopic rel endpoints in $\od$ to a path in $\whC_0$, and $P_{52}P_{53}P_6P_{11}P_{12}$ is homotopic rel endpoints in $\od$ to a path in $\whC_0$. 
  
  We consider a lift $\widetilde P$ of $P$ to the universal cover $\widetilde \Sigma$ of $\od$. For $i=1,3,5$, let $K_i$ be the standard subcomplex in $\widetilde \Si$ of type $\hat c$ that contains $\widetilde P_i$. Let $K$ be the lift of $\whC_0$ in $\widetilde \Si$ that contains $\widetilde P_{12}$. Then the previous paragraph implies that $K\cap K_i\neq \emptyset$ for $i=1,3,5$.
  Thus $y_1,y_2,y_3$ are adjacent to a common vertex of type $\hat a$ in $\Delta$.

%By conjugating $w$ by $a^{-k^1}$, we obtain $w'$ which is equivalent to $w$. Note that $w'$ simplifies to a word only involving $b$ and $c$, form which it is not hard to deduce the conclusion of the lemma.

\begin{proof}[Proof of Claim~\ref{claim:endpoint}]
	We first show the starting point and ending point of $\wtQ_3$ are $q_2$ and $q_3$ respectively. To see this, let $q'_2=\widetilde V_{w_1}\cap \widetilde V_{3}\cap \widetilde V_z$ (the intersection is non-empty as $x'_1,y'_2,w_1$ span a triangle) and $q'_3=\widetilde V_{w_2}\cap \widetilde V_3\cap \widetilde V_z$. Let $\wtQ'_3$ be a concatenation of the following path:
	\begin{enumerate}
		\item $\wtQ'_{31}\subset \widetilde V_2\cap \widetilde V_3$ from the starting point of $\wtQ_3$ to $q_2$;
		\item  $\wtQ'_{32}\subset \widetilde V_3\cap \widetilde V_{w_1}$ from $q_2$ to $q'_2$;
		\item $\wtQ'_{33}\subset \widetilde V_{z}\cap \widetilde V_3$ from $q'_2$ to $q'_3$;
		\item $\wtQ'_{34}\subset \widetilde V_{w_2}\cap \widetilde V_3$ from $q'_3$ to $q_3$;
		\item $\wtQ'_{35}\subset \widetilde V_4\cap \widetilde V_3$ from $q_3$ to the endpoint of $\wtQ_3$.
	\end{enumerate}    We can assume each $\wtQ'_{3i}$ is a geodesic path. Then $\wtQ_3$ and $\wtQ'_3$ are homotopic rel endpoints in $\widetilde V_3$. As $\bar x'_1=\bar x'_2=\bar x'_3=\bar z$, we know $$V_2=V_4=V_6=V_z.$$
 Thus $$Q'_3\subset V_3\cap (V_{w_1}\cup V_z\cup V_{w_2}).$$
 As $Q_3=\hat e(P_3)$, the form of $P_3$ implies that
 \begin{center}
$Q_3\subset V_3\cap (V_{w_1}\cup V_z\cup V_{w_2})$.
 \end{center}
  As $\bar w_j$ is adjacent to $\bar z,\bar y'_j$ and $\bar y'_{j+1}$, this uniquely determine the position of $\bar w_j$ in $\bC_{234}$. This implies that
  \begin{center}
 $V_3\cap (V_{w_1}\cup V_z\cup V_{w_2})$ is a union of three circles.
  \end{center}
	As $Q'_3(Q_3)^{-1}$ is homotopically trivial path in $V_3$, we know from Lemma~\ref{lem:injective} that $Q'_3(Q_3)^{-1}$ is homotopically trivial in $V_3\cap (V_{w_1}\cup V_z\cup V_{w_2})$ whose fundamental group is a free group. If both segments $Q_{31}$ and $Q_{35}$ are non-trivial, then $Q'_3(Q_3)^{-1}$ does not contain back-tracking, hence is a local geodesic loop in the graph $$V_3\cap (V_{w_1}\cup V_z\cup V_{w_2}),$$ which can not be null-homotopic in this graph. If one of $Q_{31}$ and $Q_{35}$ is non-trivial, then by killing all the back-tracking of the loop $Q'_3(Q_3)^{-1}$, we still have a homotopically non-trivial local-geodesic loop left, which is again not possible.
	Thus the segments $Q_{31}$ and $Q_{35}$ must be trivial, implying the claim. By a similarly argument, $\wtQ_i$ starts at $q_{i-1}$ and ends at $q_{i}$ for $i=1,3,5$. Thus $\wtQ_i$ starts at $q_{i-1}$ and ends at $q_{i}$ for $1\le i\le 6$. 
\end{proof}

	\smallskip
	\noindent
	\underline{Case 5.2: $\bar x$ is of type $\hat c$.} We write $\bar y_i\sim\bar y_j$ if they are adjacent to the same vertex of type $\hat b$ in $\bC$. Up to a permutation of $\{\bar y_i\}_{i=1}^3$, there are only two possibilities, namely either $\bar y_1\sim\bar y_2$ and $\bar y_2\sim \bar y_3$, or $\bar y_1\sim \bar y_3, \bar y_1\nsim \bar y_2$ and $y_2\nsim y_3$.  We assume $P_i\subset \whC_{\bar x}$ for $i=1,3,5$, and $P_{2i}\subset \whC_{y_i}$ for $i=1,2,3$. 
	
	First we consider the subcase when $\bar y_1\sim\bar y_2$ and $\bar y_2\sim \bar y_3$. Note that $\whC_{13}$ is a point and $\whC_{12}$ is a standard subcomplex of type $c$. By considering the loop $\Pi_{\whC_1}(P)$, we read of a word of form $$b^*w_{bc}b^*c^*b^*$$ representing the trivial element in $A_{bc}$ where the three $b^*$ subwords are associated with paths in $\whC_{10}$.
	Thus up to passing a word equivalent to $w$, we can assume $w_{bc}=c^*$ and it corresponds to a path in $\whC_{12}$. Let $\bar z$ be the vertex of type $\hat b$ that is adjacent to both $\bar y_1$ and $\bar y_2$. Using the product structure on $\whC_{\bar z}$, we know $P_2$ is homotopic rel endpoints to a path $P_{21}P_{22}P_{23}$ where $P_{2i}\subset \whC_{0}\cap \whC_{\bar z}$ for $i=1,3$ and $P_{22}\subset \whC_{2}\cap \whC_{\bar z}$. By combing $P_{21}$ with $P_1$ and combining $P_{23}$ with $P_3$, we can assume $P_2\subset \whC_2$. Then $$P\subset \whC_0\cup \whC_2\cup\whC_3.$$ In the level of the cycle $\omega$ in Lemma~\ref{lem:triple}, this has the effect of replace $x_1$ by another type $\hat a$ vertex which is adjacent to $y_1$ and $y_2$. Thus we are reduced to Case 3.
	
	Now we consider the subcase when $\bar y_1\sim \bar y_3, \bar y_1\nsim \bar y_2$ and $\bar y_2\nsim \bar y_3$. Then $\whC_{13}$ is a standard subcomplex of type $c$ and $\whC_{12}$ is a point. Thus this is similar to the previous subcase.
	
	\medskip
	\noindent
	\underline{Case 6: the $\pi$-image of $\omega$ is four edges.} In this case, the only possibility is that $\pi(\omega)$ is a 4-cycle. Let $\bar x_1,\bar x_2,\bar x_3,\bar x_4$ be consecutive vertices on this 4-cycle. Let $\whC_i=\whC_{\bar x_i}$ and $\whC_{ij}=\Pi_{\whC_i}(\whC_j)$. Up to a cyclic permutation of the index and symmetries of $\bC$, we can assume $\bar x_1$ is of type $\hat \hat c$, and one of the following holds true: 
	\begin{enumerate}
		\item $P_i\subset \whC_i$ for $1\le i\le 4$, $P_5\subset \whC_1$ and $P_6\subset \whC_4$;
		\item $P_i\subset \whC_i$ for $1\le i\le 4$, $P_5\subset \whC_1$ and $P_6\subset \whC_2$.
	\end{enumerate}
We will only treat the first situation, as the second is similar.

	 By considering the loop $\Pi_{\whC_2}(P)$, we read of a word of form $$b^*w_{bc}b^*c^*b^*c^*$$ which represents the trivial element in $A_{bc}$. Thus $w_{bc}=b^*c^*b^*c^*b^*$ where the first and second $b^*$ represent a path in $\whC_1\cap \whC_2$, the third $b^*$ represents a path in $\whC_2\cap \whC_3$, and the two $c^*$ represent paths in $\whC_{24}$. By considering the loop $\Pi_{\whC_3}(P)$ and carry out a similar analysis, we know $w'_{ab}=b^*a^*b^*a^*b^*$ where the second and third $b^*$ represent paths in $\whC_3\cap \whC_4$, the first $b^*$ represents a path in $\whC_2\cap \whC_3$, and the two $a^*$ represent paths in $\whC_{31}$. 
	 
	We consider $\Pi_{\whC_{23}}(P)$. Note that $\Pi_{\whC_{23}}(P_i)$ is a point for $i=1,4,5,6$, $\Pi_{\whC_{23}}(P_2)$ is the path corresponding to the last $b^*$ of $w_{bc}$, and $\Pi_{\whC_{23}}(P_2)$ is the path corresponding to the first $b^*$ of $w'_{ab}$. Thus the last $b^*$ of $w_{bc}$ and first $b^*$ of $w'_{ab}$ cancel. By merging the first $b^*$ of $w_{bc}$ with $w_{ab}$, and merging the last $b^*$ of $w'_{ab}$ with $w'_{bc}$, we assume
	 \begin{center}
	 	 $w_{bc}=c^{k_1}b^{k_2}c^{k_3}$ and $w'_{ab}=a^{k_4}b^{k_5}a^{k_6}$. 
	 \end{center}
	 
	 If $k_2=0$, then $w_{bc}=c^*$. By rewriting $w_{bc}=ac^*a^{-1}$ which corresponds to the concatenation of a path in $\whC_{13}$, a path in $\whC_{42}$ and a path in $\whC_{31}$, and merging $a$ with $w_{ab}$ and $a^{-1}$ with $w'_{ab}$, we can assume $P_2\subset \whC_{42}\subset \whC_4$. Then $$P\subset \whC_1\cup\whC_3\cup\whC_4,$$ and we are reduced to Case 2. 
	 
If $k_5=0$, we can perform a similar replacement of $P_3$ and arrange that $P\subset \whC_1\cup\whC_2\cup\whC_4$, at the cost of replacing $y_2$ by another vertex of type $\hat c$ which is adjacent to both $x_1$ and $x_2$. Then we are done by Case 3 and Claim~\ref{claim:change}.

Thus from now on we assume $k_2\neq 0$ and $k_5\neq 0$. We can also assume $k_1\neq 0$, otherwise we can combine $b^{k_2}$ with $w_{ab}$ and still have $w_{bc}=c^*$. Similarly, we assume $k_6\neq 0$. As $c^{k_3}$ corresponds to a path from $\whC_1$ to $\whC_3$, we know $k_3\neq 0$.

	Let $\whZ$ be defined in the end of Section~\ref{subsec:aug1}, viewed as a subcomplex of $\widehat \Sigma$ and a subcomplex of $\whX_1\cup\whX_2$ (we identify $\whX_{22}$ with $\whC_1$ and $\whX_{11}$ with $\whC_{4}$). Then $P\subset\whZ$ and by Lemma~\ref{lem:transfer1}, $P$ is null-homotopic in $\whX_1\cup\whX_2$. 
	The word $w_{bc}=c^*b^*c^*$ induces a decomposition $$P_2=P_{21}P_{22}P_{23}$$ with $P_{21},P_{23}\subset \whX_{21}$ and $P_{22}\subset \whX_{22}$. Similarly,
	\begin{center}
	 $P_{3}=P_{31}P_{32}P_{33}$ with $P_{31},P_{33}\subset \whX_{21}$ and $P_{32}\subset \whX_{11}$.
	\end{center} 
Note that $P_1,P_5\subset \whX_{22}$ and $P_4,P_6\subset \whX_{11}$. As $P$ is null-homotopic in $\whX_1\cup \whX_2$, it lifts to a loop $\wtP$ in the universal cover $\wtX$ of $\whX_1\cup \whX_2$. Let
\begin{itemize}
	\item $T_1,T_5$ be the standard subcomplexes of $\wtX$ containing $\wtP_1,\wtP_5$ of type $\whX_{22}$;
	\item $T_4,T_6$ be the standard subcomplexes of $\wtX$ containing $\wtP_4,\wtP_6$ of type $\whX_{11}$;
	\item $T_{21}$ and $T_{23}$ be the subcomplexes of $\wtX$ containing $\wtP_{21}$ and $\wtP_{23}$ of type $\whX_{21}$;
	\item $T_{22}$ be the standard subcomplex of $\wtX$ containing $\wtP_{22}$ of type $\whX_{22}$. 
\end{itemize}
 We define $T_{31},T_{32}$ and $T_{33}$ similarly. Note that $T_{23}=T_{31}$.
	Now $\wtP$ gives a loop $\omega_{\mathbb U}$ in $\mathbb U$ of form
	$$
	z_1\to z_{21}\to z_{22}\to z_{23}=z_{31} \to z_{32}\to z_{33}\to z_4\to z_5\to z_6
	$$ 
	where $z_i$ is the vertex associated with $T_i$ and $z_{ij}$ is the vertex associated with $T_{ij}$.
	
	As $k_i\neq 0$ for $i=1,2,3,5,6$, we know $\angle_{z_{21}}(z_1,z_{22})=\angle_{z_{33}}(z_{32},z_4)=\pi$, and Lemma~\ref{lem:piangle} implies that $\angle_{z_{22}}(z_{21},z_{23})=\angle_{z_{32}}(z_{31},z_{33})=\pi$. Thus the subsegment from $z_1\to z_{23}$ is a geodesic, and the subsegment from $z_{31}$ to $z_4$ is a geodesic. These two geodesics have the same length, and intersect in an angle $=\angle_{z_{23}}(z_{22},z_{32})=\pi/2$. On the other hand, the subsegment $$z_4\to z_5\to z_6\to z_1$$ has length equal to $\sqrt{2}$ times the length of the subsegment from $z_1\to z_{23}$. Thus by CAT$(0)$ geometry, the subsegment $z_4\to z_5\to z_6\to z_1$ is a geodesic, and $$\angle_{z_1}(z_{21},z_6)=\angle_{z_4}(z_{33},z_5)=\pi/4.$$ In particular $z_1,z_{21},z_6$ form a triangle in $\mathbb U$. Hence $T_1\cap T_6\cap T_{21}\neq\emptyset$. As $\wtP_1$ is a path from a point in $T_1\cap T_6$ to $T_1\cap T_{21}$, we know $\wtP_1$ is homotopic in $T_1$ rel endpoints to a path that is contained in $$(T_1\cap T_6)\cup (T_1\cap T_{21})$$ and passes through $T_1\cap T_6\cap T_{21}$. Thus $w_{ab}=b^*a^*$ in $A_{ab}$. By combining the $b^*$ part of $w_{ab}$ with $w''_{bc}$, we can assume $w_{ab}=a^*$. A similar argument implies that we can assume $w'_{bc}=c^*$.
	Now the word $w$ becomes $$a^* \cdot c^*b^*c^*\cdot  a^*b^*a^*\cdot c^*\cdot w''_{ab}w''_{bc}=1.$$ 
	We assume $w$ starts with $a^{k_1}\cdot c^{k_2}$. Replace $w_{ab}$ by $c^{k_2}a^{k_1}c^{-k_2}$, combine the $c^{k_2}$ part of $c^{k_2}a^{k_1}c^{-k_2}$ with $w''_{bc}$, and the $c^{-k_2}$ with $w_{bc}$. The new word $w$ takes form:
	$$a^*\cdot b^*c^*\cdot a^*b^*a^*\cdot c^*w''_{ab}w''_{bc}=1.$$
	On the level of $\omega$, this has the effect of replace $y_1$ by a vertex $y'_1$ of type $\hat c$ that are adjacent to both $x_1$ and $x_3$. Now we replace $w'_{bc}$ by $a^{-*}c^*a^*$, combine the $a^{-*}$ part with $w'_{ab}$ and the $a^*$ part with $w''_{ab}$. The new word $w$ takes form:
	$$a^*\cdot b^*c^*\cdot a^*b^*\cdot c^*w''_{ab}w''_{bc}=1.$$ 
	On the level of $\omega$, this has the effect of replace $x_2$ by a vertex $x'_2$ of type $\hat c$ that are adjacent to both $y_2$ and $y_3$.
By the way we handling the word \eqref{eq:word1}, we know $\{y'_1,y_3,y_5\}$ is adjacent to a common vertex in $\Delta$.
Then we are done by Claim~\ref{claim:change}.

	\medskip
	\noindent
	\underline{Case 7: the $\pi$-image of $\omega$ is five edges.} Then $\pi(\omega)$ is a 4-cycle with an extra edge. We assume vertices of this 4-cycle are $\bar x_1,\bar x_2,\bar x_3,\bar x_4$, and $\bar x_1$ is adjacent to a vertex $\bar x_0\in \pi(\omega)$ which is outside the 4-cycle. Let $\whC_i=\whC_{\bar x_i}$ and $\whC_{ij}=\Pi_{\whC_i}(\whC_j)$. 
	
	\smallskip
	\noindent
	\underline{Case 7.1: $\bar x_1$ is of type $\hat c$.} Up to a cyclic permutation of the index $i$, we assume
	\begin{center}
	 $P_i\subset\whC_i$ for $1\le i\le 4$, $P_5\subset \whC_1$ and $P_6\subset \whC_0$. 
	\end{center}
	First we consider the subcase that $\bar x_0$ and $\bar x_2$ are adjacent to a common vertex $\bar z$ of type $\hat b$. Then $\whC_{02}=\whC_{03}$ is a standard subcomplex of type $c$, and $\whC_{04}$ is a single point. As the loop $\Pi_{\whC_0}(P)$ is null-homotopic in $\whC_0$, we know that in $A_{bc}$ $$b^*c^*b^*w''_{bc}=1.$$  Thus up to passing to an equivalent $w$, we can assume $w''_{bc}=c^*$ and it corresponds to a path in $\whC_{02}$. Then $P''_{bc}$ is homotopic rel endpoints in $\widehat \Si$ to a concatenation of a path $P_{61}\subset \whC_{1}\cap \whC_{\bar z}$, a path $P_{62}\subset \whC_{\bar z}\cap\whC_{2}$, and a path  $P_{63}\subset \whC_{1}\cap \whC_{\bar z}$. By combining $P_{61}$ with $P_5$ and $P_{63}$ with $P_1$, we can assume $P_6\subset \whC_2$. This reduces to Case 6. 
	
	Second we consider the subcase that $\bar x_0$ and $\bar x_2$ are not adjacent to any common vertex of type $\hat b$. Then $\whC_{02}$ and $\whC_{04}$ are single points, and $\whC_{01}=\whC_{03}$ is a standard subcomplex of type $b$. As $\Pi_{\whC_0}(P)$ is null-homotopic loop in $\whC_0$, we conclude that $w''_{bc}=b^*$ in $A_{bc}$, which contradicts the assumption that $\omega$ is a local embedding at $x_3$.
	
	\smallskip
	\noindent
	\underline{Case 7.2: $\bar x_1$ is of type $\hat a$.} Up to a cyclic permutation of the index $i$, we assume
	\begin{center}
 $P_i\subset\whC_{i+1}$ for $i=1,2,3$, $P_4\subset \whC_1$, $P_5\subset \whC_0$ and $P_6\subset \whC_1$. 
	\end{center}
Note that $\whC_{30}=\whC_{31}$ is a standard subcomplex of type $a$. As $\Pi_{\whC_3}(P)$ is null-homotopic loop in $\whC_3$, we conclude that $b^*w_{bc}b^*a^*=1$ in $A_{bc}$. Thus up to passing to an equivalent $w$, we can assume $w_{bc}=a^*$ and it corresponds to a path in $\whC_{31}$. Then $P_2$ is homotopic rel endpoints in $\widehat\Sigma$ to the concatenation of a path $P_{21}\subset \whC_{24}$, a path $P_{22}\subset \whC_{31}$ and a path $P_{23}\subset \whC_{42}$. By combining $P_{21}$ with $P_1$, and $P_{23}$ with $P_3$, we can assume $P_2\subset \whC_4$. And this reduces to Case 5.
	
	\medskip
	\noindent
	\underline{Case 8: the $\pi$-image of $\omega$ is six edges.} Then $\pi(\omega)$ is an embedded $6$-cycle. We assume vertices of this 6-cycle are $\{\bar x_i\}_{i=1}^6$.  Let $\whC_i=\whC_{\bar x_i}$ and $\whC_{ij}=\Pi_{\whC_i}(\whC_j)$. Up to a cyclic permutation, we assume $P_i\subset \whC_i$ for $1\le i\le 6$. From the geometry of $\Sigma$, we know there is a vertex $\bar x_0$ of type $\hat a$ such that $\bar x_0$ is adjacent to $\bar x_i$ for $i=1,3,5$.
	
	We first consider the situation that $\bar x_i\neq \bar x_0$ for $i=2,4,6$. Then $\whC_{24}$ and $\whC_{26}$ are single points, and $\whC_{25}$ is a standard subcomplex of type $c$. Thus by considering the null-homotopic loop $\Pi_{\whC_2}(P)$ in $\whC_2$, we know $b^*w_{bc}b^*c^*=1$. Thus up to passing to an equivalent $w$, we can assume $w_{bc}=c^*$ and it corresponds to a path in $\whC_{25}$. By the same argument as in Case 7.2, we can perform of homotopy (rel endpoints) of $P_2$ such that it is a concatenation of a path in $\whC_1$, a path in $\whC_{\bar x_0}$ and a path in $\whC_3$. Thus up to combining suitable subpaths of $P_2$ with $P_1$ and $P_3$, we can assume $P_2\subset\whC_{\bar x_0}$. Similar arguments implies that we can assume $P_4,P_6\subset \whC_{\bar x_0}$. On the level of $\omega$, this has the effect of replacing $x_i$ by another vertex of type $\hat a$ that are adjacent to $y_i$ and $y_{i+1}$ for $i=1,2,3$. Thus we are reduced to Case 5.
	
	It remains to consider the case that one of $\bar x_2,\bar x_4,\bar x_6$ is $ \bar x_0$, say $\bar x_6=\bar x_0$. Then $\whC_{24}$ is a single point, and $\whC_{25}=\whC_{26}$ is a standard subcomplex of type $c$. Thus by considering the null-homotopic loop $\Pi_{\whC_2}(P)$ in $\whC_2$, we know $b^*w_{bc}b^*c^*=1$. By the argument in the previous paragraph, we can still assume $P_2,P_4\subset \whC_{\bar x_0}$ and it reduces to Case 5 again.
\end{proof}

\begin{cor}
	Conjecture~\ref{conj:compareB} holds true whenever $A_S$ is not of type $H_4$.
\end{cor}

\begin{proof}
	The bowtie free part of Conjecture~\ref{conj:compareB} follows from Theorem~\ref{thm:bowtie free}. For the upward flag part, then the possibility left are the case $A_S$ is of type $B_n$ follows from Theorem~\ref{thm:triple}; the case that $A_S$ is of type $F_4$ and $A_{S'}$ is of type $B_3$ follows from Proposition~\ref{prop:F4}; the case $A_S$ being type $H_3$ follows from Theorem~\ref{thm:tripleH}. This exhausts all the cases, assuming If $A_S$ is not of type $H_4$. 
\end{proof}

\section{A remark on weakly flagness}
\label{sec:weakly flag}
In this section we give a technical remark on the relation between weakly flagness and downward flagness.

The following was proved in \cite{huang2023labeled} in the case when $X$ is an appropriate relative Artin complex, however, the same proof work in the slightly more general setting of simplicial complexes of type $S$.
\begin{lem}(\cite[Lemma 6.9]{huang2023labeled})
	\label{lem:bowtie free criterion}
	Let $X$ be a simplicial complex of type $S$, with its vertex set $V$ endowed with the relation $<$ defined as above. 
	We assume that
	\begin{enumerate}
		\item $<$ is a partial order;
		\item for each $v\in V$ of type $\hat s_1$ or $\hat s_n$, the vertex set of $\lk(v,X)$ with the induced order from $(V,\le)$ is a bowtie free poset; 
		\item for any embedded 4-cycle $x_1y_1x_2y_2$ in $X$ such that $x_1,x_2$ have type $\hat s_1$ and $y_1,y_2$ have type $\hat s_n$, there is a vertex $z\in V$ such that $x_i\le z\le y_j$ for $1\le i,j\le 2$.
	\end{enumerate}
	Then $(V,\le)$ is bowtie free. 
\end{lem}

\begin{lem}
	\label{lem:weakly flag equivalent}
	Suppose $A_\Lambda$ is an Artin group whose Coxeter diagram $\Lambda$ contains a linear admissible subgraph $\Lambda'$ with its three consecutive vertices being $\{b_2,a,b_3\}$.	
	If $\Delta=\Delta_{\Lambda,\Lambda}$ is bowtie free and weakly flag (Definition~\ref{def:bowtie free}), then the $(b_2,b_3)$-subdivision of $\Delta=\Delta_{\Lambda,\Lambda'}$ (Definition~\ref{def:subdivision}), denoted $\Delta'$, viewed as a simplicial complex of type $S$ with $S=\{1,2,3\}$, is bowtie free and downward flag.
\end{lem}

\begin{proof}
Let $t$ be the type function on the vertex set $V\Delta'$ of $\Delta'$ as in Definition~\ref{def:subdivision}.	We verify the bowtie free condition using Lemma~\ref{lem:bowtie free criterion}. For each $y\in \Delta'$ with $t(y)=1$, $\lk(y,\Delta')$ and $\lk(y,\Delta)$ are isomorphic, moreover, such isomorphism preserves the order of vertices (inherited respectively from $(V\Delta',\le)$ and from $\Delta$ with an appropriate choice of linear order on $\{b_2,a,b_3\}$). As $\Delta$ is bowtie free, $\lk(y,\Delta)$ is bowtie free, hence the same holds for $\lk(y,\Delta')$. It remains to verify Assumption 3 of Lemma~\ref{lem:bowtie free criterion}.
	Given $\{x_1,x_2,y_1,y_2\}\subset (V\Delta',\le)$ with $x_i<y_j$ for $1\le i,j\le 2$, $t(x_1)=t(x_2)=1$ and $t(y_1)=t(y_2)=3$. Then $y_1$ and $y_2$ have type $\hat a$, and $x_1$ and $x_2$ have type $\hat b_2$ or $\hat b_3$. 
	If $\type(x_1)=\type(x_2)$, then the bowtie free assumption  on $\Delta$ implies that either $x_1=x_2$ or $y_1=y_2$. If $\type(x_1)\neq \type(x_2)$, then the bowtie free assumption  on $\Delta$ implies that $x_1$ and $x_2$ are adjacent in $\Delta$. Then $x_1,x_2,y_j$ span a triangle in $\Delta$ for $j=1,2$ as $\Delta$ is flag.
	Let $z\in\Delta'$ be the vertex corresponding to the midpoint of $x_1$ and $x_2$. Then $x_i\le z\le y_j$ for $1\le i,j\le 2$ in $\Delta'$.
	
	Now we verify the downward flagness condition for $(V\Delta',\le)$.
 Take $\{x_1,x_2,x_3\}\subset (V\Delta',\le)$ such that $x_i$ and $x_{i+1}$ has a common lower bound $y_i$ for $i\in \mathbb Z/3\mathbb Z$. Assume without loss of generality that $t(y_1)=t(y_2)=t(y_3)=1$ (if $t(y_i)=2$, then we find $y'_i<y_i$ in $V\Delta'$ and replace $y_i$ by $y'_i$). Let $\omega$ be the 6-cycle in $\Delta'$ of form $x_1y_1x_2y_2x_3y_3$. We assume $\omega$ is embedded, otherwise the downward flagness is trivial.
	
	First we consider the case that $t(x_1)=t(x_2)=t(x_3)=3$. Then we can view $\omega$ as a 6-cycle in the 1-skeleton of $\Delta$. Let $P$ be the vertex set of $\Delta$ with induced order from $b_2<a<b_3$. If $y_1,y_2,y_3$ are all of type $\hat b_2$ in $\Delta$, then none of $x_i$ and $y_j$ are maximal in $P$. Hence the weakly flagness of $P$ implies that $\{y_1,y_2,y_3\}$ has a common upper bound, denoted $z$, in $P$. By considering the 4-cycle spanned by $\{z,y_1,y_2,x_1\}$ and applying the bowtie free property, we know either $x_1$ is adjacent to $z$ in $\Delta$, or $x_1=z$. The same statement holds true if we replace $x_1$ by $x_2$ or $x_3$. As $\{x_1,x_2,x_3\}$ are pairwise distinct and they have the same type in $\Delta$, we know $z$ is of type $\hat b_3$, and $z$ is adjacent to each of $\{x_1,x_2,x_3\}$ in $\Delta$, hence also in $\Delta'$. As $t(z)=1$, we $z$ is a common lower bound for $\{x_1,x_2,x_3\}$ in $(V\Delta',\le)$.
		 The case $y_1,y_2,y_3$ are all of type $\hat b_3$ is similar.
		 If two of $\{y_1,y_2,y_3\}$ have different types in $\Delta$, then we can assume without loss of generality that $y_1$ is of type $\hat b_2$ and $y_2,y_3$ are of type $\hat b_3$. Then $y_1$ and $y_i$ are adjacent in $\Delta$ for $i=2,3$. As $x_3$ is of type $\hat a$ in $\Delta$, by applying the bowtie free condition to $\{y_1,y_2,y_3,x_3\}$, we know $y_1$ and $x_3$ are adjacent in $\Delta$, hence in $\Delta'$. Thus $y_1$ is a common lower bound for $\{x_1,x_2,x_3\}$ in $V\Delta'$.

	Suppose the $t$-value of exactly two of $\{x_1,x_2,x_3\}$ is $3$. Assume without loss of generality that $t(x_1)=t(x_2)=3$ and $t(x_3)=2$. Then $y_2$ and $y_3$ are adjacent in $\Delta$ and $x_3$ is the midpoint between $y_2$ and $y_3$. We assume without loss of generality that $\type(y_2)=\hat b_2$ and $\type(y_3)=\hat b_3$. If $\type(y_1)=\hat b_3$, then $y_2$ is adjacent to both $y_1$ and $y_3$ in $\Delta$. 
	By applying the bowtie free condition of $\Delta$ to $\{y_1,y_2,y_3,x_1\}$, we know $x_1$ and $y_2$ are adjacent in $\Delta$. Thus $y_2$ is a lower bound for $\{x_1,x_2,x_3\}$ in $V\Delta'$. The case $\type(y_1)=\hat b_2$ is symmetric. 
	
	The case that at most one element in $\{x_1,x_2,x_3\}$ has $t$-value $3$ is similar and simpler, which we leave to the reader.
\end{proof}

Now the following is a combination of Proposition~\ref{prop:ori link0} and Lemma~\ref{lem:weakly flag equivalent}.
\begin{prop}
	\label{prop:ori link}
	Suppose $A_\Lambda$ is an Artin group whose Coxeter diagram $\Lambda$ contains a star shaped induced subgraph $\Lambda'$ made of three edges glued at a common node $a$. Let the other three nodes be $b_1,b_2,b_3$. 	Suppose $\Lambda'$ is an admissible subgraph of $\Lambda$.
	
	For $i=1,2,3$, let $\Lambda_i$ be the connected component of $\Lambda\setminus \{b_i\}$ that contains $a$. Let $\Lambda'_i=\Lambda'\setminus \{b_i\}$.
	Suppose that the following holds:
	\begin{enumerate}
		\item for $i=2,3$, the vertex set of the relative Artin complex $\Delta_{\Lambda_i,\Lambda'_i}$, endowed with the order induced from $b_i<a<b_1$, is a bowtie free, upward flag poset;
		\item $\Delta_{\Lambda_1,\Lambda'_1}$ is bowtie free and weakly flag.
	\end{enumerate}
	Then the $(b_2,b_3)$-subdivision of $\Delta=\Delta_{\Lambda,\Lambda'}$, denoted $\Delta'$, viewed as a simplicial complex of type $S$ with $S=\{1,2,3,4\}$, satisfies all the assumptions of Theorem~\ref{thm:contractibleII}. Hence $\Delta'$ is contractible. Thus $\Delta$ is contractible.
\end{prop}

\section{Propagation of bowtie free and flagness}
\label{sec:propagation}
We discuss several propagation results in the sense that if we know bowtie free or flagness properties on the links of some relative Artin complexes, then we can deduce that those relative Artin complexes also satisfy similar properties, under suitable assumptions.

\subsection{Case $\widetilde B_n$}

\begin{prop}
	\label{cor:propagation}
	Let $\Lambda,\Lambda',\Lambda_i,\Lambda'_i,\{b_i\}_{i=1}^n$ be as in Proposition~\ref{prop:ori link0}. Suppose all the assumptions in Proposition~\ref{prop:ori link0} holds true. Then the following holds true:
	\begin{enumerate}
		\item For $i=1,2$, the vertex set of the relative Artin complex $\Delta_{\Lambda,\Lambda'_i}$, endowed with the order induced from $b_i<b_3<\cdots<b_{n+1}$, is a bowtie free, upward flag poset.
		\item The $(b_1,b_2)$-subdivision of $\Delta_{\Lambda,\Lambda'_{n+1}}$ is bowtie free and downward flag.
		\item Assume in addition that $\Delta_{\Lambda_{n+1},\Lambda'_{n+1}}$ satisfies the labeled 4-wheel condition, then $\Delta_{\Lambda,\Lambda'}$ satisfies the labeled 4-wheel condition.
	\end{enumerate}
\end{prop}

\begin{proof}
Let $\Delta'$ be the subdivision of $\Delta$ as in Definition~\ref{def:subdivision}, whose vertex set is endowed with the partial order as in Definition~\ref{def:subdivision}.	
For the first assertion, note that the embedding $\Delta_{\Lambda,\Lambda'_i}\to \Delta'$ preserves the order on the respective vertex sets.	Take pairwise distinct vertices $\{x_1,x_2,x_3\}$ in $\Delta_{\Lambda,\Lambda'_i}$ such that $x_i$ and $x_{i+1}$ has a common upper bound $y_i$ in $\Delta_{\Lambda,\Lambda'_i}$ for $i\in \mathbb Z/3\mathbb Z$. Viewing $x_i$ and $y_j$ as vertices in $\Delta'$ and applying Proposition~\ref{prop:ori link0} and Lemma~\ref{lem:big lattice}, we know $\{x_1,x_2,x_3\}$ have a common upper bound $z$ in $V\Delta'$. Note that $t(z)\neq 2$, as $t(z)=2$ implies that there are at most two elements in $V\Delta'$ which are $<z$, contradicting $\{x_1,x_2,x_3\}$ being pairwise distinct. Hence $z\in \Delta_{\Lambda,\Lambda'_i}$, which implies $z$ is a common upper bound for $\{x_1,x_2,x_3\}$ in $\Delta_{\Lambda,\Lambda'_i}$. Similarly, we can deduce Assertion 2 from Lemma~\ref{lem:big lattice} and Proposition~\ref{prop:ori link0}.

For Assertion 3, by Assertion 1 and Lemma~\ref{lem:4wheel}, it suffices to show $\Delta_{\Lambda,\Lambda'_0}$ is bowtie free, where $\Lambda'_0=\overline{b_1b_3}\cup\overline{b_2b_3}$.
We use Lemma~\ref{lem:bowtie free criterion}. Let $\Delta'_{\Lambda,\Lambda'_0}$ be the subdivision of $\Delta_{\Lambda,\Lambda'_0}$ induced by $\Delta'$. Then $\Delta'_{\Lambda,\Lambda'_0}$ is the full subcomplex of $\Delta'$ spanned by vertices of type $1,2$ and $3$.
Take a vertex $v\in \Delta_{\Lambda,\Lambda'_0}$ of type $\hat b_i$ for $i=1$ or $2$. Then $\lk(v,\Delta_{\Lambda,\Lambda'_0})$ and $\lk(v,\Delta'_{\Lambda,\Lambda'_0})$ are isomorphic. As $\lk(v,\Delta'_{\Lambda,\Lambda'_0})$ is bowtie free by Proposition~\ref{prop:ori link0} and Lemma~\ref{lem:big lattice}, we know $\lk(v,\Delta_{\Lambda,\Lambda'_0})$ is bowtie free.

It remains to verify Assumption 3 of Lemma~\ref{lem:bowtie free criterion} for $\Delta_{\Lambda,\Lambda'_0}$. We refer to Figure~\ref{fig:8cycle}. Take a 4-cycle $x_1x_2x_3x_4$ in $\Delta_{\Lambda,\Lambda'_0}$ such that $x_i$ is of type $\hat b_1$ (resp. $\hat b_2$) for $i=1,3$ (resp. $i=2,4$). Let $Y$ be the thickening of $\Delta'$, in the sense defined in Theorem~\ref{thm:contractibleII}. By Proposition~\ref{prop:ori link0} and Theorem~\ref{thm:contractibleII}, $Y$ is a Helly graph. We also view $x_i$ as a vertex in $Y$. Let $y_i$ be the vertex in $\Delta'_{\Lambda,\Lambda'_0}$ which is the middle point between $x_i$ and $x_{i+1}$. For $i=2,3$, let $y'_i$ be a vertex of $\Delta'$ with $t(y'_i)=n+1$ such that $y'_i$ is adjacent in $\Delta'$ to each of $\{x_i,x_{i+1},y_i\}$.
 Let $d_Y$ denotes the combinatorial distance between vertices in $Y$, and let $B_Y(x,m)$ denotes the combinatorial balls in $Y$ centered at $x$ with radius $m$. For two vertices $x,y\in Y$, we write $x\sim_Y y$ if either $x=y$ or $x$ and $y$ are adjacent in $Y$. 

\begin{figure}
	\centering
	\includegraphics[scale=0.68]{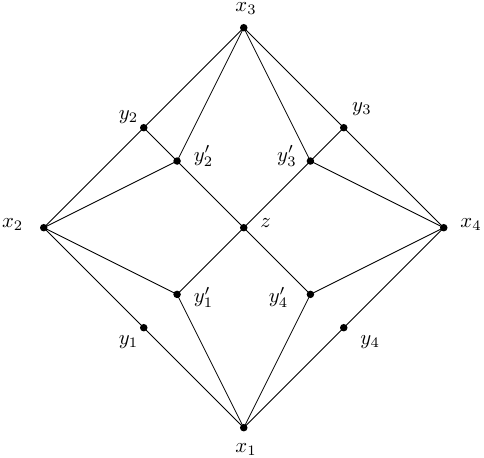}
		\caption{Proof of Proposition~\ref{cor:propagation}.}
	\label{fig:8cycle}
\end{figure}
 
Note that $B_Y(y'_2,1)$, $B_Y(y'_3,1)$ and $B_Y(x_1,2)$ pairwise intersect, thus they have a common intersection in $Y$, denoted by $z$. Thus $z\sim_Y y'_i$ for $i=2,3$. As $t(y'_i)=n+1$, it follows from the definition of edges in $Y$ that $z\sim_{\Delta'} y'_i$. Now we consider $B_Y(x_2,1)$, $B_Y(z,1)$, and $B_Y(x_1,1)$, which pairwise intersect, and let $y'_1$ be a vertex in the common intersection. We claim that we can choose $y'_1$ such that $t(y'_1)=n+1$ and $$y'_1\in B_{\Delta'}(x_2,1)\cap B_{\Delta'}(z,1)\cap B_{\Delta'}(x_1,1).$$
Indeed, if $t(y'_1)\neq n+1$, then the definition of edges in $Y$ implies that there are vertices $w_1,w_2\in \Delta'$ with $t(w_1)=1$ and $t(w_2)=n+1$ such that $w_1\le \{y'_1,z\}\le w_2$ in $(V\Delta',\le)$. In particular, $w_2\sim_{\Delta'} z$. As $t(x_2)=t(x_1)=1$, we know $y'_1\sim_{\Delta'} x_i$ for $i=1,2$. Thus $x_i\le y'_1$ in $(V\Delta',\le)$ for $i=1,2$. Since $y'_1\le w_2$, we have $x_i\le w_2$ for $i=1,2$ in $(V\Delta',\le)$. Thus the claim is proved if we replace $y'_1$ by $w_2$. Similarly, we can choose $y'_4$ such that $t(y'_4)=n+1$ and $$y'_4\in B_{\Delta'}(x_4,1)\cap B_{\Delta'}(z,1)\cap B_{\Delta'}(x_1,1).$$
If $t(z)=n+1$, then $z=y'_1=y'_2=y'_3=y'_4$. Thus $x_i\in \lk(z,\Delta)$ for $1\le i\le 4$. As $\lk(z,\Delta)\cong \Delta_{\Lambda_{n+1},\Lambda'_{n+1}}$ satisfies labeled 4-wheel condition, there is $z'\in \lk(z,\Delta)$ of type $\hat b_3$ such that $z'$ is adjacent to each of $x_i$ in  $\lk(z,\Delta)$ (hence in $\Delta_{\Lambda,\Lambda'_{n+1}}$), as desired. 
Now we assume $t(z)\le n$. As $y'_i>z$ for each $i$, by replacing $z$ by an element which is less than $z$ in $(V\Delta',\le)$, we can assume that $t(z)=1$ and $z$ is adjacent to each $y'_i$ in $\Delta'$. Then we can view each $x_i$ and $z$ as vertices in $\Delta$.

First we assume $z\notin\{x_1,x_2,x_3,x_4\}$. We claim there is a vertex $z'$ with $t(z')=n+1$ such that $z'$ is adjacent in $\Delta$ to at least three of $\{x_1,x_2,x_3,x_4\}$. Now we prove the claim. As each pair from $\{y'_2,y_1,y'_4\}$ have a lower bound, there is a common lower bound $w$ of $\{y'_2,y_1,y'_4\}$ by Lemma~\ref{lem:big lattice}. We can assume $t(w)=1$. As $w<y_1$, and in $\Delta'$ there is only two vertices with $t$-value $1$ that are below $y_1$, we know $w=x_2$ or $x_1$.
If $w=x_2$, then $y'_4$ is adjacent in $\Delta$ to each of $\{x_2,x_1,x_4\}$ and the claim follows. If $w=x_1$, then $y'_2$ is adjacent in $\Delta$ to each of $\{x_2,x_1,x_3\}$ and the claim follows.

Now we claim the same $z'$ is adjacent to each of $\{x_1,x_2,x_3,x_4\}$.
Assume without loss of generality that $z'$ is adjacent to $\{x_1,x_2,x_3\}$. As each pair from $\{z',y_3,y_4\}$ have a lower bound, there is a common lower bound $w$ of $\{z',y_3,y_4\}$ by Lemma~\ref{lem:big lattice}. We can assume $t(w)=1$. As $w<y_3$, the argument in the previous paragraph implies that $w\in \{x_3,x_4\}$. Similarly, $w<y_4$ implies $w\in \{x_4,x_1\}$. Thus $w=x_4$, which implies that $z'$ is adjacent to $x_4$ in $\Delta$. This claim implies that the 4-cycle $x_1x_2x_3x_4$ are contained in $\lk(z',\Delta)$. As $\lk(z',\Delta)\cong \Delta_{\Lambda_{n+1},\Lambda'_{n+1}}$ are assumed to satisfy labeled 4-wheel condition, we know from Lemma~\ref{lem:4wheel} that there is a vertex $z''\in \lk(z',\Delta)$ of type $\hat b_3$ such that $z''$ is adjacent in $\Delta$ to each of $x_i$ for $1\le i\le 4$. 

Suppose $z\in \{x_1,x_2,x_3,x_4\}$, say $z=x_2$. Then $y'_3$ is adjacent in $\Delta$ to each of $\{x_2,x_3,x_4\}$, and we finish as before.
\end{proof}

\subsection{Case $\widetilde D_n$}

\begin{prop}
	\label{cor:propagation2}
	Let $\Lambda,\Lambda',\{\Lambda_{a_i}\}_{i=1}^2,\{\Lambda_{c_i}\}_{i=1}^2,\{\Lambda'_{a_i}\}_{i=1}^2,\{\Lambda'_{c_i}\}_{i=1}^2$ be as in Proposition~\ref{prop:ori link2}. Suppose all the assumptions in Proposition~\ref{prop:ori link2} holds true. 
	Then the following holds true.
	\begin{enumerate}
		\item For $i=1,2$, the $(a_1,a_2)$-subdivision of $\Delta_{\Lambda,\Lambda'_{c_i}}$ is bowtie free and downward flag.
		\item For $i=1,2$, the $(c_1,c_2)$-subdivision of $\Delta_{\Lambda,\Lambda'_{a_i}}$ is bowtie free and downward flag.
		\item Let $\Lambda''=\Lambda'_{c_j}\cap \Lambda'_{a_i}$ for $1\le i,j\le 2$. Then $\Delta_{\Lambda,\Lambda''}$ is bowtie free.
	\end{enumerate}
\end{prop}

\begin{proof}
	Let the complex $\Delta'$ and the poset $(V\Delta',\le)$ be as in Definition~\ref{def:subdivisionD}.	
	By Lemma~\ref{lem:big lattice}, $(V\Delta',\le)$ is bowtie free and flag. To prove Assertion (1), we view the $(a_1,a_2)$-subdivision of $\Delta_{\Lambda,\Lambda'_{c_i}}$ (denoted by $\Delta'_{a_1,a_2}$) as the full subcomplex of $\Delta'$ span by vertices of types $1,2,\ldots,n+2$, and vertices of type $n+4$ that are also of type $\hat c_i$ in $\Delta_{\Lambda,\Lambda'}$. The order of vertices on  $\Delta'_{a_1,a_2}$ as in Definition~\ref{def:subdivision} coincides with the order inherit from $(V\Delta',\le)$. 
	
	Take three vertices $\{v_1,v_2,v_3\}$ of $\Delta'_{a_1,a_2}$ such that each pair of them have a lower bound. Viewing them as vertices of $V\Delta'$ and using the flagness of $V\Delta'$, we can find a common lower bound $v$ of $\{v_1,v_2,v_3\}$ in $V\Delta'$. If $v\in \Delta'_{a_1,a_2}$, then we are done, otherwise $v$ is of type $n+3$, in which case we find $v'$ of type $n+2$ with $v'<v$, then $v'\in \Delta'_{a_1,a_2}$ and $v'$ is a lower bound of $\{v_1,v_2,v_3\}$. 
	
	For the bowtie free property, take pairwise distinct vertices $\{v_1,v_2,v_3,v_4\}$ of $\Delta'_{a_1,a_2}$ with $v_i\le v_3$ and $v_i\le v_4$ for $i=1,2$. Viewing them as vertices in $V\Delta'$, by Lemma~\ref{lem:big lattice}, we find $v\in V\Delta'$ with $\{v_1,v_2\}\le v\le \{v_3,v_4\}$. 
	If $v\in \Delta'_{a_1,a_2}$, then we are done, otherwise $v$ is of type $n+3$. Each type $n+3$ vertices $<$ exactly two type $n+4$ vertices in $\Delta'$, one with type $\hat c_1$ in $\Delta$ and another one with type $\hat c_2$ in $\Delta$. Thus $v$ can not be of type $n+3$ as $v_3$ and $v_4$ are two different vertices in $\Delta$ with type $\hat c_i$. Thus $\Delta'_{a_1,a_2}$ is bowtie free. This proves Assertion (1). Assertion (2) can be proved similarly.
	
	For Assertion 3, up to symmetry, it suffices to consider $\Lambda''=\Lambda'_{c_1}\cap \Lambda'_{a_1}$. We endow $\Delta_{\Lambda,\Lambda''}$ with the order induced from $a_1<b_1<\cdots<b_n<c_1$. Then $V\Delta_{\Lambda,\Lambda''}$ with this order coincides with the order inherit from $(V\Delta',\le)$. Take pairwise distinct vertices $\{v_1,v_2,v_3,v_4\}$ of $\Delta_{\Lambda,\Lambda''}$ with $v_i<v_3$ and $v_i<v_4$ for $i=1,2$. As $\Delta'$ is bowtie free, there is a vertex $v\in V\Delta'$ with $\{v_1,v_2\}\le v\le \{v_3,v_4\}$. Then $1<t(v)<n+4$. If $t(v)=n+3$, then there are exactly two vertices in $V\Delta'$ which is bigger than $v$, one has type $\hat c_1$ in $\Delta$ and one has type $\hat c_2$ in $\Delta$. As $v_3,v_4$ have the same type in $\Delta$, we can not have $t(v)=n+3$. Similarly $t(v)\neq 2$. Thus $3\le t(v)\le n+2$ and $v\in \Delta_{\Lambda,\Lambda'}$, as desired.
\end{proof}

\subsection{Downward flagness with respect to different subdivisions}
Throughout this subsection, let $\Lambda,\Lambda',\{\Lambda_i\}_{i=1}^3,\{\Lambda'_i\}_{i=1}^3$, $\Delta$ and $\Delta'$ be as in Proposition~\ref{prop:ori link}. Our main goal in this subsection is to prove:
\begin{prop}
	\label{prop:different subdivision}
	Let $\Lambda,\Lambda',\{\Lambda_i\}_{i=1}^3,\{\Lambda'_i\}_{i=1}^3$ be as in Proposition~\ref{prop:ori link}. Suppose all the assumptions in Proposition~\ref{prop:ori link} holds true. Then for any $i\neq j$, the $(b_i,b_j)$-subdivision of $\Delta_{\Lambda,\Lambda'}$ is bowtie free and downward flag.
\end{prop}

We first establish several auxiliary lemmas before we prove Proposition~\ref{prop:different subdivision}.
\begin{lem}
	\label{lem:4and6cycles}
 Let $x_1x_2x_3x_4$ be an embedded four cycle in $\Delta$ such that $\type(x_1)=\type(x_3)=\hat b_i$, $\type(x_2)=\hat b_j$, $\type(x_4)=\hat b_k$ with $\{i,j,k\}$ being pairwise distinct. Then $x_2\sim_\Delta x_4$.
\end{lem}

\begin{proof}
By Proposition~\ref{cor:propagation} and \cite[Proposition 6.15]{huang2023labeled}, the complex $\Delta_{\Lambda,\Lambda'}$ satisfies the labeled four wheel condition in the sense of \cite[Definition 6.12]{huang2023labeled}. Thus if $x_2$ and $x_4$ are not adjacent, then there is a vertex $y\in \Delta$ of type $\hat a$ such that $y$ is adjacent to each of $x_i$. As $\Lambda'$ is admissible in $\Lambda$, Corollary~\ref{cor:adj} implies $x_2\sim_\Delta x_4$, contradiction.	
\end{proof}

	%\item Let $\{x_i\}$ be consecutive vertices of a 6-cycle in $\Delta$ such that $\type(x_i)=\hat b_1$ for $i=2,4,6$, $\type(x_i)\in \{\hat b_2,\hat b_3\}$ for $i=1,3,5$, $\type(x_1)\neq \type(x_i)$ for $i=3,5$ and $\type(x_3)=\type(x_5)$. Then $x_1$ is adjacent in $\Delta$ to each of $x_i$ with $2\le i\le 6$.
%\end{enumerate}

\begin{lem}
	\label{lem:8cycle1}
Let $x_1y_1x_2y'_2x_3y'_3x_4y_4$ be consecutive vertices of a 8-cycle in $\Delta'$ such that $\type(x_1)=\hat b_3$, $t(x_3)=1$; $\type(x_i)=\hat b_2$ for $i=2,4$; $t(y_i)=2$ for $i=1,4$ and $t(y'_i)=4$ for $i=2,3$. Then at least one of the following holds true:
\begin{enumerate}
	\item there exists a vertex $z\in \Delta$ with $t(z)=4$ such that $z$ is adjacent in $\Delta$ to both $x_2$ and $x_4$;
	\item  $x_1$ is adjacent in $\Delta$ to both $y'_2$ and $y'_3$;
	\item there exists a vertex $z\in \Delta$ of type $\hat b_3$ such that it is adjacent in $\Delta$ to each of $\{x_2,y'_2,y'_3,x_4\}$.
\end{enumerate}
\end{lem}

\begin{proof}
We will assume $x_2\neq x_3$ and $x_4\neq x_3$, otherwise we are in Case 1. Assume $x_1\neq x_3$, otherwise we are in Case 2. 
Assume $y'_2\neq y'_3$, otherwise we are in case 1. Repeating the argument in the proof of Proposition~\ref{cor:propagation} (2) (see Figure~\ref{fig:8cycle}, though we caution the reader that now we are not assuming $x_3$ is adjacent in $\Delta$ to $x_2$ and $x_4$), we can find $y'_1$ with $t(y'_1)=4$ such that $y'_1$ is adjacent in $\Delta$ to both $x_1$ and $x_2$, $y'_4$ with $t(y'_4)=4$ such that $y'_4$ is adjacent in $\Delta$ to both $x_1$ and $x_4$, and $z$ with $t(z)=1$ such that $z$ is adjacent in $\Delta$ to each of $\{y'_1,y'_2,y'_3,y'_4\}$.
 We assume $y'_1\neq y'_4$, otherwise we are in case 1.  
 
 If $y'_1=y'_2$, then $\{x_1,x_3,x_4\}$ pairwise have a upper bound in $(V\Delta',<)$. Thus they have a common upper bound by Lemma~\ref{lem:big lattice}, denoted by $w$. We can assume $t(w)=4$. If $w=y'_2$, then we are in case 1. Now suppose $w\neq y'_2$. If $\type(x_3)=\hat b_2$, then by applying Lemma~\ref{lem:4and6cycles} to $x_1y'_2x_3w$, we know $x_1\sim_\Delta x_3$. Applying Lemma~\ref{lem:4and6cycles} to $x_1x_3y'_3x_4$ (which is embedded as $x_3\neq x_4$), we know $x_1\sim_{\Delta} y'_3$, hence we are in case 2. Suppose $\type(x_3)=\hat b_3$.  By considering the $4$-cycle $x_1y'_2x_3w$ (which is embedded as $x_1\neq x_3$) and apply Proposition~\ref{cor:propagation}, we know there exists a vertex $z'$ of type $\hat a$ such that $z'$ is adjacent in $\Delta$ to each vertex of this 4-cycle. If $w=y'_3$, then we are in Case 2. If $w\neq y'_3$, then by applying Lemma~\ref{lem:4and6cycles} to $wx_3y'_3x_4$, we know $x_3\sim_\Delta x_4$. By considering the 4-cycle $x_1z'x_3x_4$ and Proposition~\ref{cor:propagation} (1), we know $z'\sim_\Delta x_4$. Thus $y'_2\sim_\Delta x_4$ and we are in case 1. A similar analysis can be applied if $y'_3=y'_4$. Thus we assume $y'_i\neq y'_{i+1}$ for all $i$ from now on.

If $\type(z)=\hat b_2$, then by Lemma~\ref{lem:4and6cycles}, $x_1\sim_\Delta z$. Assume $z\neq x_2$ and $z\neq x_4$, otherwise we are in case 1.
By applying Lemma~\ref{lem:4and6cycles} to the embedded 4-cycles $x_1zy'_2x_2$ and $x_1zy'_3x_4$ in $\Delta$, we know that $x_1\sim_\Delta y'_i$ for $i=2,3$. Suppose $\type(z)=\hat b_3$. Applying Lemma~\ref{lem:4and6cycles} to the embedded 4-cycles $zy'_2x_2y'_1$ and $zy'_2x_4y'_4$, we know $z\sim_\Delta x_i$ for $i=2,4$, hence we are in case 3. 
\end{proof}

\begin{proof}[Proof of Proposition~\ref{prop:different subdivision}]
The case of $(b_2,b_3)$-subdivision follows from Lemma~\ref{lem:big lattice}. The cases of $(b_1,b_2)$-subdivision and $(b_1,b_3)$-subdivision are symmetric, so we only treat the $(b_1,b_3)$-subdivision $\Delta''$ of $\Delta$. Let $t''$ be the type function on $V\Delta''$. By Lemma~\ref{lem:poset}, $(V\Delta'',<)$ is a poset. The bowtie free property of $(V\Delta'',<)$ follows from that $\Delta_{\Lambda,\Lambda'}$ satisfies the labeled four wheel condition (see Lemma~\ref{lem:4and6cycles}).

Take three different vertices $\{x_1,x_3,x_5\}$ in $\Delta''$ such that they pairwise have a lower bound in $(V\Delta'',<)$. We need to show they have a common lower bound. We will only treat the case when $t''(x_i)=4$ for $i=1,3,5$; as the other cases are similar and much simpler. Let $x_{i+1}$ be a lower bound for $\{x_i,x_{i+2}\}$, with $i=1,3,5\in \mathbb Z/6\mathbb Z$. We can assume $t''(x_i)=1$ for $i=2,4,6$. Then  the vertices $\{x_i\}_{i=1}^6$ form a 6-cycle in $\omega$. We can assume this 6-cycle is embedded, otherwise we are already done. Note that $\omega$ is also a 6-cycle in $\Delta$. Let $m_i\in \Delta'$ be the midpoint of $\overline{x_ix_{i+1}}$. Note that $\type(x_i)=\hat b_2$ for $i$ odd. And $\type(x_i)=\hat b_1$ or $\hat b_3$ for $i$ even.

%, and let $m'_i\in \Delta'$ be a vertex with $t(m'_i)=4$ and $m'_i\sim_{\Delta'}m_i$. Then $m'_i\sim_{\Delta'}x_i$ and $m'_i\sim_{\Delta'}x_{i+1}$.

\smallskip
\noindent
\underline{Case 1: all of $\{x_2,x_4,x_6\}$ has type $\hat b_1$}.
Then $\{x_1,x_3,x_5\}$ are pairwisely upper bounded in the poset $(V\Delta',<)$. By Lemma~\ref{lem:big lattice} and Proposition~\ref{prop:ori link}, there is a vertex $z\in V\Delta'$ which is a common upper bound for $\{x_1,x_3,x_5\}$ in $(V\Delta',<)$. We can assume $t(z)=4$. Then $t''(z)=1$ and $z$ is a common lower bound for $\{x_1,x_3,x_5\}$ in $(V\Delta'',<)$. 

\smallskip
\noindent
\underline{Case 2: exactly two of $\{x_2,x_4,x_6\}$, say $x_2$ and $x_6$, have type $\hat b_1$}. Then Lemma~\ref{lem:8cycle1} applies to the $8$-cycle $x_5x_6x_1x_2x_3m_3x_4m_4$ in $\Delta'$. If we are in case 1 of Lemma~\ref{lem:8cycle1}, then we reduce to the previous paragraph. If we are in case 2 of Lemma~\ref{lem:8cycle1}, then $x_4$ is adjacent in $\Delta$ to $x_6$ and $x_2$. Applying Lemma~\ref{lem:4and6cycles} to the $4$-cycle $x_4x_6x_1x_2$ in $\Delta$, we know $x_4\sim_\Delta x_1$. Thus $x_4$ is a common lower bound of $\{x_1,x_3,x_5\}$ in $(V\Delta'',<)$. If we are in case 3 of Lemma~\ref{lem:8cycle1}, then let $z$ be the vertex of type $\hat b_3$ such that $z\sim_\Delta \{x_5,x_6,x_2,x_3\}$. Applying Lemma~\ref{lem:4and6cycles} to the $4$-cycle $zx_6x_1x_2$ in $\Delta$, we know $z\sim_\Delta x_1$. Thus $z$ is a common lower bound of $\{x_1,x_3,x_5\}$ in $(V\Delta'',<)$.

Before we discuss the remaining cases, we need an extra observation.
Let $Y$ be the graph as in the proof of Proposition~\ref{cor:propagation}. As consecutive vertices of $\omega$ has distance $\le 2$ in $\Delta'$, hence the same holds in $Y$. Thus $\{B_Y(x_i,2)\}_{i=1,3,5}$ pairwise intersects. Let $z\in Y$ be a vertex in their common intersection. Then for $i=1,3,5$, either $z\sim_Y x_i$ or there is $z_i\in Y$ with $z_i\notin \{z,x_i\}$ such that $z_i\sim_Y x_i$ and $z_i\sim_Y z$. 
If $t(z)=4$, then the definition of $Y$ implies that for $i=1,3,5$, either $z\sim_{\Delta'} x_i$, or $z_i\sim_{\Delta'} x_i$ and $z_i\sim_{\Delta'} z$. In the latter case, $t(x_i)<t(z_i)<t(z)$, hence $z\sim_{\Delta'} x_i$ as well. Thus $z$ is a common lower bound for $\{x_1,x_3,x_5\}$ in $(V\Delta'',<)$. Now we assume $t(z)<4$. Then by an argument in the proof of Proposition~\ref{cor:propagation} (2), for each $i=1,3,5$, we can assume $t(z_i)=4$ such that $z_i\sim_{\Delta'} x_i$ and $z_i\sim_{\Delta'} z$; moreover, up to replacing $z$, we can assume $t(z)=1$. Then $\type(z)=\hat b_2$ or $\hat b_3$.

\smallskip
\noindent
\underline{Case 3: exactly one of $\{x_2,x_4,x_6\}$, say $x_6$, has type $\hat b_1$.} First we look at the situation that there exists a vertex $x'_6$ of type $\hat b_1$ such that $x'_6$ is adjacent in $\Delta$ to $\{x_5,z,x_1\}$. See Figure~\ref{fig:12cycle} (I). Now we apply Lemma~\ref{lem:8cycle1} to the 8-cycles:
$$
\omega=x'_6zz_3x_3m_2x_2m_1x_1,\ \omega'=x'_6zz_3x_3m_3x_4m_4x_5.
$$
If at least one of $\omega$ and $\omega'$ are in  Lemma~\ref{lem:8cycle1} (1), then we are reduced to the previous cases. If none of the $\omega$ and $\omega'$ are in Lemma~\ref{lem:8cycle1}, then up to possibly replacing $x_2$ and $x_4$ by different vertices of type $\hat b_3$, we can assume $x_2\sim_{\Delta} \{x_1,x'_6,z_3,x_3\}$ and $x_4\sim_{\Delta} \{x_3,x'_6,z_3,x_5\}$. As $x_2\neq x_4$ (otherwise we are done), by applying Lemma~\ref{lem:4and6cycles} to the 4-cycle $x'_6x_2x_3x_4$ in $\Delta$, we know $x'_6\sim_\Delta x_3$. Hence $x'_6$ is common lower bound for $\{x_1,x_3,x_5\}$ in $(V\Delta'',<)$.

\begin{figure}[h]
	\label{fig:12cycle}
	\centering
	\includegraphics[scale=1.2]{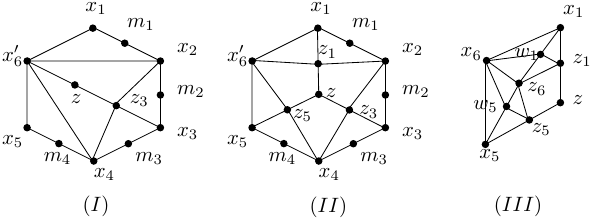}
	\caption{Proof of Proposition~\ref{prop:different subdivision}}
\end{figure}

Now we assume there does not exists any vertex of type $\hat b_1$ such that it is adjacent in $\Delta$ to $\{x_5,z,x_1\}$. In particular, $z_1\neq x_6$, $z_5\neq x_6$ and $z_1\neq z_5$. We claim there exists a vertex $x'_6$ of type $\hat b_3$ such that it is adjacent in $\Delta$ to each of $\{z_1,z_5,x_1,x_5\}$. Assuming the claim for the moment, then we apply Lemma~\ref{lem:8cycle1} to the 8-cycles:
$$
\omega''=zz_1x_1m_1x_2m_2x_3z_3,\ \omega'''=zz_3x_3m_3x_4m_4x_5z_5.
$$
See Figure~\ref{fig:12cycle} (II).
If one of these two cycles are in Lemma~\ref{lem:8cycle1} (1), then we are reduced to previous cases. Otherwise up to replacing $x_2$ and $x_4$ by different vertices of type $\hat b_3$, we can assume $x_2\sim_\Delta\{x_1,z_1,z_3,x_3\}$ and $x_4\sim_\Delta\{x_3,z_3,z_5,x_5\}$. Considering the $6$-cycle $x'_6z_1x_2z_3x_4z_5$ in $\Delta$, we know $\{x'_6,x_2,x_4\}$ are pairwisely upper bounded in $(V\Delta',<)$, hence they have a common upper bound, denoted by $w$, in $(V\Delta',<)$. We can assume $t(w)=4$. If $\{x'_6,x_2,x_4\}$ are pairwise distinct, then by applying Lemma~\ref{lem:4and6cycles} to the 4-cycles $x'_6x_1x_2w$, $wx_2x_3x_4$ and $wx_4x_5x'_6$ in $\Delta$, we know $w$ is adjacent in $\Delta$ to $x_i$ for $i=1,3,5$. Thus $w$ is a common lower bound for $\{x_1,x_3,x_5\}$ in $(V\Delta'',<)$. If at least two of $\{x'_6,x_2,x_4\}$ is the same, say $x_2=x_4$, then $x_2$ is a common lower bound for $\{x_1,x_3,x_5\}$ in $(V\Delta'',<)$.

It remains to prove the claim. See Figure~\ref{fig:12cycle} (III). As $\{x_6,z_1,z_5\}$ is pairwisely lower bounded in $(V\Delta',<)$, Lemma~\ref{lem:big lattice} and Proposition~\ref{prop:ori link} imply they have a common lower bound in the poset $(V\Delta',<)$, denoted by $z_6$. Moreover, we can require $t(z_6)=1$. As there does not exist any vertex of type $\hat b_1$ adjacent in $\Delta$ to each of $\{x_5,x_1,z\}$, we know $z_6\neq z$, $z_6\neq x_1$, $z_6\neq x_5$, $z_1\neq z_5$, $z_1\neq x_6$ and $z_5\neq x_6$.
If $\type(z_6)=\hat b_3$, by applying Lemma~\ref{lem:4and6cycles} to the embedded 4-cycles $z_6x_6x_1z_1$, $z_6z_1zz_5$ and $z_6z_5x_5x_6$, we know $z_6$ is adjacent in $\Delta$ to each of $\{x_1,x_5\}$. Thus the claim follows. Now assume $\type(z_6)=\hat b_2$. We will show this is impossible. By considering the 4-cycles $z_6x_6x_1z_1$ and $z_6z_5x_5x_6$, we deduce from Proposition~\ref{cor:propagation} (1) that there is a vertex $w_1$ of type $\hat a$ adjacent in $\Delta$ to each vertex of $z_6x_6x_1z_1$, and a vertex $w_5$ of type $\hat a$ adjacent in $\Delta$ to each vertex of $z_6z_5x_5x_6$. Now $\{w_1,w_5,z\}$ is pairwisely upper bounded in $(V\Delta',<)$, then they have a common upper bound, denoted $x'_6$. As $w_1\neq w_5$ (otherwise $x_1\sim_\Delta z_5$, hence $z_5$ is adjacent to each of $x_1,x_5,z$, which was forbidden), we know $t(x'_6)=4$. Applying the bowtie free property of $(V\Delta',<)$ to $\{x_6,w_1,x'_6,w_5\}$, we know $x_6=x'_6$. As $x'_6\sim_\Delta z$, we have $x_6\sim_\Delta z$. Hence $x_6$ is adjacent to each of $x_1,x_5,z$, which is forbidden.

\smallskip
\noindent
\underline{Case 4: $\type(x_i)=\hat b_3$ for $i=2,4,6$.}
Now we consider the following three 8-cycles:
$$\omega_2=zz_1x_1m_1x_2m_2x_3z_3,\  \omega_4=zz_3x_3m_3x_4m_4z_5z_5,\ \omega_6=zz_5x_5m_5x_6m_6x_1z_1.$$
If one of these 8-cycles belong to Lemma~\ref{lem:8cycle1} (1), then we are reduced to previous cases. If none of them belong to Lemma~\ref{lem:8cycle1} (1), then for $i=2,4,6$, up to possibly replacing $x_i$ by another vertex of type $\hat b_3$, we can assume $x_i\sim_\Delta\{x_{i-1},z_{i-1},x_{i+1},z_{i+1}\}$. Now we repeat the argument in Case 3 to finish the proof.
\end{proof}

\section{$K(\pi,1)$-conjecture for some 3D Artin groups}
\label{sec:3-dim}

The goal of this section is to prove Corollary~\ref{cor:3hyperbolic}. We start with a preparatory lemma.
\begin{lem}(\cite[Lemma 6.6]{huang2023labeled})
	\label{lem:inherit and girth}
	Suppose $\Lambda'$ is an admissible linear subgraph of the Coxeter diagram $\Lambda$ of $A_S$ and suppose the consecutive nodes of $\Lambda'$ are $S'=\{s_i\}_{i=1}^n$.  Let $\Delta$ be the $(\Lambda,\Lambda')$-relative Artin complex. Let $V$ be the vertex set of $\Delta$. We fix an order on $S'$ such that it is compatible with one of the two linear orders on $\Lambda'$, and endow $V$ with the induced order. Let $\Lambda''$ be a linear subgraph of $\Lambda'$. Then the following holds.
	\begin{enumerate}
		\item If $\Delta_{\Lambda,\Lambda'}$ is bowtie free, then $\Delta_{\Lambda,\Lambda''}$ is bowtie free. 
		\item If $\Lambda''$ contains the largest (resp. smallest) node of $\Lambda'$ and $\Delta_{\Lambda,\Lambda'}$ is upward (resp. downward) flag, then $\Delta_{\Lambda,\Lambda''}$ is upward (resp. downward) flag.
		\item Suppose $\Lambda'$ only have two nodes.  If $\Delta_{\Lambda,\Lambda'}$ is bowtie free, then it is a graph with girth $\ge 6$. If in addition $\Delta_{\Lambda,\Lambda'}$ is upward flag or downward flag, then it is a graph with girth $\ge 8$. 
	\end{enumerate}
\end{lem}

\begin{proof}
	Assertions (1) and (2) follow from definition. Now we prove assertion (3).  Suppose there is a 4-cycle in $\Delta_{\Lambda,\Lambda'}$ with consecutive its vertices $x_1,y_1,x_2,y_2$. Then up to a cyclic permutation of these vertices, we know  $x_i\le y_j$ for $1\le i,j\le 2$. By the bowtie free condition, there exists $z\in P$ such that $x_i\le z\le y_j$ for $1\le i,j\le 2$. If $z$ has the same type as $x_1$, then $x_1=x_2$. If $z$ has the same type as $y_1$, then $y_1=y_2$. Thus there are no embedded 4-cycles in $\Delta_{\Lambda,\Lambda'}$. Similarly we deduce no embedded 6-cycles under the additional assumption of one side flagness.
\end{proof}
\begin{prop}
	\label{prop:not compact}
	Let $A_S$ be an Artin group with $|S|=4$ which is free-of-infinity (i.e. its Coxeter diagram does not have $\infty$-labeled edges). If there exists $S'\subset S$ with $|S'|=3$ such that $A_{S'}$ is not spherical, then the Artin complex for $A_S$ is contractible.
\end{prop} 

\begin{proof}
	Let $S=\{a,b,c,d\}$ and $S'=\{a,b,c\}$. Let $\Delta_S$ (resp. $\Delta_{S'}$) be the Artin complex for $A_{S}$ (resp. $A_{S'}$). By Lemma~\ref{lem:link}, the link of each vertex of type $\hat d$ in $\Delta_S$ is isomorphic to $\Delta_{S'}$. As $A_{S'}$ is not spherical, it is a 2-dimensional Artin group. It is known the modified Deligne complex, defined in \cite{CharneyDavis}, for any 2-dimensional Artin group is contractible \cite{CharneyDavis}; and the modified Deligne complex isomorphic to the barycentric subdivision of $\Delta_{S'}$. Thus the link of each vertex of type $\hat d$ is contractible. Then $\Delta_S$ is homotopic equivalent to the relative Artin complex $\Delta'=\Delta_{S,S'}$ by Lemma~\ref{lem:dr}.
	
	Take a vertex $w\in\Delta'$ of type $\hat a$. We claim the girth of $\lk(w,\Delta')$ is $\ge 2 m_{bc}$ if $m_{bc}\neq 5$, and the girth is $\ge 8$ if $m_{bc}=5$. Assuming the claim is true, we now deduce the theorem as follows. The quotient of $\Delta'$ by $A_S$ is a triangle $T$, with three vertices of type $\hat a,\hat b$ and $\hat c$ respectively. As $A_{S'}$ is not spherical, we know $\frac{1}{m_{ab}}+\frac{1}{m_{bc}}+\frac{1}{m_{ac}}\le 1$. Let $m'_{ab}=m_{ab}$ if $m_{ab}\neq 5$ and $m'_{ab}=4$ if $m_{ab}=5$. Similarly, we define $m'_{bc}$ and $m'_{ac}$. Then we still have $\frac{1}{m'_{ab}}+\frac{1}{m'_{bc}}+\frac{1}{m'_{ac}}\le 1$. Thus we can realize $T$ as a geodesic triangle in $\mathbb E^2$ or $\mathbb H^2$ such that the angle at $\hat a$ is $\frac{\pi}{m'_{bc}}$, the angle at $\hat b$ is $\frac{\pi}{m'_{ac}}$ and the angle at $\hat c$ is $\frac{\pi}{m'_{ab}}$. This metric on $T$ induces a piecewise Euclidean (or hyperbolic) metric on $\Delta'$. The claim implies that $\Delta'$ is locally CAT$(0)$ (or CAT$(-1)$) with such a metric. On the other hand, $\Delta'$ is simply-connected (cf. Lemma~\ref{lem:sc}). Thus $\Delta'$ is CAT$(0)$ (or CAT$(-1)$), hence contractible.
	
	It remains to prove the claim. By Lemma~\ref{lem:link} $\lk(w,\Delta_S)\cong \Delta_{\{b,c,d\}}$ and $\lk(w,\Delta')\cong \Delta_{\{b,c,d\},\{b,c\}}$. The claim is clear if $m_{bc}=2$, as $\lk(w,\Delta')$ is a bipartite graph. We assume $m_{bc}\neq 2$ from now on.
	First we consider the case when $A_{bcd}$ is spherical. Note that the claim follows \cite[Lemma 39]{crisp2005automorphisms} if $m_{bd}=m_{cd}=2$. Now we assume $A_{bcd}$ is irreducible. Then $m_{bc}=3,4$ or $5$. If $m_{bc}=3$, then by Theorem~\ref{thm:bowtie free} and Lemma~\ref{lem:inherit and girth}, $\lk(w,\Delta')$ has girth $\ge 6$. If $m_{bc}=4$ or $5$, then $A_{bcd}$ is either of type $B_3$ or $H_3$. By Theorem~\ref{thm:tripleH}, Theorem~\ref{thm:triple} and Lemma~\ref{lem:inherit and girth}, $\lk(w,\Delta')$ has girth $\ge 8$.

	Now we look at the case $A_{bcd}$ is not spherical. Then $\lk(w,\Delta_S)$ quotiented by the action of $A_{bcd}$ is a triangle $T'$, whose vertices are denoted by $\hat b,\hat c,\hat d$. As $\frac{1}{m_{bc}}+\frac{1}{m_{cd}}+\frac{1}{m_{bd}}\le 1$, we can realize $T'$ as a geodesic triangle in $\mathbb E^2$ or $\mathbb H^2$ such that the angle at $\hat b$ is $\frac{\pi}{m_{cd}}$, the angle at $\hat c$ is $\frac{\pi}{m_{bd}}$ and the angle at $\hat d$ is $\frac{\pi}{m_{bc}}$. This induces a metric on $\lk(w,\Delta_S)$ which is known to be CAT$(0)$ (cf. \cite{CharneyDavis}). As $\lk(w,\Delta')$ is the full subcomplex of $\lk(w,\Delta)$ spanned by vertices of type $\hat b$ and $\hat c$, we know $\lk(w,\Delta')$ has girth $2m_{bc}$ by \cite[Lemma 9.8 (1)]{huang2023labeled}. Thus the claim is proved.
\end{proof}

\begin{thm}
	\label{thm:4generator}
	Suppose $A_S$ is an Artin group with $|S|\le 4$. Assume that its Coxeter diagram is not a $(3,5,3)$-linear diagram. Then $A_S$ satisfies the $K(\pi,1)$-conjecture.
\end{thm}

\begin{proof}
	We assume $A_S$ is irreducible. By \cite{ellis2010k}, it suffices to consider the case when $A_S$ is free-of-infinity. We also assume $A_S$ is not spherical, otherwise the result follows from \cite{deligne}.
When $|S|\le 3$, $A_S$ is 2-dimensional and the theorem follows from \cite{CharneyDavis}. Now we assume $|S|=4$. If there exists a subset $S'\subset S$ such that $|S'|=3$ and $A_{S'}$ is not spherical, then the theorem follows from Proposition~\ref{prop:not compact}. It remains to consider the case when $|S|=4$ and $A_{S'}$ is spherical whenever $S'$ is a 3-element subset of $S$.  By a result of Lanner \cite{MR42129}, $A_S$ is either an affine Artin group, or a 3-dimensional hyperbolic cocompact tetrahedron group, i.e. its associated Coxeter group acts on $\mathbb H^3$ properly and cocompactly by isometries such that its fundamental domain is a tetrahedron. The $K(\pi,1)$-conjecture for affine Artin groups are proved in \cite{paolini2021proof}. For the 3-dimensional hyperbolic tetrahedron groups, there are only nine of them, five has Coxeter diagram being a cycle, which is treated in \cite{Garside,huang2023labeled}. Then remaining four groups have Coxeter diagrams as in Figure~\ref{fig:remain}. By Theorem~\ref{thm:kpi1}, it suffices to show the associated Artin complex is contractible.
	\begin{figure}[h!]
		\centering
		\includegraphics[scale=0.85]{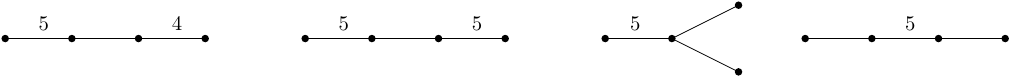}
		\caption{Four remaining cases.}
		\label{fig:remain}
	\end{figure}
	
	By Theorem~\ref{thm:kpi1}, it suffices to show the Artin complexes associated with these Coxeter diagrams are contractible.  First we look at the case of $(5,3,4)$ and $(5,3,5)$ linear diagrams. Assume the nodes in the Coxeter diagram are ordered from left to right, and we endow the vertex set of the associated Artin complex $\Delta$ with the induced order. Then the contractibility of $\Delta$ follows from Theorem~\ref{thm:contractibleII}, Theorem~\ref{thm:tripleH}, Theorem~\ref{thm:triple}, Theorem~\ref{thm:bowtie free} and Lemma~\ref{lem:inherit and girth}. In the case when $\Lambda$ is a tripod, we deduce the contractibility of the Artin complex from Proposition~\ref{prop:ori link}, Theorem~\ref{thm:tripleH}, Theorem~\ref{thm:weaklyflagA}, and Theorem~\ref{thm:bowtie free}.
\end{proof}

\begin{cor}
	\label{cor:3hyperbolic}
	Let $W_S$ be a reflection group acting properly on $\mathbb H^3$ with finite volume fundamental domain. Assume the Coxeter diagram $\Lambda$ of $W_S$ is not the linear graph with consecutive edges labeled by $(3,5,3)$. Then the $K(\pi,1)$-conjecture holds true for the associated Artin group $A_S$.
\end{cor}

\begin{proof}
	By \cite{ellis2010k}, it suffices to show if $\Lambda'$ is an induced subgraph of $\Lambda$ which is free-of-infinity, then the $K(\pi,1)$-conjecture holds for $A_{\Lambda'}$. We claim $\Lambda'$ has at most 4 nodes. The corollary follows from this claim and Theorem~\ref{thm:4generator}, as if $\Lambda$ contain the Coxeter diagram $[3,5,3]$ as an induced subgraph and $W_\Lambda$ is a 3-dimensional hyperbolic reflection group, then $\Lambda=[3,5,3]$. It remains to show the claim. Let $P\subset \mathbb H^3$ be a fundamental domain with respect to the reflection group $W_S$ such that $W_S$ is generated by orthogonal reflection according the codimension 1 faces $\{P_i\}_{i=1}^k$ of $P$ (some of the vertices of $P$ might be in the boundary at infinity of $\mathbb H^3$). The \emph{support} of $P_i$ is defined to be the hyperplane in $\mathbb H^3$ that contains $P_i$. We define another graph $\Gamma$, whose vertices are in 1-1 correspondence with $\{P_i\}_{i=1}^k$, and two vertices are adjacent if the associated codimension 1 faces intersect in a codimension 2 face. Then $\Gamma$ is planar. 
	By \cite{andreev1970intersection}, two vertices of $\Gamma$ are adjacent if and only the support of the associated codimension 1 faces have non-empty intersection. It follows that nodes of $\Lambda'$ span a complete subgraph of $\Gamma$, as $\Lambda'$ is free-of-infinity. Thus the claim follows from the planarity of $\Gamma$.
\end{proof}

\section{Artin groups with complete bipartite Coxeter diagrams}
The goal of this section is to prove $K(\pi,1)$-conjecture for Artin groups with complete bipartite Coxeter diagrams. We recall an additional tool from \cite{huang2023labeled} in Section~\ref{subsec:folded} and prove the main theorem in Section~\ref{subsec:main}.

\subsection{Folded Artin complexes}
\label{subsec:folded}
We recall a variation of Artin complexes and relative Artin complexes introduced in \cite{huang2023labeled}.
Let $\Lambda,\Lambda'$ be a Coxeter diagram. A \emph{special folding} is a surjective graph morphism $f:\Lambda\to \Lambda'$ (i.e. it map nodes to nodes and edges to edges) such that if nodes $x$ and $y$ are adjacent in $\Lambda'$, then each node in $f^{-1}(x)$ is adjacent in $\Lambda$ to every node in $f^{-1}(y)$. 

An induced subgraph $\Lambda_1$ is $f$-folded if $f$ restricted to $\Lambda_1$ is not injective.

Given a special folding $f:\Lambda\to \Lambda'$, we also use $f:S\to S'$ to denote the induced map on the set of nodes.
We define the folded Artin complex $\Delta_{\Lambda,f}$ as follows. Vertices of $\Delta_{\Lambda,f}$ are in 1-1 correspondence with left cosets of $A_{S\setminus \{f^{-1}(s')\}}$ in $A_S$. A collection of vertices span a simplex if the intersection of the associated collection of left cosets is non-empty. By \cite[Lemma 4.7 and Proposition 4.5]{godelle2012k}, $\Delta_{\Lambda,f}$ is a flag complex. 

One can also define relative version of folded Artin complex, i.e. for any induced subgraph $\Lambda''\subset \Lambda'$, the associated relative folded Artin complex $\Delta_{\Lambda,f,\Lambda''}$ is defined to be a simplicial complex whose vertices are in 1-1 correspondence with left cosets of $A_{S\setminus \{f^{-1}(s'')\}}$ in $A_S$ with $s''\in \Lambda''$, and simplices in $\Delta_{\Lambda,f,\Lambda''}$ corresponds to non-empty intersection of associated left cosets. We label vertices of $\Delta_{\Lambda,f,\Lambda''}$ by vertices in $\Lambda''$, i.e. the vertex corresponding to $gA_{S\setminus \{f^{-1}(s'')\}}$ with $s''\in \Lambda''$ is defined to be of type $\hat s''$. In particular, $\Delta_{\Lambda,f,\Lambda''}$ is viewed as a simplicial complex of type $V\Lambda''$, where $V\Lambda''$ denotes the vertex set of $\Lambda''$. If $\Lambda''$ is a tree, then it makes sense to talk about labeled four wheel condition on $\Delta_{\Lambda,f,\Lambda''}$.

Let $S''=f^{-1}(S')$. Then there is a natural piecewise linear embedding $$i:\Delta_{\Lambda,f,\Lambda''}\to \Delta_{S,S''}$$ as follows. Given a vertex $x\in \Delta_{\Lambda,f,\Lambda''}$ corresponds to a coset of form $gA_{S\setminus \{f^{-1}(s'')\}}$ with $s''\in \Lambda''$, $i$ sends $x$ to the barycenter of the simplex of $\Delta_{S,S''}$ spanned by vertices of form $gA_{S\setminus \{s\}}$ with $s\in f^{-1}(s'')$. Note that $i$ sends vertices inside a simplex to points inside a simplex. Thus we can extend $i$ linearly. 

\begin{lem}\cite[Lemma 10.3]{huang2023labeled}
	\label{lem:order fold}
	Let $f:\Lambda\to\Lambda'$ be a special folding with $\Lambda'$ connected. Suppose $\Lambda''\subset \Lambda'$ is an induced linear subgraph with consecutive nodes $\{s'_i\}_{i=1}^n$. For simplicity we will say a vertex of $Y=\Delta_{\Lambda,f,\Lambda''}$ has type $i$ if it has type $\hat s'_i$. We define a relation in the vertex set $V$ of $Y$ as follows. For two vertices $x,y\in Y$, we put $x<y$ if they are adjacent in $Y$ and $\type(x)<\type(y)$. 
	
	Suppose $\Lambda''$ is an admissible subgraph of $\Lambda'$. Then $(V,\le)$ is an order.
\end{lem}

\begin{lem}\cite[Lemma 10.5]{huang2023labeled}
	\label{lem:tree fold}
	Suppose $f:\Lambda\to \Lambda'$ be a special folding between two trees. If $\Delta_{\Lambda}$ satisfies the labeled 4-wheel condition, then $\Delta_{\Lambda,f}$ satisfies the labeled 4-wheel condition.
\end{lem}

\subsection{Proof of contractibility}
\label{subsec:main}
\begin{thm}
	\label{thm:downflagD4}
	Suppose	$\Lambda$ is the Coxeter diagram of type $D_4$, and take two leaf nodes $a$ and $b$ from $\Lambda$.
	The $(a,b)$-subdivision of $\Delta_\Lambda$ with its vertex set endowed with the order introduced in Definition~\ref{def:subdivision}, is a bowtie free and downward flag poset.
\end{thm}

\begin{proof}
	First we show the bowtie free part by verifying the assumptions of Lemma~\ref{lem:bowtie free criterion}. Let $\Delta'$ be the $(a,b)$-subdivision of $\Delta_\Lambda$. 
	If $v\in \Delta'$ is a vertex with $t(x)=4$ (the quantity $t(x)$ is defined in Definition~\ref{def:subdivision}), then $\lk(v,\Delta')$ is isomorphic in an order-preserving way to the $(a,b)$-subdivision $\Delta'_{a,d,b}$ of $\Delta_{a,d,b}$. It follows from Theorem~\ref{thm:4 wheel} that $\Delta'_{a,d,b}$ satisfies the assumptions of Lemma~\ref{lem:bowtie free criterion}. Thus $\lk(v,\Delta')$ is bowtie free. 
	If $v\in \Delta'$ is a vertex with $t(x)=1$, then $v$ is also a vertex of $\Delta_\Lambda$, and it is of type $\hat a$ or $\hat b$. Moreover, $\lk(v,\Delta')\cong \lk(v,\Delta_\Lambda)$, and this is an isomorphism of posets, where $\lk(v,\Delta')$ is endowed with the induced order from $\Delta'$ and $\lk(v,\Delta_\Lambda)\cong \Delta_{b,d,c}$ or $\Delta_{a,d,c}$ is endowed with the order induced from $b<d<c$ or $a<d<c$ as in Lemma~\ref{lem:poset structure}. By Theorem~\ref{thm:bowtie free}, $\lk(v,\Delta_\Lambda)$ is bowtie free, thus $\lk(v,\Delta')$ is bowtie free. 
	
	Take an embedded 4-cycle $x_1y_1x_2y_2$ in $\Delta'$ with $t(x_i)=1$ for $i=1,2$ and $t(y_i)=4$ for $i=1,2$. Then both $y_1$ and $y_2$ are of type $\hat c$, and each $x_i$ is of type $\hat a$ or $\hat b$. If $x_1$ and $x_2$ have different types, then it follows from \cite[Lemma 4.4]{huang2024Dn} that $x_1$ and $x_2$ are adjacent in $\Delta_\Lambda$. Then there is a vertex of type $m$ in $\Delta'$ which is adjacent to each of $\{x_1,x_2,y_1,y_2\}$ in $\Delta'$, as desired. If $x_1$ and $x_2$ have same type, then Theorem~\ref{thm:4 wheel} implies there is a vertex of type $\hat d$ which is adjacent to each of $\{x_1,x_2,y_1,y_2\}$ in $\Delta_\Lambda$, hence in $\Delta'$, as desired. Thus Lemma~\ref{lem:bowtie free criterion} (3) is verified and the bowtie free part of the lemma follows.
	
The downward flagness follows from \cite[Corollary 7.7]{huang2024Dn}.
\end{proof}

\begin{prop}
	\label{prop:tripod}
Suppose $\Lambda$ is a tripod Coxeter diagram, with a central node $a$ and three leave nodes $b_1,b_2,b_3$. Then for any $i\neq j$, then $(b_i,b_j)$-subdivision of $\Delta_\Lambda$ is bowtie free and downward flag.
\end{prop}

\begin{proof}
The case all edges of $\Lambda$ are labeled 3 follows from Theorem~\ref{thm:downflagD4}.

Now we assume at least one edge of $\Lambda$, say $\overline{ab_1}$, is labeled by a number $\ge 4$. By Proposition~\ref{prop:different subdivision}, it suffices to verify the the two assumptions of Proposition~\ref{prop:ori link}. Assumption 1 follows from Lemma~\ref{lem:2dimensional} below. Now we verify Assumption 2. If both $\overline{ab_2}$ and $\overline{ab_3}$ are labeled by $3$, then we are done by Theorem~\ref{thm:weaklyflagA}. If at least one of these two edges, say $\overline{ab_2}$, has label $\ge 4$, then Lemma~\ref{lem:2dimensional} below implies that $\Delta_{\Lambda_1}$ with induced order from $b_2<a<b_3$, where $\Lambda_1=\Lambda\setminus\{b_1\}$, is bowtie free and downward flag. In order to proof $\Lambda_1$ is weakly flag, it suffices to prove:
\begin{enumerate}
	\item if $\{x_i\}_{i=1}^3$ are type $\hat b_2$ elements such that $x_i$ and $x_{i+1}$ has a common upper bound $y_i$ of type $\hat a$ for $i\in \mathbb Z/3\mathbb Z$, then $\{x_i\}_{i=1}^3$ has a  common upper bound;
	\item if $\{x_i\}_{i=1}^3$ are type $\hat b_3$ elements such that $x_i$ and $x_{i+1}$ has a common lower bound $y_i$ of type $\hat a$ for $i\in \mathbb Z/3\mathbb Z$, then $\{x_i\}_{i=1}^3$ has a  common lower bound.
\end{enumerate}
For (1), we assume $\{y_1,y_2,y_3\}$ are pairwise distinct, otherwise it is trivial. By the downward flagness, we know there is a common lower bound $x$ for $\{y_1,y_2,y_3\}$. Then $x$ is of type $\hat a$. By the bowtie free property, we must have $x=x_1=x_2=x_3$. (2) follows directly from downward flagness.
\end{proof}

\begin{lem}
	\label{lem:2dimensional}
Let $\Lambda$ be a linear Coxeter diagram with its nodes $\{a,b,c\}$. Let $\Delta=\Delta_\Lambda$ be the associated Artin complex, with the order on its vertex set induced from $a<b<c$. Then 
\begin{enumerate}
	\item if $m_{ab}\ge 4$ and $m_{bc}= 3$, then $\Delta$ is bowtie free and downward flag;
	\item if $m_{ab}\ge 4$ and $m_{bc}\ge 4$, then $\Delta$ is bowtie free and flag.
\end{enumerate}

\end{lem}

\begin{proof}
For Assertion (1), the case of $m_{ab}=4$ and $m_{ab}=5$ follows from Theorem~\ref{thm:bowtie free}, Theorem~\ref{thm:triple} and Theorem~\ref{thm:tripleH}. We assume $m_{ab}\ge 6$. The bowtie free condition follows from \cite[Corollary 9.13]{huang2023labeled} and \cite[Lemma 6.13]{huang2023labeled}. For the downward flagness, by the same argument as in the proof of Theorem~\ref{thm:tripleH}, it suffices to show if we have $\{x_i\}_{i=1}^3$ of type $\hat c$ which are pairwisely lower bounded, then they have a common lower bound. Also we can assume for $i\in \mathbb Z/2\mathbb Z$, a lower bound of $x_i$ and $x_{i+1}$ is $y_i$ and $y_i$ is of type $\hat a$.
We metrize triangles in $\Delta_\Lambda$ as flat triangles with angle $\pi/6$ at vertices of type $\hat c$, angle $\pi/2$ at vertices of type $\hat b$ and angle $\pi/3$ at vertices of type $\hat a$. By \cite[Lemma 6]{appel1983artin}, $\Delta_\Lambda$ is locally CAT$(0)$, hence CAT$(0)$. 

Let $\omega$ be the loop $x_1y_1x_2y_2x_3y_3$ in $\Delta$. We can assume without loss of generality that $\omega$ is an embedded $6$-cycle.
Let $\bD\to \Delta$ be a minimal area singular disk diagram (\cite[Section 3.6]{huang2023labeled}) for $\omega$. We endow $\bD$ be the induced metric from $\Delta$, and by slightly abusing the notation, we use $x_i$ (resp. $y_i$) to denote the point in the boundary cycle of $\bD$ mapping to $x_i$ (resp. $y_i$). For $v\in \bD^{(0)}$, let $\kappa(v)$ be the quantity defined in \cite[Section 3.6]{huang2023labeled}. Then $\kappa(y_i)\le \pi/3$ for $1\le i\le 3$, and $\kappa(v)\le 0$ for any interior vertex $x\in \bD$. Thus \cite[Equation 3.17]{huang2023labeled} implies that $\sum_{i=1}^3\kappa(x_i)\ge \pi$. 

Note that if $\kappa(x_i)$ is positive, then the only possible value for $\kappa(x_i)$ is $\pi/3$ and $2\pi/3$. If $\kappa(x_i)=2\pi/3$, then there is a vertex $x'_i$ of type $\hat b$ such that $x_i$ is adjacent to both $y_i$ and $y_{i+1}$, and $\overline{y_ix'_i}$ and $x_iy_{i+1}$ fit together to form a geodesic in $\Delta$. So if $\kappa(x_i)=2\pi/3$ for all $i$, then $x'_1y_1x'_2y_2x'_3y_3$ form a geodesic triangle in $\Delta$ with apexes at $\{y_i\}_{i=1}^3$, which is impossible. If $\kappa(x_i)=2\pi/3$ for exactly two of $\{x_i\}_{i=1}^3$, say $x_1$ and $x_2$, then $x'_1y_1x'_2y_2x_3y_3$ form a geodesic 4-gon in $\Delta$, with apexes at $y_1,y_2,x_3$ and $y_3$. Now we consider the subdiagram $\mathbb D'$ of $\mathbb D$ for this 4-gon. Then for $\mathbb D'$, $\kappa(x_3)\le \pi/3$, $\kappa(y_i)\le 2\pi/3$ for $i=2,3$, $\kappa(x'_i)=0$ for $i=1,2$ (as the link of type $\hat b$ vertices in $\Delta$ is complete bipartite), and $\kappa(y_1)\le \pi/3$. Thus by \cite[Equation 3.17]{huang2023labeled}, we must $\kappa(y_i)=2\pi/3$ for $i=2,3$ and $\kappa(y_1)=\pi/3$. This implies that $x_3$ is adjacent in $\Delta$ to $y_1$. Thus $y_1$ is a common lower bound for $\{x_1,x_2,x_3\}$. If $\kappa(x_i)=2\pi/3$ for exactly one of $\{x_i\}_{i=1}^3$, say $x_1$, then $\sum_{i=1}^3\kappa(x_i)\le 4\pi/3$, which implies that $\sum_{i=1}^3\kappa(y_i)\ge 2\pi/3$. Thus $\kappa(y_i)=\pi/3$ for at least two of $\{y_i\}_{i=1}^3$. Then there is some $y_i$ adjacent to $x_1$, say $y_1$, satisfying $\kappa(y_1)=\pi/3$. Thus $x_2$ is adjacent in $\Delta$ to $y_3$, and we finish as before. It remains to consider that $\kappa(x_i)\le \pi/3$ for each $i$. Then \cite[Equation 3.17]{huang2023labeled} implies that $\kappa(x_i)=\kappa(y_i)=\pi/3$ for each $i$ and the disk diagram $\bD$ is flat. It follows from the combinatorial structure of flat diagram that there is a vertex $z\in \Delta$ of type $\hat a$ such that $z$ is adjacent to $x_i$ for $1\le i\le 3$, as desired.

For Assertion (2), by \cite[Lemma 3]{appel1983artin}, $\lk(x,\Delta)$ has girth $\ge 8$ if $x$ is of type $\hat a$ or $\hat c$. Now Assertion (2) follows from Lemma~\ref{lem:inherit and girth}, Theorem~\ref{thm:contractibleII} and Lemma~\ref{lem:big lattice}.
\end{proof}

\begin{lem}
	\label{lem:star0}
	Suppose $\Lambda$ is a Coxeter diagram which is a star (i.e. it is a union of edges emanating from the same node) with $\ge 4$ nodes. Suppose each tripod subgraph of $\Lambda$ satisfies the conclusion of Proposition~\ref{prop:tripod}. Suppose the central node of $\Lambda$ is $a$. Denote the leaf nodes of $\Lambda$ by $\{b_i\}_{i=1}^n$ for $n\ge 3$. For $1\le i,j,k\le n$, let $\Lambda_{ijk}$ be the subgraph of $\Lambda$ spanned by $a,b_i,b_j,b_k$. We claim for any pairwise distinct $\{i,j,k\}$, the $(i,j)$-subdivision, $(i,k)$-subdivision and $(j,k)$-subdivision of $\Delta_{\Lambda,\Lambda_{ijk}}$ are all bowtie free and downward flag.
\end{lem}

\begin{proof}
 We define subgraphs $\Lambda_{ij}$ and $\Lambda_{ijk\ell}$ of $\Lambda$ in a similar way as in the statement of the lemma.
We prove by induction on $n$. The base case is $n=3$, which follows directly from the assumption. Take $\ell\notin \{i,j,k\}$ - this is possible as $n\ge 4$. We apply Proposition~\ref{prop:ori link2} with $\Lambda'=\Lambda_{ijk\ell}$ and $\{b_i,b_j,a,b_k,b_\ell\}$ playing the roles of $\{a_1,a_2,b,c_1,c_2\}$ in Proposition~\ref{prop:ori link2}. Note that induction assumption implies that Assumptions (1) and (2) of Proposition~\ref{prop:ori link2} are met. Thus $\Delta_{\Lambda,\Lambda_{ijk\ell}}$ is contractible, and Proposition~\ref{cor:propagation2} implies that the $(i,j)$-subdivision of $\Delta_{\Lambda,\Lambda_{ijk}}$ is bowtie free and downward flag. Proposition~\ref{cor:propagation2} (3) also implies that $\Delta_{\Lambda,\Lambda_{ik}}$ and $\Delta_{\Lambda,\Lambda_{jk}}$ are bowtie free. By replacing the role of $(i,j)$ by $(i,k)$ or $(j,k)$, 
We can treat the $(i,k)$-subdivision and $(j,k)$-subdivision in a similar way.
\end{proof}

\begin{lem}
\label{lem:star}
Suppose $\Lambda$ is a Coxeter diagram which is a star (i.e. it is a union of edges emanating from the same node) with $\ge 5$ nodes. Suppose each tripod subgraph of $\Lambda$ satisfies the conclusion of Proposition~\ref{prop:tripod}. Then
\begin{enumerate}
	\item $\Delta_\Lambda$ is contractible;
	\item for each $\Lambda'\subset\Lambda$ which is a union of two edges, $\Delta_{\Lambda,\Lambda'}$ is bowtie free.
\end{enumerate}
\end{lem}

\begin{proof}
Assertion (2) follows from Lemma~\ref{lem:star0}.
For Assertion (1), let $\mathcal C_2$ be the class of Coxeter diagrams in Lemma~\ref{lem:star}. Let $\mathcal C_1=\{\widetilde D_4\}$. It suffices to verify the first two assumptions of \cite[Proposition 7.2]{huang2023labeled}. Assumption 1 of \cite[Proposition 7.2]{huang2023labeled} follows from Lemma~\ref{lem:star0} and Proposition~\ref{prop:ori link2}. Assumption 2 is clear.
\end{proof}

We need two auxiliary results (Lemmas~\ref{lem:4-cycle} and \ref{lem:connect}) before moving forward. 

\begin{lem}(\cite[Lemma 4.8]{huang2023labeled})
	\label{lem:4-cycle}
	Let $X$ be as in Theorem~\ref{thm:contractible}, with all the three assumptions there satisfied. Then for any induced 4-cycle in the 1-skeleton of $X$, there is a vertex $x\in X$ such that $x$ is adjacent to each vertex of this 4-cycle.
\end{lem}

The following is a small variation of \cite[Lemma 6.13]{huang2023labeled}.
\begin{lem}
	\label{lem:connect}
	Suppose $\Lambda$ is a linear Coxeter diagram with its consecutive nodes being	 $S=\{s_i\}_{i=1}^n$.
	Let $X$ be a simplicial complex of type $S$. We endow $S$ with a linear order $s_1<s_2<\cdots<s_n$. Assume the vertex set of $X$ equipped with the relation in Definition~\ref{def:order} is a poset. Then $X$ satisfies bowtie free condition if and only if it satisfies the labeled 4-wheel condition.
\end{lem}

\begin{proof}
	Suppose $X$ satisfies bowtie free condition. Take an induced 4-cycle $x_1y_1x_2y_2$ in $X$. We can not have $x_1<y_1<x_2$ or $x_1>y_1>x_2$, otherwise $x_1$ is adjacent to $x_2$ by our assumption, contradicting that we have an induced 4-cycle. Thus $\{x_1,y_1,x_2,y_2\}$ forms a bowtie and the labeled 4-wheel condition follows immediately. Now suppose $X$ satisfies the labeled 4-wheel condition. If we have $\{x_1,y_1,x_2,y_2\}$ satisfying $x_i< y_j$ for $1\le i,j\le 2$, then these 4 vertices form a 4-cycle in $X$. If $x_1$ and $x_2$ are comparable, or $y_1$ and $y_2$ are comparable, then the bowtie free condition clearly holds for $\{x_1,y_1,x_2,y_2\}$. Now we assume $x_1$ and $x_2$ are not comparable, and $y_1$ and $y_2$ are not comparable. Then the 4-cycle is an induced 4-cycle. Suppose $x_1$ has type $\hat s_{x_1}$ for node $s_{x_1}\in \Lambda$. Similarly we define $s_{x_2},s_{y_1}$ and $s_{y_2}$. We assume without loss of generality that the segment $\Lambda'$ from $s_{x_1}$ to $s_{y_2}$ contain all of $\{s_{x_1},s_{x_2},s_{y_1},s_{y_2}\}$. Then the labeled 4-wheel condition implies that there is a vertex $z$ adjacent of each vertex of the 4-cycle such that $z$ has type $\hat s_z$ with node $s_z\in \Lambda'$. Clearly $x_1<z<y_2$. Now we show $x_2<z$. If this is not true, as $x_2$ and $z$ are adjacent, we must have $x_2>z$, then $x_1<x_2$, contradicting to the assumption that $x_1$ and $x_2$ are not comparable. Similarly, $z<y_1$. Thus the bowtie free condition is satisfied. 
\end{proof}

\begin{lem}
	\label{lem:bowtie free}
Let $\Lambda=\Lambda_{m,n}$ be a Coxeter diagram which is a complete bipartite graph that is the join of $m$ nodes and $n$ nodes. Assume $m>1$ and $n>1$. Take a 4-cycle $C$ with its consecutive vertices $\{x_i\}_{i=1}^4$.
Let $f:\Lambda=\Lambda_{m,n}\to C$ be a special folding from $\Lambda_{m,n}$ such that $f^{-1}(x_i)$ is a single node for $i=1,4$, $f^{-1}(x_2)$ has $n-1$ nodes and $f^{-1}(x_3)$ has $m-1$ nodes. We view the folded Artin complex $\Delta_{\Lambda,f}$ as a complex of type $S=\{x_1,x_2,x_3,x_4\}$, whose vertex set is endowed with a relation induced from the cyclic order $x_1<x_2<x_3<x_4<x_1$ as explained before Theorem~\ref{thm:contractible}.
Then the link of each vertex of the folded Artin complex $\Delta_{\Lambda,f}$ is bowtie free. 
\end{lem}

\begin{proof}
Let $C_i$ be the component of $C\setminus \{x_i\}$, and $\Lambda_i=f^{-1}(C_i)$. 
We prove by induction on $m+n$. The base case of the induction is that $m=2$ and $n=2$, where $f$ is an isomorphism, and $\Delta_{\Lambda,f}\cong \Delta_\Lambda$. Then bowtie free property of each link follows from Lemma~\ref{lem:link} and either Lemma~\ref{lem:2dimensional} or Theorem~\ref{thm:bowtie free}.

Now we assume at least one of $m$ and $n$, say $n$, is $>2$. Take $v\in \Delta_{\Lambda,f}$ be a vertex of type $\hat x_3$. Then $\lk(v,\Delta_{\Lambda,f})\cong \Delta_{\Lambda_3,f}$.  Note that $\Lambda_3$ is a star. By Lemma~\ref{lem:star} and \cite[Proposition 6.15]{huang2023labeled}, $\Delta_{\Lambda_3}$ satisfies the labeled four wheel condition. Hence the same holds for $\Delta_{\Lambda_3,f}$ by Lemma~\ref{lem:tree fold}. Thus $\Delta_{\Lambda_3,f}$ is bowtie free by Lemma~\ref{lem:connect} and Lemma~\ref{lem:order fold}. Similarly, we know $\lk(v,\Delta_{\Lambda,f})$ is bowtie free if $v$ is of type $\hat x_2$.

It remains to consider the case $v$ is of type $\hat x_4$ or $\hat x_1$. We will only treat $v$ being of type $\hat x_4$, as the other case is similar. Note that $\lk(v,\Delta_{\Lambda,f})\cong \Delta_{\Lambda_4,f}$. Consider a special folding $f':\Lambda_4\to C$ such that $(f')^{-1}(x_i)=f^{-1}(x_i)$ for $i=1,3$, $(f')^{-1}(x_4)$ is one node and $(f')^{-1}(x_2)$ has $n-2$ nodes. This is possible as $n>2$. By induction, $\Delta_{\Lambda_4,f'}$ satisfies the conclusion of the claim. Then $\Delta_{\Lambda_4,f'}$, viewed as a complex of type $S=\{x_1,x_2,x_3,x_4\}$, satisfies the assumption of Theorem~\ref{thm:contractible}. Hence Lemma~\ref{lem:4-cycle} applies to $\Delta_{\Lambda_4,f'}$. Now we define an embedding 
$$
\iota: \Delta_{\Lambda_4,f}\to \Delta_{\Lambda_4,f'}
$$
as follows. Given a vertex $w\in \Delta_{\Lambda_4,f}$ corresponding to a coset of form $gA_{\Lambda_4\setminus \{f^{-1} (s)\}}$ (with $g\in A_{\Lambda_4}$ and $s\in \{x_1,x_2,x_3\}$), $\iota(w)$ is defined to be the barycenter of the simplex in $\Delta_{\Lambda_4,f'}$ spanned by vertices corresponding to cosets of form $gA_{\Lambda_4\setminus \{(f')^{-1}(s')\}}$ with $s'\in C$ ranging over all vertices such that $(f')^{-1}(s')\subset f^{-1}(s)$. Then we extend $\iota$ linearly. By definition, for $i=1,3$, $\iota$ induces a 1-1 correspondence between vertices of type $\hat x_i$ in $\Delta_{\Lambda_4,f}$ and vertices of the same type in $\Delta_{\Lambda_4,f'}$; and a 1-1 correspondence between vertices of type $\hat x_2$ in  $\Delta_{\Lambda_4,f}$ and the barycenters of edges in $\Delta_{\Lambda_4,f'}$ spanned by a vertex of type $\hat x_2$ and a vertex of type $\hat x_4$. Let $\omega$ be a 4-cycle in $\Delta_{\Lambda_4,f}$. We will show that
\begin{enumerate}
	\item if consecutive vertices of $\omega$ have type $\hat x_1,\hat x_2,\hat x_1,\hat x_2$, or type $\hat x_2,\hat x_3,\hat x_2,\hat x_3$, then the 4-cycle is not embedded;
	\item if consecutive vertices of $\omega$ have type $\hat x_1,\hat x_3,\hat x_1,\hat x_3$, and the 4-cycle is embedded, then there is a vertex of type $\hat x_2$ adjacent to each vertices of $\omega$.
\end{enumerate}
By Lemma~\ref{lem:bowtie free criterion}, once these two properties are established, we know $\Delta_{\Lambda_4,f}$ is bowtie free. If $\omega$ is of type $\hat x_1,\hat x_2,\hat x_1,\hat x_2$, then $\iota(\omega)$ gives a 4-cycle $\omega'$ in $\Delta_{\Lambda_4,f'}$ by replacing points in $\iota(\omega)$ which are the midpoint of an edge by the vertex of type $\hat x_2$ in that edge. Then $\omega$ is embedded if and only if $\omega'$ is embedded. If $\omega'$ is embedded, it must be induced as two vertices in $\Delta_{\Lambda_4,f'}$ are not adjacent. Thus Lemma~\ref{lem:4-cycle} implies that there is a vertex $w'$ of $\Delta_{\Lambda_f'}$ which is adjacent to each vertex of $\omega'$. Thus $\omega'$ is an embedded 4-cycle in $\lk(w',\Delta_{\Lambda_4,f'})$, which contradicts the induction hypothesis that $\lk(w',\Delta_{\Lambda_4,f'})$ is bowtie free. Thus $\omega$ is not embedded. The case $\omega$ is of type $\hat x_2,\hat x_3,\hat x_2,\hat x_3$ can be treated similarly. Now we assume $\omega$ is of type $\hat x_1,\hat x_3,\hat x_1,\hat x_3$ and $\omega$ is embedded. Then $\iota(\omega)$ is an embedded 4-cycle in $\Delta_{\Lambda_4,f'}$, which must be induced. Lemma~\ref{lem:4-cycle} implies that there is a vertex $w'$ of $\Delta_{\Lambda_4,f'}$ which is adjacent to each vertex of $\iota(\omega)$. Then $w'$ is of type $\hat x_2$ or $\hat x_4$. We assume without loss of generality that $w'$ of type $\hat x_2$. Then $\iota(\omega)$ is an induced 4-cycle in $\lk(w',\Delta_{\Lambda_4,f'})$. As $\lk(w',\Delta_{\Lambda_4,f'})$ is bowtie free, there is a vertex $w''\in \lk(w',\Delta_{\Lambda_4,f'})$ adjacent to each vertex of $\iota(\omega)$. Let $u$ be the vertex in $\Delta_{\Lambda_4,f}$ such that $\iota(u)$ is the barycenter of the edge $\overline{w'w''}$. Then $u$ is adjacent to each vertex in $\omega$, as desired.
\end{proof}

\begin{lem}
	\label{lem:contractible}
Suppose $\Lambda$ is a Coxeter diagram which is a complete bipartite graph, and suppose $\Lambda$ has $\ge 5$ nodes if $\Lambda$ is a star. Suppose each tripod subgraph of $\Lambda$ satisfies the conclusion of Proposition~\ref{prop:tripod}. Then $\Delta_\Lambda$ is contractible.
\end{lem}

\begin{proof}
Suppose $\Lambda=\Lambda_{m,n}$ which is the join of $m$ nodes and $n$ nodes. The case that one of $m$ and $n$ is 1 follows from Lemma~\ref{lem:star}. Now we assume $m>1$ and $n>1$. Let $f:\Lambda\to C$ be as in Lemma~\ref{lem:bowtie free}. We claim whenever $f$ is a composition of two special foldings $f':\Lambda\to \Lambda'$ and $f'':\Lambda'\to C$ with $\Lambda'$ being complete bipartite, then $\Delta_{\Lambda,f'}$ is contractible. The lemma follows from this claim by taking $f''$ to be identity. 

It remains to prove the claim. We induct on $m+n$.
Consider a sequence of special foldings:
$$
\Lambda\stackrel{f_1}{\to} \Lambda_1\stackrel{f_2}{\to}\Lambda_2\stackrel{f_3}{\to}\cdots \stackrel{f_n}{\to}\Lambda_n=C
$$
such that $f=f_1\circ f_2\circ\cdots\circ f_n$, each $\Lambda_i$ is complete bipartite and $\Lambda_{i+1}$ has exactly one node less than $\Lambda_i$. Let $g_i=f_1\circ\cdots\circ f_i$ and $\Delta_i=\Delta_{\Lambda,g_i}$. The sequence can be arranged such that $\Delta_{\Lambda,f'}=\Delta_i$ for some $i$. So it suffices to show $\Delta_i$ is contractible for any $i$. We set $\Delta_0=\Delta_\Lambda$.

For each $i$, let $v_i\in \Lambda_i$ be the unique node such that $f^{-1}_i(v_i)$ has more than one nodes. For each edge in $\Delta_i$ whose vertices are of type $f^{-1}_{i+1}(v_{i+1})$, we add a new vertex in this edge which is the midpoint of this edge, and say this new vertex has type $\hat m_i$. Cut each top dimensional simplex in $\Delta_i$ along the new vertex into two simplices, and let the resulting complex by $\Delta'_i$. Then there is a natural embedding $\iota_{i+1}:\Delta_{i+1}\to \Delta_i$ mapping vertices of type $\hat v$ to vertices of type $\hat f^{-1}_{i+1}(v)$ for $v\in \Lambda_{i+1}\setminus\{v_{i+1}\}$, and mapping vertices of type $\hat v_{i+1}$ to vertices of type $\hat m_i$. The image of $\iota_{i+1}$ is the full subcomplex of $\Delta_i$ spanned by vertices whose types are either $\hat m_i$ or inside $f^{-1}_{i+1}(\Lambda_{i+1}\setminus\{v_{i+1}\})$.

By Lemma~\ref{lem:bowtie free}, Theorem~\ref{thm:contractible} and \cite[Lemma 10.3]{huang2023labeled}, we know $\Delta_n$ is contractible. Next we will show $\Delta_i$ and $\Delta_{i+1}$ are homotopic equivalent for $0\le i\le n-1$. By the description of how $\Delta_{i+1}$ sits as a subcomplex of $\Delta'_i$ via $\iota_{i+1}$ in the previous paragraph, it suffices to show for any $v\in f^{-1}_{i+1}(v_{i+1})$ and any vertex $x\in \Delta_i$ of type $\hat v$, $\lk(x,\Delta_i)$ is contractible, as this would imply $\Delta'_i$ deformation retracts onto $\Delta_{i+1}$. As $\Lambda_i$ is complete bipartite, $f^{-1}_{i+1}(v_{i+1})$ is contained in a join factor of $\Lambda'_i$ of $\Lambda_i$. As $f^{-1}_{i+1}(v_{i+1})$ has two elements, we know $\Lambda_i\setminus\{v\}$ is connected.
Let $\Lambda_{i,v}$ be the unique component of $\Lambda_i\setminus\{v\}$. Then $\Lambda_{i,v}$ is also complete bipartite. Let $\Theta=g^{-1}_i(\Lambda_{i,v})$.
 By \cite[Lemma 10.4]{huang2023labeled}, $$\lk(x,\Delta_i)\cong \Delta_{\Theta,g_i}.$$
By choice of $v$, we know $f_i(\Lambda_i)=f_i(\Lambda_{i,v})$. Thus $f_{i+1}\circ\cdots\circ f_n$ still maps $\Lambda_{i,v}$ onto $C$.
As $\Theta$ is a strictly smaller complete bipartite graph compared to $\Lambda$, by induction assumption, we know $\Delta_{\Theta,g_i}$ is contractible. Hence $\lk(x,\Delta_i)$ is contractible, as desired.
\end{proof}

\begin{thm}
	\label{thm:complete bipartite}
Let $\Lambda$ be a complete bipartite Coxeter diagram. Then $A_\Lambda$ satisfies the $K(\pi,1)$-conjecture.
\end{thm}

\begin{proof}
Note that when $\Lambda$ has $\le 4$ nodes, $A_\Lambda$ satisfies $K(\pi,1)$-conjecture by Theorem~\ref{thm:4generator}. Now the theorem follows by induction on the number of nodes in $\Lambda$, using Proposition~\ref{prop:tripod}, Lemma~\ref{lem:contractible} and Theorem~\ref{thm:kpi1}.
\end{proof}

\begin{cor}
	\label{cor:tree}
Let $\Lambda$ be a tree Coxeter diagram with a collection of open edges $E$ with label $\ge 6$ such that each component of $\Lambda\setminus(\cup_{e\in E}\{e\})$ is either spherical or a star. Then $A_\Lambda$ satisfies the $K(\pi,1)$-conjecture. 
\end{cor}

\begin{proof}
By Theorem~\ref{thm:bipartite} and \cite[Proposition 9.12]{huang2023labeled}, it suffices to show $\Delta_{\Lambda}$ satisfies labeled 4-wheel condition whenever $\Lambda$ is a star. Take a maximal linear subgraph $\Lambda'\subset \Lambda$. Then $\Delta_{\Lambda,\Lambda'}$ is bowtie free - the case when $\Lambda$ has $\le 3$ nodes follows from Theorem~\ref{thm:bowtie free} and Lemma~\ref{lem:2dimensional}, the case when $\Lambda$ has 4 nodes follows from Proposition~\ref{prop:tripod} and the case when $\Lambda$ has $\ge 5$ nodes follows from Lemma~\ref{lem:star}. Then Lemma~\ref{lem:4wheel} implies that $\Delta_\Lambda$ satisfies labeled 4-wheel condition.
\end{proof}

\section{$K(\pi,1)$ for some higher-dimensional families}

\label{sec:high}

\begin{prop}
	\label{prop:reduction}
Let $\Lambda$ be a connected Coxeter diagram with an induced sub-diagram $\Lambda'\subset \Lambda$ such that
\begin{enumerate}
	\item $\Lambda'$ is the Coxeter diagram of a 3-dimensional irreducible affine Coxeter group;
	\item for any node $s\in \Lambda'$, each component of $\Lambda\setminus\{s\}$ is either spherical, or has type in $\{\widetilde A_3,\widetilde B_3,\widetilde C_n\}$.
\end{enumerate}	
Then $A_\Lambda$ satisfies the $K(\pi,1)$-conjecture.
\end{prop}

\begin{proof}
Let $\mathcal C$ be the class of Coxeter diagrams satisfying the assumptions of Proposition~\ref{prop:reduction}. Let $\Lambda_0$ be a 5-cycle with its consecutive edges labeled by $\{3,3,4,3,4\}$.
Let $\mathcal C_T$ be the subclass made of members of $\mathcal C$ that are trees. Let $\mathcal C'$ be the collection of all $\Lambda\in\mathcal C$ such that there exists  $\Lambda'\subset\Lambda$ of type $\widetilde A_3$ with Assumption 2 satisfied. We first show that element of $\mathcal C=\mathcal C_T\sqcup\mathcal C'\sqcup\{\Lambda_0\}$. Indeed, take $\Lambda\in\mathcal C$ and let $\Lambda'\subset \Lambda$ be as in Assumption 2. First we consider the case $\Lambda'$ has type $\widetilde C_3$. Let $\{s_i\}_{i=1}^4$ be consecutive nodes in $\Lambda'$.
By our assumption and the classification of spherical and Euclidean Coxeter diagrams, for $i=1,4$, the only possibility types for each component of $\Lambda\setminus\{s_i\}$ are $\{B_n,F_4,\widetilde B_3,\widetilde C_n\}$. If $\Lambda$ is not a tree, then it must contain an embedded cycle $C$. As each component of $\Lambda\setminus\{s_i\}$ is a tree for $i=1,4$, we know $\{s_1,s_4\}\subset C$. Note that $s_2\in C$, otherwise $C\cup \overline{s_3s_4}$ is contained in a component of $\Lambda\setminus\{s_2\}$, which is not possible by our assumption (as the label of $\overline{s_3s_4}$ is $4$). Similarly $s_3\in C$. Thus $\Lambda'\subset C$. Let $e_i$ be the edge of $C$ that is outside $\Lambda'$ and contains $s_i$ for $i=1,4$.
As $\Lambda'$ is an induced subgraph of $\Lambda$, $s_1\notin e_4$ and $s_4\notin e_1$. Thus $e_1\cup\overline{s_1s_2}\cup\overline{s_2s_3}$ is contained in a component of $\Lambda\setminus\{s_4\}$, and by our assumption, the only possibility of this component is $F_4$. Similarly, the component of $\Lambda\setminus\{s_1\}$ containing $\overline{s_2s_3}\cup\overline{s_3s_4}\cup e_4$ is of type $F_4$. It follows from Assumption 2 of the proposition that $\Lambda=\Lambda_0$.
Now we consider the case that $\Lambda'$ has type $\widetilde B_3$. Let $a$ be the center node and $\{b_i\}_{i=1}^3$ be leave nodes of $\Lambda'$ with $m_{a,b_1}=4$. For $1\le i\le 3$, let $\Lambda_i$ be the component of $\Lambda\setminus\{b_i\}$ containing $a$. Then for $i=2,3$, $\Lambda_i$ contains an edge labeled by 4, hence is a tree by the classification of spherical and Euclidean Coxeter diagrams. If  $\Lambda_1$ is a tree, then $\Lambda$ is a tree. Otherwise $\Lambda_1$ is of type $\widetilde A_3$. Then for $i=2,3$, $\Lambda_i$ contains a path of length 3, with its edges labeled by $4,3,3$. Thus $\Lambda_i$ is of type $B_n$ or $\widetilde C_n$ for $i=2,3$. However, only type $B_n$ and $n=4$ is possible, otherwise $\Lambda_1$ is not of type $\widetilde A_3$. This implies that $\Lambda$ is obtained from a diagram of type $\widetilde A_3$ by adding an extra edge of label $4$. Thus $\Lambda\in \mathcal C'$. 

Next we show $\mathcal C'$ satisfy the assumptions of \cite[Corollary 7.3]{huang2023labeled} with $\mathcal C_2=\mathcal C'$, and $\mathcal C_1$ be the class containing only the diagram of type $\widetilde A_3$, hence each element in $\mathcal C'$ satisfies the $K(\pi,1)$-conjecture. Note that Assumptions $2,3,4$ of \cite[Corollary 7.3]{huang2023labeled} are clear. It remains to show $\Delta_{\Lambda,\Lambda'}$ is contractible. For node $s\in\Lambda'$, Let $\Lambda_s$ be the component of $\Lambda\setminus\{s\}$ containing the rest of the nodes of $\Lambda'$. If $\Lambda_s$ is not a tree for some $s$, then $\Lambda_s$ is of type $\widetilde A_3$. Then $\Lambda$ is a complete bipartite graph $K_{2,3}$. By \cite[Theorem 10.7 and Corollary 10.10]{huang2023labeled}, $\Delta_\Lambda$ is contractible. Hence $\Delta_{\Lambda,\Lambda'}$ is contractible by Lemma~\ref{lem:dr}. Now we assume $\Lambda_s$ is a tree for each $s$. By Theorem~\ref{thm:contractible} and Lemma~\ref{lem:link}, it suffices to show $\Delta_{\Lambda_s,\Lambda_s\cap\Lambda'}$ is bowtie free. However, this follows from Theorem~\ref{thm:bowtie free} if $\Lambda_s$ is spherical, and Corollary~\ref{cor:widetildeC_n} if $\Lambda_s$ is of type $\widetilde C_n$. If $\Lambda_s$ is of type $\widetilde B_3$, then by Theorem~\ref{thm:triple} and Theorem~\ref{thm:weaklyflagA}, $\Delta_{\Lambda_s}$ satisfies the assumptions of Proposition~\ref{prop:ori link}. Thus by Proposition~\ref{cor:propagation} $\Delta_{\Lambda_s,\Lambda_s\cap\Lambda'}$ is bowtie free. 

%If $\Lambda_s$ is of type $\widetilde B_4$, then by Theorem~\ref{thm:downflagD4} and Theorem~\ref{thm:triple}, $\Delta_{\Lambda_s}$ satisfies the assumptions of Proposition~\ref{prop:ori link0}. By Proposition~\ref{cor:propagation} (3) (the ``in addition'' assumption follows from Theorem~\ref{thm:downflagD4}) and Lemma~\ref{lem:4wheel}, $\Delta_{\Lambda_s,\Lambda_s\cap\Lambda'}$ is bowtie free.  If $\Lambda_3$ is of type $\widetilde D_4$, then by Theorem~\ref{thm:downflagD4}, $\Delta_{\Lambda_s}$ satisfies the assumptions of Proposition~\ref{prop:ori link2}. Thus by Proposition~\ref{cor:propagation2} $\Delta_{\Lambda_s,\Lambda_s\cap\Lambda'}$ is bowtie free.

As the $K(\pi,1)$-conjecture for Artin group with diagram $\Lambda_0$ follows from \cite[Theorem 10.9]{huang2023labeled}, it remains to show diagrams in $\mathcal C_T$ satisfies $K(\pi,1)$-conjecture.
We will show $\mathcal C_T$ satisfy the assumptions of \cite[Corollary 7.3]{huang2023labeled} with $\mathcal C_2=\mathcal C_T$ and $\mathcal C_1=\{\widetilde C_3,\widetilde B_3\}$. Again Assumptions $2,3,4$ of \cite[Corollary 7.3]{huang2023labeled} are clear, and it suffices to show $\Delta_{\Lambda,\Lambda'}$ is contractible with $\Lambda'\in \mathcal C_1$ and $\Lambda\in \mathcal C_2$.

\smallskip
\noindent
\underline{Case 1: $\Lambda'$ is of type $\widetilde C_3$.} Let consecutive nodes of $\Lambda'$ be $\{s_i\}_{i=1}^4$. Let $\Lambda_1$ be the component of $\Lambda\setminus\{s_1\}$ containing $\Lambda'\setminus\{s_1\}$. By Assumption 2, $\Lambda_1$ is either spherical or irreducible 3-dimensional Euclidean. We claim the vertex set of $\Delta_{\Lambda_1,\Lambda_1\cap \Lambda'}$, endowed with the order induced from $s_2<s_3<s_4$ is a bowtie free and upward flag poset. As $\Lambda_1$ is a tree, $\Delta_{\Lambda_1,\Lambda_1\cap\Lambda'}$ is indeed a poset. If $\Lambda_1$ is spherical, then either $\Lambda_1$ is of type $B_n$ or type $F_4$. Then $\Delta_{\Lambda_1,\Lambda_1\cap \Lambda'}$ being bowtie free and upward flag follows from Theorem~\ref{thm:bowtie free}, Theorem~\ref{thm:triple} and Proposition~\ref{prop:F4}. If $\Lambda_1$ is of type $\widetilde C_n$, then the claim follows from Corollary~\ref{cor:widetildeC_n}. If $\Lambda_1$ is of type $\widetilde B_3$, then as before we know $\Delta_{\Lambda_1}$ satisfies the assumptions of Proposition~\ref{prop:ori link}, and Proposition~\ref{cor:propagation} (1) implies that $\Delta_{\Lambda_1,\Lambda_1\cap \Lambda'}$ is bowtie free and upward flag. Thus the claim is proved. Similarly,  if $\Lambda_4$ is the component of $\Lambda\setminus\{s_4\}$ containing $\Lambda'\setminus\{s_4\}$, then the vertex set of $\Delta_{\Lambda_4,\Lambda_4\cap \Lambda'}$, endowed with the order induced from $s_1<s_2<s_3$ is a bowtie free and downward flag poset. As $\Lambda$ is a tree, $\Delta_{\Lambda,\Lambda'}$ with its vertices endowed with the order induced from $s_1<s_2<s_3<s_4$ is a poset. Now Theorem~\ref{thm:contractibleII} implies that $\Delta_{\Lambda,\Lambda'}$ is contractible.

\smallskip
\noindent
\underline{Case 2: $\Lambda'$ is of type $\widetilde B_3$.} Let $a$ be the center node and $\{b_i\}_{i=1}^3$ be leave nodes of $\Lambda'$ with $m_{a,b_1}=4$. For $1\le i\le 3$, let $\Lambda_i$ be the component of $\Lambda\setminus\{b_i\}$ containing $a$. Now we verify the assumptions of Proposition~\ref{prop:ori link} hold for $\Delta_{\Lambda,\Lambda'}$. As $\Lambda_i$ contains an edge labeled $4$ for $i=2,3$, we know the type of $\Lambda_i$ belongs to $\{F_4,B_n,\widetilde C_n,\widetilde B_3\}$. Then Assumption 1 of Proposition~\ref{prop:ori link} follows from Proposition~\ref{prop:F4} if $\Lambda_i$ is of type $F_4$, Proposition~\ref{thm:triple} if $\Lambda_i$ is of type $B_n$, Corollary~\ref{cor:widetildeC_n} if $\Lambda_i$ is of type $\widetilde C_n$, and Proposition~\ref{cor:propagation} (1) if $\Lambda_i$ is of type $\widetilde B_3$. For Assumption 2 of Proposition~\ref{prop:ori link}, note that the type of $\Lambda_1$ belongs to $\{A_n,B_n,D_4,\widetilde C_n,\widetilde B_3\}$ (otherwise we will have a contradiction with the fact that the type of $\Lambda_i$ belongs to $\{F_4,B_n,\widetilde C_n,\widetilde B_3\}$ for $i=2,3$). Thus Assumption 2 of Proposition~\ref{prop:ori link} follows from Theorem~\ref{thm:weaklyflagA} if $\Lambda_1$ is of type $A_n$,  Proposition~\ref{cor:propagation} (2) if $\Lambda_1$ is of type $\widetilde B_3$, Theorem~\ref{thm:weakflagD} if $\Lambda_1$ is of type $D_4$, and Lemma~\ref{lem:weakflag} below if $\Lambda_1$ is of type $B_n$ or $\widetilde C_n$.
\end{proof}

%Proposition~\ref{cor:propagation} (2) combined with Proposition~\ref{prop:different subdivision4} and Theorem~\ref{thm:downflagD4} if $\Lambda_1$ is of type $\widetilde B_4$,

\begin{lem}
	\label{lem:weakflag}	
Suppose $\Lambda$ is a Coxeter diagram of type $B_n,\widetilde C_n,H_3$ or $F_4$. Let $\Lambda'$ be a linear subgraph with three nodes. Then $\Delta_{\Lambda,\Lambda'}$ is weakly flag.
\end{lem}
\begin{proof}
First assume $\Lambda$ is of type $B_n$. We label consecutive nodes of $\Lambda$ by $\{s_i\}_{i=1}^n$ with $m_{s_{n-1},s_n}=4$. Then vertices of $\Delta_{\Lambda}$ with the order induced from $s_1<s_2<\cdots<s_n$ is upward flag. Suppose $\Lambda'$ has nodes $\{s_i,s_{i+1},s_{i+2}\}$. Then we need to show:
\begin{enumerate}
	\item if $\{x_i\}_{i=1}^3$ are type $\hat s_i$ elements such that $x_i$ and $x_{i+1}$ have a common upper bound $y_i$ of type $\hat s_{i+1}$ for $i\in \mathbb Z/3\mathbb Z$, then $\{x_i\}_{i=1}^3$ have a  common upper bound of type $\hat s_{i+2}$;
	\item if $\{x_i\}_{i=1}^3$ are type $\hat s_{i+2}$ elements such that $x_i$ and $x_{i+1}$ have a common lower bound $y_i$ of type $\hat s_{i+1}$ for $i\in \mathbb Z/3\mathbb Z$, then $\{x_i\}_{i=1}^3$ have a  common lower bound of type $\hat s_i$.
\end{enumerate}
We will assume without loss of generality that $\{x_i\}_{i=1}^3$ and $\{y_i\}_{i=1}^3$ are pairwise distinct.
For (1), note that the upward flagness implies that $\{x_i\}_{i=1}^3$ has a common upper bound, say $z$, of type $s_j$. By the pairwise distinct assumption, $j\ge i+2$. If $j=i+2$, then we are done. If $j>i+2$, let $\mathcal P$ be the collection of vertices in $\Delta_\Lambda$ which is $<z$. As $y_i$ is the join of $x_i$ and $x_{i+1}$, we know $y_i\in \mathcal P$ for all $i$. Let $\Lambda_j$ be the subgraph of $\Lambda$ spanned by all $s_i$ with $i<j$. Then $\mathcal P$ can be identified with the vertex set of $\Delta_{\Lambda_j}$, endowed with the induced order from $s_1<\cdots<s_j$. As $\Lambda_j$ is of type $A_{j-1}$, by Theorem~\ref{thm:weaklyflagA}, $\{x_i\}_{i=1}^3$ has a common upper bound of type $\hat s_{i+2}$. For (2), by upward flagness, $\{y_i\}_{i=1}^3$ has a common upper bound $z$. Let $\mathcal P$ be the collection of vertices in $\Delta_\Lambda$ which is $<z$. Again we have $\{x_i\}_{i=1}^3\subset \mathcal P$, and we are done by Theorem~\ref{thm:weaklyflagA}. The case $\Lambda$ is of type $\widetilde C_n$ is similar, using Corollary~\ref{cor:widetildeC_n}, and the $B_n$ version of Lemma~\ref{lem:weakflag}. The case of $F_4$ or $H_3$ is similar as well, using Proposition~\ref{prop:F4} and Theorem~\ref{thm:tripleH}.
\end{proof}

\begin{cor}
The downward flag part of Conjecture~\ref{conj:compareD} holds true when $n=3$, and $A_S$ is of type $A_n,B_n,H_3,F_4$. 
\end{cor}

\begin{proof}
This is a combination of Lemma~\ref{lem:weakly flag equivalent}, Lemma~\ref{lem:weakflag}, Theorem~\ref{thm:weaklyflagA}, Theorem~\ref{thm:tripleH} and Proposition~\ref{prop:F4}. 
\end{proof}

\begin{cor}
	\label{cor:seven families}
Suppose $\Lambda$ is a Coxeter diagram belonging to one of the following families in Figure~\ref{fig:families}. Then $A_\Lambda$ satisfies the $K(\pi,1)$-conjecture.
\end{cor}
\begin{figure}[h]
	\label{fig:families}
	\centering
	\includegraphics[scale=0.83]{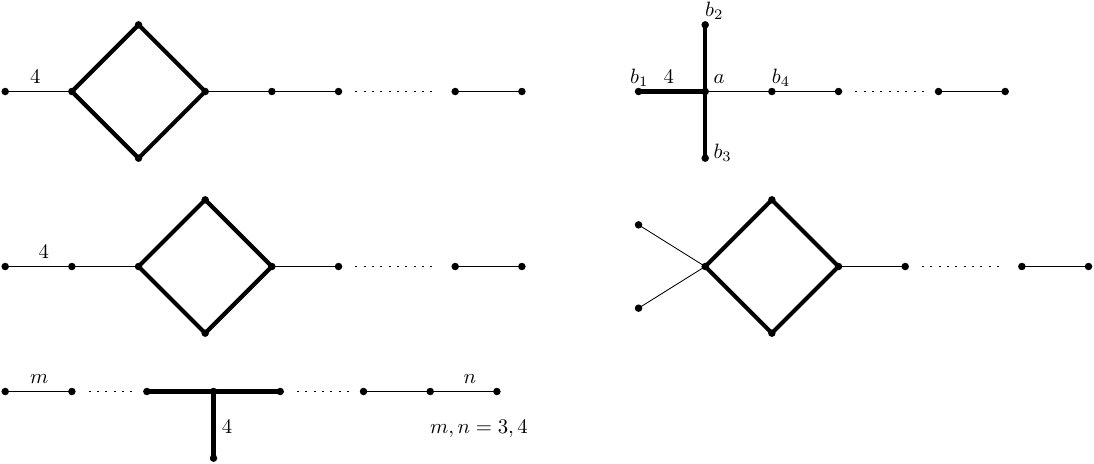}
	\caption{Several families.}
\end{figure}
\begin{proof}
Let $\mathcal C_2,\mathcal C_3,\mathcal C_4$ be the families in first row right, second row left, and second two right respectively in Figure~\ref{fig:families}. 	All families in Figure~\ref{fig:families} except $\mathcal C_2,\mathcal C_3$ and $\mathcal C_4$ follow from Proposition~\ref{prop:reduction} (the choice of $\Lambda'$ are indicated in the thickened subgraphs). Let $\mathcal C_2$ be the remaining family of Coxeter diagrams. We will verify the assumptions of \cite[Corollary 7.3]{huang2023labeled} with $\mathcal C_1=\{\widetilde B_3\}$. Assumptions $2,3$ of \cite[Corollary 7.3]{huang2023labeled} are clear. 
Take $\Lambda\in\mathcal C_2$ and let $\Lambda'$ be the thickened subgraph of $\Lambda$. Then $\Lambda\setminus\{s\}$ is spherical for $s=b_1$ or $a$, and $\Lambda\setminus\{s\}$ belongs to the family on the left side of second row for $s=b_2$ or $b_3$. Thus Assumption $4$ of \cite[Corollary 7.3]{huang2023labeled} follows. It remains to show $\Delta_{\Lambda,\Lambda'}$ is contractible. For $i=1,2$, let $\Lambda_i=\Lambda\setminus\{b_i\}$. Let $\Lambda'_i$ be the subgraph spanned by $\Lambda_i\cap \Lambda'$ and $\{b_4\}$. Then by the same argument as in the proof of Proposition~\ref{prop:reduction}, we know $\Delta_{\Lambda_i,\Lambda'_i}$ satisfies the assumptions of Proposition~\ref{prop:ori link}. Thus Proposition~\ref{cor:propagation} (1) implies that $\Delta_{\Lambda_i,\Lambda_i\cap\Lambda'}$ is a bowtie free, upward flag poset with its vertex set endowed with the order induced from $b_i<a<b_1$. Thus Assumption 1 of Proposition~\ref{prop:ori link} holds for $\Delta_{\Lambda,\Lambda'}$. Assumption 2 also holds, by Theorem~\ref{thm:weakflagD}. Thus $\Delta_{\Lambda,\Lambda'}$ is contractible by Proposition~\ref{prop:ori link}.

Treating families $\mathcal C_3$ and $\mathcal C_4$ reduces to showing $\Delta_{\Lambda,\Lambda'}$ is contractible, where $\Lambda$ is a diagram in $\mathcal C_3$ or $\mathcal C_4$, and $\Lambda'$ is the thickened $\widetilde A_3$ subdiagram in Figure~\ref{fig:families}. For a node $s\in \Lambda'$, let $\Lambda_s$ be the component of $\Lambda\setminus\{s\}$ containing $\Lambda'\setminus\{s\}$. By Theorem~\ref{thm:contractibleII} and Lemma~\ref{lem:sc}, it suffices to show $\Delta_{\Lambda_s,\lambda_s\cap\Lambda'}$ is bowtie free for each $s\in \Lambda'$. Note that $\Lambda_s$ is of type $B_n,D_n,\widetilde B_4$ or $\widetilde D_4$. The $B_n$ and $D_n$ case follows from Theorem~\ref{thm:bowtie free}.
If $\Lambda_s$ is of type $\widetilde B_4$, then by Theorem~\ref{thm:downflagD4} and Theorem~\ref{thm:triple}, $\Delta_{\Lambda_s}$ satisfies the assumptions of Proposition~\ref{prop:ori link0}. By Proposition~\ref{cor:propagation} (3) (the ``in addition'' assumption follows from Theorem~\ref{thm:downflagD4}) and Lemma~\ref{lem:4wheel}, $\Delta_{\Lambda_s,\Lambda_s\cap\Lambda'}$ is bowtie free.  If $\Lambda_3$ is of type $\widetilde D_4$, then by Theorem~\ref{thm:downflagD4}, $\Delta_{\Lambda_s}$ satisfies the assumptions of Proposition~\ref{prop:ori link2}. Thus by Proposition~\ref{cor:propagation2} $\Delta_{\Lambda_s,\Lambda_s\cap\Lambda'}$ is bowtie free.
\end{proof}

\begin{cor}
	\label{cor:quasilanner}
Let $W_\Lambda$ be a reflection group acting on $\mathbb H^n$ with $n\le 4$ such that the fundamental domain is a finite volume non-compact simplex, and $\Lambda$ is the Coxeter diagram. Then the $K(\pi,1)$-conjecture holds for $A_\Lambda$.
\end{cor}
	
\begin{figure}[h]
	\label{fig:quasilanner}
	\centering
	\includegraphics[scale=0.83]{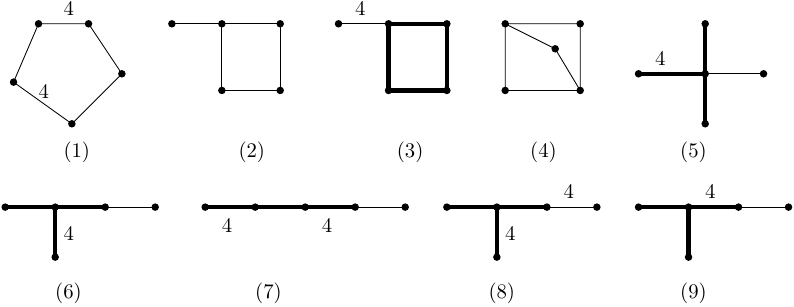}
	\caption{Coxeter diagrams for 4-dimensional quasi-Lanner groups.}
\end{figure}
\begin{proof}
The $n\le 2$ case follows from \cite{CharneyDavis}. The $n=3$ case follows from Corollary~\ref{cor:3hyperbolic}. For $n=4$, the only possible Coxeter diagrams are shown in Figure~\ref{fig:quasilanner}. The $K(\pi,1)$-conjecture for diagrams (1), (2) and (4) are proved in \cite{huang2023labeled}. The remaining diagrams are consequences of Proposition~\ref{prop:reduction} (the choice of $\Lambda'$ is the thickened subgraphs in Figure~\ref{fig:quasilanner}).
\end{proof}

\appendix

\section{Injective metric on type $B_n$ Falk complex}

In this section, $\Lambda$ will be the Coxeter diagram of type $B_n$.
Let consecutive nodes in $\Lambda$ be $\{s_1,\ldots, s_n\}$ with $m_{s_{n-1},s_n}=4$.
Let $\ca$ be the reflection arrangement of type $B_n$ in $\mathbb R^n$. Then hyperplanes of $\ca$ are of form $\{x_i=0\}_{i=1}^n$ and $\{x_i=x_j\}_{1\le i\neq j\le n}$.
Let $H$ be the hyperplane $x_1=0$. Then a simple calculation implies that deconing of $\ca$ with respect to $H$ gives the affine arrangement $\cb$ in $\mathbb R^{n-1}$ made of $y_i=-1,0,1$ for $1\le i\le n-1$, and $y_i=\pm y_j$ for $1\le i\neq j\le n-1$.

Let $S_\ca$ and $D_\cb$ be defined in Section~\ref{subsec:deligne complex}. Then $S_\ca$ is isomorphic to the Coxeter complex of the associated Coxeter group, hence each vertex of $S_\ca$ has a type $\hat s_i$ for some $i$. The complex $D_\cb$ is a subdivision of $[-1,1]^{n-1}$ as follows: first we subdivide into $2^{n-1}$ unit cubes, then we subdivide each of these unit cubes into $(n-2)!$ copies of $n$-dimensional orthoschemes, such that each of the orthoschemes contain $(0,\ldots,0)$. 
Note that $D_\cb$ can be realized as the maximal subcomplex of $S_\ca$ which is contained in the interior of a hemisphere bounded by $H\cap S_\ca$. Thus it makes sense to talk about types of vertices of $D_\cb$, using this embedding $D_\cb\to S_\ca$. It follows that the vertex $(0,\ldots,0)$ is of type $\hat s_1$. In general a vertex of $D_\cb$ is of type $\hat s_i$ if and only if the coordinate of this vertex has $i-1$ nonzero entries.

Let $\bD_\cb$ and $\bSD_\ca$ be the Falk complex and spherical Deligne complex defined in Section~\ref{subsec:deligne complex}. Then types of vertices in $D_\cb$ and $S_\ca$ pull back to type of vertices in $\bD_\cb$ and $\bSD_\ca$. As in Section~\ref{subsec:injective}, we can view $\bD_\cb$ as a subcomplex of $\bSD_\ca$ which is a connected component of the inverse image of $D_\cb$ (viewed as a subcomplex of $S_\ca$) under the map $\bSD_\ca\to S_\ca$.

%A vertex of $\bD_\cb$ is of type $\hat s_i$ if it maps to a vertex of type $\hat s_i$ under $\bD_\cb\to D_\cb$. Similarly, we define types of vertices in $\bSD_\ca$. Note that $\bD_\cb$ and $\bSD_\ca$ are simplicial complexes of type $S=\{s_1,\ldots,s_n\}$. Moreover, $\bSD_\ca$ is isomorphic to the Artin complex of type $A_n$ which preserves the types of vertices.

\begin{prop}
	\label{prop:Bn Helly}
	The vertex set of $\bD_\cb$, endowed with the relation $<$ induced from $s_1<s_2<\cdots<s_n$ as in Definition~\ref{def:order}, is a poset satisfying all the assumptions of Theorem~\ref{thm:contractibleII}.
\end{prop}

\begin{proof}
Similar to the proof of Proposition~\ref{prop:An Helly}, it suffices to verify Theorem~\ref{thm:contractibleII} (3) for type $\hat s_1$ vertices and Theorem~\ref{thm:contractibleII} (4) for type $\hat s_n$ vertices. If $x$ is of type $\hat s_1$, then $x$ maps to $\bar x=(0,0,\ldots,0)$ under $\bD_\cb\to D_\cb$. Note that the local arrangement of $\cb$ at $\bar x$ (as defined in Section~\ref{subsec:deligne complex}) is an arrangement of type $B_{n-1}$. Then we are done by Theorem~\ref{thm:triple}. If $x$ is of type $\hat s_n$, then up to symmetry, we can assume $x$ maps to $\bar x=(-1,-1,\ldots,-1)$ under $\bD_\cb\to D_\cb$. The local arrangement of $\cb$ at $\bar x$ is made of hyperplanes $\{y_i=-1\}_{i=1}^{n-1}$ and $\{y_i=y_j\}_{1\le i\neq j\le n-1}$, which is of type $A_{n-1}$ (up to a linear transformation). Then we can finish in the same way as the proof of Proposition~\ref{prop:An Helly}, using Lemma~\ref{lem:link An}, Lemma~\ref{lem:link deligne}, and Theorem~\ref{thm:triple}. 
\end{proof}

\bibliographystyle{alpha}
\bibliography{mybib}

\end{document}